\documentclass[11pt]{article}%
\usepackage[utf8]{inputenc}
\usepackage{amssymb,amsmath,amsfonts,amsthm,array,bm,bbm,color}%
\usepackage{stmaryrd}
\usepackage{graphicx}
\providecommand{\U}[1]{\protect\rule{.1in}{.1in}}

\numberwithin{equation}{section}

\newtheorem{theorem}{Theorem}[section]
\newtheorem{lemma}[theorem]{Lemma}
\newtheorem{corollary}[theorem]{Corollary}
\newtheorem{proposition}[theorem]{Proposition}
\newtheorem{remark}[theorem]{Remark}

\newtheorem{definition}[theorem]{Definition}
\newtheorem{hypothesis}[theorem]{Hypothesis}

\newtheorem{assumption}[theorem]{Assumption}

\def\<{\langle}
\def\>{\rangle}
\def\d{{\rm d}}

\def\E{\mathbb{E}}
\def\N{\mathbb{N}}
\def\P{\mathbb{P}}
\def\R{\mathbb{R}}
\def\T{\mathbb{T}}
\def\Z{\mathbb{Z}}
\def\F{\mathcal{F}}
\def\cH{\mathcal{H}}
\def\cE{\mathcal{E}}
\def\cG{\mathcal{G}}
\def\cL{\mathcal{L}}

\def\eps{\varepsilon}

\newcommand{\red}[1]{{\color{red} #1}}

\begin{document}

\title{LDP and CLT for SPDEs with Transport Noise}

\author{Lucio Galeati\thanks{Email: lucio.galeati@epfl.ch. Research Group Math-AMCV, EPFL, 1015 Lausanne, Switzerland.}
\quad Dejun Luo\thanks{Email: luodj@amss.ac.cn. Key Laboratory of RCSDS,
Academy of Mathematics and Systems Science, Chinese Academy of Sciences,
Beijing 100190, China, and School of Mathematical Sciences, University of
Chinese Academy of Sciences, Beijing 100049, China. } }

\maketitle

\vspace{-20pt}

\begin{abstract}
In this work we consider solutions to stochastic partial differential equations with transport noise, which are known to converge, in a suitable scaling limit, to solution of the corresponding deterministic PDE with an additional viscosity term. Large deviations and Gaussian fluctuations underlying such scaling limit are investigated in two cases of interest: stochastic linear transport equations in dimension $D\geq 2$ and $2$D Euler equations in vorticity form. In both cases, a central limit theorem with strong convergence and explicit rate is established. The proofs rely on nontrivial tools, like the  solvability of transport equations with supercritical coefficients and $\Gamma$-convergence arguments.
\end{abstract}



\tableofcontents

\section{Introduction}

In recent years, there have been numerous results concerning the convergence of stochastic partial differential equations with transport noise, i.e. of the form
\begin{equation}\label{eq:intro-abstract-SPDE}
\d u^n + \circ \d \tilde W^n\cdot \nabla u^n = F(u^n)\, \d t,
\end{equation}
to their deterministic counterparts with enhanced viscosity
\begin{equation}\label{eq:intro-limit-pde}
\partial_t u =  F(u) + \nu \Delta u
\end{equation}
whenever the noises $\{\tilde W^n\}$ undergo a suitable scaling limit.
Here $F$ denotes a generic functional, which might be nonlinear and nonlocal, while $\circ \d $ means that the stochastic differential must be understood in the Stratonovich sense.
The noises $\{\tilde W^n\}$ are Brownian in time and coloured, divergence-free in space and the coefficient $\nu>0$ comes from the noise intensity in the scaling limit; more details on the precise structure of $\{\tilde W^n\}$ will be given shortly.

The scaling limit of \eqref{eq:intro-abstract-SPDE} to \eqref{eq:intro-limit-pde} was first observed in \cite{Galeati} for linear transport equation and has then been extended to various nonlinear PDEs, including stochastic 2D Euler equation in vorticity form \cite{FGL21a}, stochastic mSQG equation \cite{LuoSaal} and stochastic 2D inviscid Boussinesq system \cite{Luo21}; see also \cite{FGL21b} for a large class of nonlinearities which include as particular cases the $3$D Keller-Segel system and the $2$D Kuramoto-Sivashinsky equation.
All of the aforementioned results are set on the torus $\T^d=\R^d/\Z^d$; some results on stochastic heat equations in bounded domain or infinite channel can be found in \cite{FGL21d, FlaLuongo22}.
Finally let us mention the recent work \cite{FGL21c}, where we improved some of these results by providing quantitative convergence rates.


The scaling limit of \eqref{eq:intro-abstract-SPDE} to \eqref{eq:intro-limit-pde} has several important features.
Firstly, the starting equation \eqref{eq:intro-abstract-SPDE} can be of hyperbolic type (e.g. in the case of stochastic transport equations) but the limit \eqref{eq:intro-limit-pde} is always strictly parabolic; as such, the scaling limit can be rigorously proved even in situations where the well-posedness of \eqref{eq:intro-abstract-SPDE} is not known to hold, as long as we are able to solve \eqref{eq:intro-limit-pde} (this is for instance the case for stochastic $2$D Euler converging to deterministic Navier--Stokes).
Secondly, even when the starting equation \eqref{eq:intro-abstract-SPDE} is truly parabolic, the extra strong diffusion operator $\nu\Delta$ appearing in the limit equation can help displaying features of the transport noise like mixing and dissipation enhancement, see \cite{FGL21c,FGL21d}.
We have taken advantage of this phenomenon and shown in \cite{FlaLuo21, FGL21b} that transport noise is capable, with large probability, of suppressing blow-up of solutions or even yielding global existence.

Given the above results, it is of great importance to obtain a better understanding of these scaling limits; the purpose of the present work is to study the large deviations and Gaussian fluctuations underlying them.

To explain what we mean, let us first briefly explain why convergence of \eqref{eq:intro-abstract-SPDE} to \eqref{eq:intro-limit-pde} holds.
The noises $\{\tilde W^n\}$ must be taken of the form $\tilde{W}^n = \sqrt{\nu \eps_n}\, W^n$, where $\{W^n\}$ are now $Q^n$-Wiener processes on $L^2$ such that the family $\{Q^n \}_n$ are uniformly bounded operators, $\| Q^n\|_{L^2\to L^2}\leq C$; at the same time, they must satisfy ${\rm Tr}(Q^n)= c/\eps_n$, where $c$ is a suitable proportionality constant and $\eps_n\to 0$ as $n\to\infty$. We refer to \cite[Example 1.3]{FGL21d} for explicit choices of such noise on the torus $\T^2$, see Section \ref{sec:intro-noise} below for a special case.
For this choice of $\{\tilde W^n\}$, it can be shown that the equivalent It\^o form of \eqref{eq:intro-abstract-SPDE} is
\begin{equation}\label{eq:intro-abstract-SPDE-ito}
\d u^n + \sqrt{\nu \eps_n}\, \d W^n\cdot\nabla u^n = [ F(u^n) + \nu \Delta u^n]\,\d t.
\end{equation}
As the noises $W^n$ have uniformly bounded covariances $Q^n$, convergence of \eqref{eq:intro-abstract-SPDE-ito} to \eqref{eq:intro-limit-pde} is not surprising, as long as uniform estimates on $\{u^n\}$ are available (here the divergence free property of $\tilde{W}^n$ becomes crucial).
Such a convergence can now be interpreted as a small noise limit, which suggests an underlying Large Deviation Principle (LDP); in this sense, we can also interpret the basic convergence of $u^n$ to a deterministic limit $u$ as some kind of Law of Large Numbers (LLN).

In order to prove a LDP, we need to assume that $Q^n$ converges to some bounded linear operator $Q$ on $L^2$, so that the corresponding noises $W^n$ converge in law to some $W$, which is a $Q$-Wiener process (cf. Section \ref{sec:intro-noise} below).
This covers the majority of the cases considered in \cite{Galeati, FGL21b}, but it excludes the noise from \cite{FlaLuo21} (similarly \cite[Example 1.3]{FGL21d}), in which case the operators $Q^n$ converge weakly to $0$; see Remark \ref{rem:other-noise} below for a deeper discussion on this point.
By the weak convergence approach to large deviations (see \cite{BD00,BDM08}), we expect a LDP associated to $u^n$ with speed $\eps_n$ and rate function determined by solving the skeleton equation
\begin{equation}\label{eq:intro-skeleton-abstract}
\partial_t v + g\cdot\nabla v = F(v)+\nu\Delta v
\end{equation}
for any $g$ satisfying the condition $\int_0^T \| Q^{-1/2} g_s\|_{L^2}^2 \,\d s<\infty$.
Let us point out that solvability of \eqref{eq:intro-skeleton-abstract} is highly non trivial, even in the case $F\equiv 0$, where \eqref{eq:intro-skeleton-abstract} corresponds to an advection-diffusion equation with rough drift; indeed the only available information on $g$ roughly speaking amounts to it being divergence free and satisfying  $\int_{[0,T]\times \T^d} |g(t,x)|^2 \,\d t \d x<\infty$. 

Interpreting the convergence of $u^n$ to $u$ as a LLN, one might try to investigate the fluctuations around this convergence, by looking at $U^n= (u^n-u)/\sqrt{\eps_n}$.
Assuming the sequence $U^n$ admits a limit $U$, we have the formal expansions
\begin{equation}\label{eq:intro-formal-expansion}
u^n\approx u + \sqrt{\eps_n}\, U,\quad F(u^n)\approx F(u) + \sqrt{\eps_n}\, DF_u(U)
\end{equation}
where $DF_u$ denotes the differential of $F$ evaluated at the point $u$. Inserting \eqref{eq:intro-formal-expansion} in \eqref{eq:intro-abstract-SPDE-ito}, we see that $U$ is expected to solve
\begin{equation}\label{eq:intro-abstract-CLT}
\d U + \sqrt{\nu}\,\d W\cdot\nabla u = [DF_u(U) + \nu \Delta U]\,\d t.
\end{equation}
Equation \eqref{eq:intro-abstract-CLT} is linear in $U$ and driven by the forcing term $\sqrt{\nu}\, \d W\cdot\nabla u$, which is Gaussian since $u$ is deterministic.
As a consequence, the solution $U$ (if it exists) is a Gaussian field as well; the convergence of fluctuations $U^n= (u^n-u)/\sqrt{\eps_n}$ to a Gaussian limit then can be naturally intepreted as a Central Limit Theorem (CLT) result.

Although the above reasoning is very heuristical, it already provides useful information on the key quantities to study in order to obtain rigorous proofs, which are given respectively by equations \eqref{eq:intro-skeleton-abstract} and \eqref{eq:intro-abstract-CLT}.
The aim of this paper is to formalize the above discussion in two main cases of interest, namely for stochastic transport equations and $2$D Euler equations in vorticity form.
We preferred to start this introduction with the more abstract framework of equations \eqref{eq:intro-abstract-SPDE}, in order to emphasize that the guiding principle behind them is very general and can be applied to other SPDEs (see also discussions in Sections \ref{sec:conclusive-LDP} and \ref{sec:conclusive-CLT}).

Before giving rigorous statements, we need to provide more details on the exact choice of noise we will deal with.

\subsection{Structure of the noise}\label{sec:intro-noise}

In this paper we will always deal with equations defined on the $d$-dimensional torus $\T^d$ for $d\geq 2$ and on a finite time interval $[0,T]$. From now on we will set for simplicity $\nu=1$.

For technical reasons that will become clear later (cf. the end of the proof of Lemma \ref{prop-cond-2}), we cannot take the limit noise $W$ to be white in space, but rather we need to impose some additional regularity on it as encoded by a parameter $\alpha\in (0,d/2)$; this can be accomplished as follows.

Let us denote by $\cH$ the space of square integrable, divergence free vector fields with zero mean on $\T^d$; it is a closed subspace of $L^2=L^2(\T^d,\R^d)$ and so we can define the projection $\Pi$ of $L^2$ to $\cH$, usually referred as the Leray--Helmholtz projection.
We denote by $e_k(x)= e^{2\pi i k\cdot x}$ the standard elements of the Fourier basis; recall the definition of the fractional Laplacian $(-\Delta)^{s/2}$ by $(-\Delta)^{s/2} e_k = |k|^s e_k$ for any $s\in \R$.
We also set $\cH^\alpha:= H^\alpha(\T^d,\R^d) \cap \cH$, so that $\cH^\alpha = (-\Delta)^{-\alpha/2} (\cH) = (-\Delta)^{-\alpha/2}\,\Pi (L^2)$.

With these preparations, we can define the limit noise $W:=W^\alpha$ as a $\cH^\alpha$-cylindrical Wiener process; alternatively, $W^\alpha$ can be constructed as $W^\alpha= (-\Delta)^{-\alpha/2} \Pi\, Z$, where $Z$ is a cylindrical noise on $L^2$. It then becomes clear that the covariance function associated to $W^\alpha$ is
\[ Q^\alpha = (-\Delta)^{-\alpha} \Pi \]
where we used the fact that $\Pi$ and $(-\Delta)^{-\alpha/2}$ commute. It follows in particular that, for any sufficently regular, predictable $\R^d$-valued field $f$, the stochastic integral $\int_0^t \langle f_r, \d W^\alpha_r\rangle$ is well defined with It\^o isometry
\begin{equation}\label{eq:intro-basic-estim}
\E\bigg[\Big|\int_0^t\<f_r,\d W^\alpha_r\>\Big|^2\bigg]
= \E\int_0^t \| (Q^\alpha)^{1/2} f_r\|_{L^2}^2\, \d r
= \E\int_0^t \| \Pi f_r \|_{H^{-\alpha}}^2\, \d r
\leq \E\int_0^t \| f_r\|_{H^{-\alpha}}^2\, \d r.
\end{equation}

Next, we need to define the approximating noises appearing in \eqref{eq:intro-abstract-SPDE-ito} and the coefficients $\eps_n$;
to this end, denote by $\Pi_n$ the projection of $L^2$ on the subspace spanned by Fourier modes smaller than $n$, namely $\Pi_n e_k = e_k\, \mathbbm{1}_{|k|\leq n}$.
We set
\[ W^{n,\alpha}= \Pi_n W^\alpha\]
so that they have associated covariance $Q^{n,\alpha}:= \Pi_n (-\Delta)^{-\alpha} \Pi $. A similar formula \eqref{eq:intro-basic-estim} holds for $\int_0^t \<f_r, \d W^{n,\alpha}_r \>$ up to replacing $Q^{\alpha}$ by $Q^{n,\alpha}$.
Given this choice, we must take $\eps_n$ proportional to the inverse of the trace of $Q^{n,\alpha}$, namely
\begin{equation}
\eps_n
= c_d \Bigg( \sum_{k\in \Z^d, 0<|k|\leq n} |k|^{-2\alpha} \Bigg)^{-1}
\sim \bigg( \int_1^n r^{d-2\alpha-1}\, \d r\bigg)^{-1}
\sim n^{2\alpha-d}
\end{equation}
where $c_d=d/(d-1)$; since $\alpha\in (0,d/2)$, $\eps_n\to 0$ as $n\to\infty$ as desired.

In order to perform computations, it is useful to have a more explicit expression for $W^\alpha$, following \cite[Section 2.2]{Galeati}.
Set $\Z^d_0:=\Z^d\setminus\{0\}$ and let $\{\sigma_{k,i}: k\in \Z^d_0, i=1,\ldots, d-1 \}$ be a CONB of $\cH$, defined as $\sigma_{k,i}(x) = a_{k,i} e_{k}(x)$
where $\{a_{k,i}\}_{k,i}$ is a subset of the unit sphere $\mathbb{S}^{d-1}$ such that: (i) $a_{k,i}=a_{-k,i}$ for all $k\in \mathbb{Z}^{d}_{0},\, i=1,\ldots,d-1$; (ii) for fixed $k$, $\{a_{k,i}\}_{i=1}^{d-1}$ is an ONB of $k^{\perp}=\{y\in\mathbb{R}^{d}:y\cdot k=0 \}$.
It holds that $\nabla\cdot\sigma_{k,i} = (a_{k,i}\cdot k) e_k \equiv 0$ for all $k\in \Z^d_0$ and $1\leq i\leq d-1$.
If $d=2$, then we can explicitly define
\begin{equation*}
  a_{k,1}= a_k = \begin{cases}
  \frac{k^\perp}{|k|}, & k\in \Z^2_+, \\
  -\frac{k^\perp}{|k|}, & k\in \Z^2_-,
  \end{cases}
\end{equation*}
where $k^\perp$ now denotes the vector $(-k_2, k_1)$, $\Z^2_0 = \Z^2_+ \cup \Z^2_-$ is a partition of $\Z^2_0$ satisfying $\Z^2_+ = -\Z^2_-$.
Then $W^\alpha$ can be realized as
\begin{equation}\label{eq:noise}
W(t,x)= \sum_{k\in \Z^d_0}\sum_{i=1}^{d-1} |k|^{-\alpha}  \,\sigma_{k,i}(x)\, B^{k,i}_t ,
\end{equation}
where $\{B^{k,i}_\cdot: k\in \Z^d_0, i=1,\ldots, d-1 \}$ are $\mathbb{C}$-valued standard Brownian motions such that $B^{k,i}$ and $B^{l,j}$ are independent whenever $k\neq {-l}$ or $i\neq j$ and such that $\overline{B^{k,i}}=B^{-k,i}$ for all $k,\,i$ (so that $\overline{W^\alpha}=W^\alpha$, i.e. $W^\alpha$ is real valued).
In the sequel we will write $\sum_{k\in \Z^d_0} \sum_{i=1}^{d-1}$ simply as $\sum_{k,i}$. Clearly, a similar expression as \eqref{eq:noise} holds for $W^{n,\alpha}$ up to restricting the sum on $|k|\leq n$.

\subsection{Main results}

\subsubsection{Stochastic transport equations}
We will first consider linear stochastic transport equations on $\T^d$, $d\geq 2$:
\begin{equation}\label{intro-SLTE}
\d f^n + b\cdot\nabla f^n\, \d t + \sqrt{\eps_n} \circ \d W^{n,\alpha} \cdot \nabla f^n =0, \quad f^n\vert_{t=0}=f_0,
\end{equation}
with initial data $f_0\in L^2$.
From here on, we will adopt shortcut notations for the function spaces we use, for which we refer the reader to the notations introduced in Section \ref{subs-notations}.

We assume $b$ to be a vector field in the Krylov--R\"ockner class (see \cite{KR05}); namely, $b\in L^q_t L^p_x$ for some parameters $p,\,q\in [2,\infty]$ satisfying
\begin{equation}\label{eq:intro-KR-condition}
\gamma:= \frac{2}{q}+\frac{d}{p}<1.
\end{equation}
For simplicity of exposition, we will additionally impose $b$ to be divergence-free; the condition can be slightly weakened, depending on the results in consideration, see respectively Assumptions \ref{ass:drift-transport} and \ref{ass:drift-CLT}.
Under the above conditions, strong existence and pathwise uniqueness holds for \eqref{intro-SLTE}; moreover we have the solution formula $f^n_t(x)=f_0\big((\varphi^n_t)^{-1} (x)\big)$, where $\{\varphi^n_t\}$ is the stochastic flow associated to the underlying SDE and $\{(\varphi^n_t)^{-1} \}$ denotes its inverse.

The skeleton equation \eqref{eq:intro-skeleton-abstract} in this setting reads
\begin{equation}\label{eq:intro-skeleton-STLE}
\partial_t f + (b+g)\cdot \nabla f  = \Delta f, \quad f\vert_{t=0}=f_0
\end{equation}
where $g$ is an element of the Cameron-Martin space associated to $W^\alpha$, namely $g\in L^2_t \cH^\alpha$.
Observe that since $\alpha\in (0,d/2)$, $\cH^\alpha$ does not embed in $L^\infty$, implying that $b+g$ does not belong to $L^q_t L^p_x$ with $p,\, q$ satisfying \eqref{eq:intro-KR-condition}; this is sometimes referred as a \textit{supercritical drift}.

Treating equations with supercritical drift, especially in connection with SDEs and stochastic transport equations, has been the subject of several recent papers, see \cite{FeNeOl, NevOli,zhazha2021}.
An earlier important contribution, addressing the conditional uniqueness of Fokker--Plack equations with drift $b\in L^2_t L^2_x$, is due to Porretta \cite{Porretta}, who also established renormalizability of such solutions (assuming they exist).
Here we show  that equation \eqref{eq:intro-skeleton-STLE} is wellposed for any $f_0\in L^2$ and any divergence free $g\in L^2_{t,x}$, see Propositions \ref{prop:stability-strong-skeleton} and \ref{prop:stability-skeleton-L2}.

While working on the current manuscript, we became aware of the recent work \cite{BCC}, which contains a comprehensive study of such advection-diffusion equations, see Section 3 therein. Their results are comparable to ours, although the derivations are quite different; the authors in \cite{BCC} mostly employ commutator estimates in the style of Di Perna--Lions theory, while we make use of classical tools from parabolic theory like maximal regularity, Lions--Magenes lemma and duality techniques.
Propositions \ref{prop:stability-strong-skeleton} and \ref{prop:stability-skeleton-L2} below have the advantage of providing several stability results, both in strong and weak topologies, which are needed for our purposes; moreover they allow for non-divergence free vector fields.

Having established well-posedness and stability for the skeleton equation \eqref{eq:intro-skeleton-STLE}, we are able to rigorously prove a LDP for the sequence $\{f^n \}_n$.
\begin{theorem}\label{thm:intro-thm1}
Let $f_0\in L^2$ and $b$ be a divergence free drift in the Krylov--R\"ockner class; then, for any $\delta>0$, the sequence $\{f^n \}_n$ of solutions to \eqref{intro-SLTE} satisfy a LDP in $C^0_t H^{-\delta}$, with rate function
\begin{equation*}
I_{f_0}(f):=\inf\bigg\{\frac{1}{2}\int_0^T \| g_s\|_{\cH^\alpha}^2\, \d s : g\in L^2_t \cH^\alpha \text{ such that \eqref{eq:intro-skeleton-STLE} holds} \bigg\}.
\end{equation*}
\end{theorem}

In fact, we will prove a slightly stronger statement, namely a Laplace Principle in $C^0_t H^{-\delta}$, uniformly over compact sets of $f_0\in L^2$, see Theorem \ref{thm-L2-data}.
The restriction $\delta>0$ is necessary, as it is known that the sequence $f^n$ does not converge, in the topology of $L^2_{t,x}$, to the solution $\bar{f}$ of $\partial_t \bar{f} +b\cdot\nabla \bar{f}=\Delta\bar{f}$, thus a LDP therein cannot hold.

Theorem \ref{thm:intro-thm1} is not just a direct application of the aforementioned weak convergence approach to large deviations.
Indeed, the stability results needed for it to work only provide an LDP for $f_0\in L^\infty$, see Proposition \ref{thm-bounded-data};
once this is established, we can develop additional arguments, based on the linearity of the equation and the use of $\Gamma$-convergence, to transfer the LDP to general initial data $f_0\in L^2$, see the proof of Theorem \ref{thm-L2-data}.

Regarding the CLT, equation \eqref{eq:intro-abstract-CLT} in this case reads
\begin{equation}\label{eq:intro-CLT-SLTE}
\d X + b\cdot\nabla X\,\d t = \Delta X \,\d t - \d W^\alpha \cdot\nabla \bar{f}.
\end{equation}
We cannot expect to make sense of \eqref{eq:intro-CLT-SLTE} in the classical way. Indeed, $W^\alpha$ enjoys at most spatial regularity $C^{\alpha-d/2-}$; once \eqref{eq:intro-CLT-SLTE} is written in the mild form, even assuming $\bar{f}$ smooth, the regularity of $\int_0^t P_{t-r}(\d W^\alpha_r\cdot\nabla \bar{f}_r)$ (and thus of $X_t$) can be at most $C^{\alpha+1-d/2-}$. In turn this implies $\nabla X_t\in C^{\alpha-d/2-}$, which makes the product $b\cdot\nabla X$ ill-defined.

The solution to the problem comes from a duality formulation of \eqref{eq:intro-CLT-SLTE} (see Definition \ref{defn:weak-sol}), close in spirit to a martingale problem formulation, which makes the equation well posed (Proposition \ref{prop:wellposedness-limit-eq}). Furthermore, setting $X^n:= (f^n-\bar{f})/\sqrt{\eps_n}$, we can show strong convergence of $X^n$ to $X$ with an explicit rate:

\begin{theorem}\label{thm:intro-2}
Let $f_0\in L^2$, $X^n$ defined as above and $X$ solution to \eqref{eq:intro-CLT-SLTE}. Then there exists an exponent $\beta>0$ such that, for any $\eps>0$, it holds
\[
\sup_{t\in [0,T]} \E\big[ \|X^n_t-X_t\|_{H^{-d/2-1}}^2\big] \lesssim n^{-2\beta+\eps} \| f_0\|_{L^2}^2 \quad\forall\, n\in \N.
\]
In particular, the exponent can be chosen as
\[
\beta = [\alpha\wedge (1-\gamma)]\,\bigg(1-\frac{2\alpha}{d}\bigg).
\]
\end{theorem}

The statement is a particular subcase of Theorem \ref{thm:CLT-SLTE-general} from Section \ref{subs-CLT-SLTE}; we do not know whether the exponent $\beta$ is sharp.

\subsubsection{Stochastic 2D Euler equations} \label{subsec-intro-SEE}

Next, we consider the vorticity form of stochastic 2D Euler equation driven by transport noise:
\begin{equation}\label{eq:SEE-intro}
\d \xi^n + (K\ast\xi^n)\cdot\nabla \xi^n\, \d t + \sqrt{\eps_n} \circ\d W^{n,\alpha} \cdot \nabla \xi^n=0, \quad \xi^n_0 = \xi_0.
\end{equation}
Here $\xi^n$ stands for the vorticity and $K$ is the Biot-Savart kernel on $\T^2$, thus $K\ast\xi^n$ is the fluid velocity.
The initial data $\xi_0$ is assumed to satisfy $\int_{\T^2} \xi_0(x)\, \d x=0$, a property which is preserved at all subsequent times.

By \cite[Theorem 2.10]{BFM}, equation \eqref{eq:SEE-intro} admits a strong, pathwise unique solution for any $\xi_0\in L^\infty$; moreover it is of the form $\xi^n_t(x)=\xi_0\big((\varphi^n_t)^{-1}(x) \big)$, where $\{\varphi^n_t \}$ is the unique associated measure preserving stochastic flow, see Theorem 2.15 therein.
Instead, for any $\xi_0\in L^2$, we only have the existence of probabilistically weak solutions $\xi^n$, satisfying the $\P$-a.s. pathwise bound $\sup_{t\in [0,T]} \|\xi^n_t \|_{L^2} \leq \| \xi_0\|_{L^2}$ (see \cite[Theorem 2.2]{FGL21a}); their uniqueness is an open problem.

We always interpret the equation analytically in a weak sense: an $\mathcal{F}_t$ adapted process $\xi^n$ solves \eqref{eq:SEE-intro} if for any $\varphi\in C^\infty$, $\P$-a.s. it holds
\[
\<\xi^n_t,\varphi\> = \< \xi_0,\varphi\> + \int_0^t \big\<\xi^n_r, (K\ast\xi^n_r) \cdot\nabla \varphi + \Delta\varphi \big\> \,\d r + \sqrt{\eps_n}\int_0^t \langle \xi^n_r\nabla \varphi, \d W^{n,\alpha}_r\rangle \quad \forall\, t\in [0,T],
\]
where we have passed to the corresponding It\^o form, giving rise to the term $\Delta \varphi$. By the properties of Biot-Savart kernel, estimate \eqref{eq:intro-basic-estim} and the pathwise bounds on $\xi^n$, it is easy to check that all the integrals appearing above are finite.

Given $\xi_0\in L^2$, it was first shown in \cite{FGL21a} that as $n\to\infty$, $\xi^n$ converge in law to $\bar{\xi}$, which is the unique solution to the deterministic 2D Navier-Stokes equation in vorticity form:
\begin{equation}\label{eq:2D-NSE-intro}
\partial_t \bar\xi + (K\ast\bar\xi)\cdot\nabla \bar\xi = \Delta \bar\xi,\quad \bar\xi|_{t=0} =\xi_0.
\end{equation}

The skeleton equation \eqref{eq:intro-skeleton-abstract} in this case reads
\begin{equation}\label{eq:intro-skeleton-euler}
\partial_t \xi + (K\ast\xi)\cdot\nabla \xi + g\cdot\nabla \xi = \Delta\xi,\quad \xi\vert_{t=0}=\xi_0.
\end{equation}
Again we are able to show well-posedness of \eqref{eq:intro-skeleton-euler}, at least for $\xi_0\in L^\infty$, see Lemma \ref{lem:uniq-skeleton-navier-stokes}. In turn, this allows to prove a LDP for the family $\{\xi^n \}_n$:
\begin{theorem}\label{thm:intro-3}
Let $\xi_0\in L^\infty$; then, for any $\delta>0$, the sequence $\{\xi^n \}_n$ of solutions to \eqref{eq:SEE-intro} satisfy a LDP in $C^0_t H^{-\delta}$, with rate function
\begin{equation*}
I_{\xi_0}(\xi):=\inf\bigg\{\frac{1}{2}\int_0^T \| g_s\|_{\cH^\alpha}^2\, \d s : g\in L^2_t \cH^\alpha \text{ such that \eqref{eq:intro-skeleton-euler} holds} \bigg\}.
\end{equation*}
\end{theorem}
In fact, we will obtain a slightly stronger statement, concerning the Laplace principle on a suitable Polish space $\cE^{T,R}$, uniformly over compact sets of $\xi_0\in \cE^R_0$, see Theorem \ref{thm-LDP-S2EE} for more details.
The restriction to $\xi_0\in L^\infty$ is needed in order to apply the results from \cite{BFM} and \cite{BDM08}; it is an open problem to understand whether the LDP can be extended to the case $\xi_0\in L^2$, as accomplished for linear transport equations.

Next, we pass to study the behaviour of the fluctuation $\Xi^n = (\xi^n -\bar\xi)/\sqrt{\eps_n}$. In this case we allow $\xi_0\in L^2$, so that we are dealing with probabilistically weak solutions.
This means that, in order to be rigorous, the stochastic basis $(\Omega, \mathcal{F}_t, \P; W^\alpha)$ should depend on $n$, but let us drop such dependence for notational simplicity. Uniqueness in law for $\xi^n$ is not known, but the results presented below will hold for \textit{any} sequence $\{\xi^n \}_n$ of weak solutions to \eqref{eq:SEE-intro} satisfying the $\P$-a.s. pathwise bound $\sup_{t\in [0,T]} \| \xi^n_t\|_{L^2} \leq \| \xi_0\|_{L^2}$.

We expect $\Xi^n$ converges to $\Xi$, which solves the equivalent of equation \eqref{eq:intro-abstract-CLT}, namely
\begin{equation}\label{eq:intro-CLT-euler}
\d \Xi + \big[ (K\ast\Xi)\cdot\nabla \xi + (K\ast\xi)\cdot\nabla \Xi\big]\, \d t + \d W^{\alpha}\cdot \nabla \xi= \Delta \Xi\,\d t.
\end{equation}
Contrary to the approach we adopted to treat equation \eqref{eq:intro-CLT-SLTE}, here we are able to give an analytical meaning to \eqref{eq:intro-CLT-euler} once it is written in the mild form, see Proposition \ref{prop:CLT-limit-analytic}.
Also in this case, we can establish explicit rates of convergence:
\begin{theorem}\label{thm:intro-4}
Let $\xi_0\in L^2$, $\{\xi^n \}_n$ be a sequence of weak solutions to \eqref{eq:SEE-intro}, $\Xi^n$ defined as above and $\Xi$ solving \eqref{eq:intro-CLT-euler}. Then there exists a $\beta>0$ such that, for any $\eps>0$, it holds
\[
\sup_{t\in [0,T]} \E\big[ \|\Xi^n_t-\Xi_t\|_{H^{-1}}^2\big] \lesssim n^{-\beta+\eps} \quad\forall\, n\in \N.
\]
In particular, the exponent can be chosen as $\beta=\alpha(1-\alpha)$.
\end{theorem}
The result follows from the more general Theorem \ref{thm:CLT-SEE-1} and Corollary \ref{cor:CLT-euler-1}. Let us finally remark that Theorem \ref{thm:intro-4} implies that
\[
\sup_{t\in [0,T]} \E\big[ \|\xi^n_t-\bar{\xi}_t\|_{H^{-1}}^2\big] \sim \eps_n ,
\]
and thus it improves \cite[Theorem 1.1-(i)]{FGL21c} and shows that the latter is almost sharp, at least in the special case of the noises $\{W^{n,\alpha}\}$ considered in this work.

\begin{remark}\label{rem:other-noise}
As mentioned above, we excluded from our analysis the choice of the noise adopted in \cite{FlaLuo21}; it corresponds in the abstract setting of \eqref{eq:intro-abstract-SPDE} to $\tilde{W}^n= \sqrt{\nu \eps_n}\, W^n$, for $W^n$ with associated covariance $Q^n = \tilde{\Pi}_n (-\Delta)^{-\alpha} \Pi$, where now $\tilde{\Pi}_n$ denotes the operator that acts as a Fourier multiplier by $\tilde\Pi_n e_k = e_k \mathbbm{1}_{n\leq |k|\leq 2n}$.
At a technical level this does not create any particular issue, but since now the operators $Q^n$ converge weakly to $0$, the corresponding LDP and CLT become degenerate.
Specifically, the candidate skeleton equation \eqref{eq:intro-skeleton-abstract} now must be replaced by equation \eqref{eq:intro-limit-pde}, regardless the choice of $g$, so that the resulting rate function $I_{f_0}$ assigns value $0$ to the unique solution $f$ to \eqref{eq:intro-limit-pde} with initial value $f_0$ and $+\infty$ to any other function.
Similarly in the CLT, instead of the SPDE \eqref{eq:intro-abstract-CLT}, one obtains the corresponding PDE without stochastic forcing, whose solution is necessarily $0$ due to linearity and the prescribed initial condition $0$.
\end{remark}

\subsection{Relations with the existing literature} \label{subsec-references}

Let us discuss here some links between our work and previous contributions, organized by different topics.

\textit{Stochastic fluid dynamics and turbulence.}
An idealised description of the effects of small, possibly turbulent, fluid scales by means of a white-in-time, coloured-in-space noise (in the style of $W^\alpha$) was first proposed by Kraichnan in the context of passive scalar turbulence (namely eq. \eqref{intro-SLTE}), see the works \cite{Kra1, Kra2}, as well as the reviews \cite{CGHPV, FaGaVe}; the value $\alpha=d/2+2/3$ corresponds to Kolmogorov's K41 theory of turbulence, see \cite{Frisch}.
The noise $W^\alpha$ in consideration here belongs to the class of measure preserving, isotropic Brownian flows as defined in \cite{Kunita}, which have been studied in detail in \cite{LeJan,BaxHar} and later revisited \cite{LeJRai, LeJRai2}.
Early mathematical attempts to incorporate transport noise in nonlinear fluid dynamics equations go back to Inoue and Funaki \cite{InoFun}; their idea, based on a variational principle, has been later substantially revisited by Holm \cite{Holm15} and expanded to several classes of equations, see \cite{DrHo, CHLN} and the references therein.
A different, more Lagrangian-flavoured approach was proposed in \cite{BCF91, BCF92}, where Brz\'ezniak et al. derived stochastic Navier--Stokes equations with transport noise by computing the stochastic material derivative along the trajectories of fluid particles. More complete studies of such models were then presented by Mikulevicius and Rozovskii \cite{MikRoz04, MikRoz}, establishing the well-posedness of the stochastic Navier--Stokes equations.
Yet another model for turbulent fluid dynamical equations, incorporating transport noise, has been proposed by M\'emin in \cite{Memin}.
In all of the above cases, the adoption of Stratonovich noise is physically justified  by the Wong--Zakai principle \cite{WonZak}, which is known to hold also for SPDEs, see e.g. \cite{BreFla}.
An heuristic argument for the introduction of transport noise, based on a separation of scales, can be found in \cite{FlaLuo21}, which has been made rigorous by Flandoli and Pappalettera \cite{FlaPap, FlaPap21} using methods of stochastic model reduction; see \cite{DebPap22} for more general results. Recently, advection-diffusion equations with Stratonovich transport noise have been derived by homogenization techniques in \cite{Fehrman}.

\textit{Regularisation by noise.} It was first shown in \cite{FGP} that Stratonovich transport noise can improve the well-posedness theory of transport equations; this phenomenon, nowadays referred to regularisation by noise, has been treated extensively also in the context of nonlinear SPDE; see \cite{Flandoli} for an (incomplete) survey.
Already in \cite[Section 6.2]{FGP} the authors realised that too simple noise cannot regularise nonlinear PDEs like Euler equations; however more complicated, space-dependent noise can be more promising, like in \cite{FGP2} where it is shown to (slightly) regularise the dynamics of point vortices.
In the aforementioned works \cite{FlaLuo21, FGL21b}, we were able to exploit the scaling limit of \eqref{eq:intro-abstract-SPDE} to \eqref{eq:intro-limit-pde} to show that for a large class of models transport noise can delay or even suppress blow-up with high probability. Similar techniques have been recently applied to averaged Navier--Stokes systems in \cite{Lange} and reaction diffusion equations in \cite{Agresti}; let us also mention that \cite{Agresti} discusses a nice and natural interpretation of the scaling of \eqref{eq:intro-abstract-SPDE} to \eqref{eq:intro-limit-pde} in the context of homogenization.

\textit{Large deviations for SPDEs.}
The weak convergence approach to large deviations  developed in \cite{BD00, BDM08} is nowadays a very popular method for establishing LDPs for the laws of solutions to SPDEs.
Among the first applications in stochastic fluid dynamics there is the work of Sritharan and Sundar \cite{SS06} for $2$D Navier--Stokes equations driven by multiplicative noise.
Subsequent applications include: a Boussinesq model for the B\'enard convection under random influences as in \cite{DuanMillet09}; a class of abstract nonlinear stochastic models, covering $2$D Navier--Stokes equations, $2$D MHD models and shell models of turbulence in \cite{ChuMillet10}; stochastic models of incompressible second grade fluids in \cite{ZhaiZhang17}; $3$D stochastic primitive equations in \cite{DongZhaiZhang17};
stochastic porous media equations in \cite{RZhang20}.
Bessaih and Millet \cite{BessMillet12} proved a LDP for the solution of $2$D stochastic Navier--Stokes equations with vanishing viscosity and noise intensity proportional to the square root of the viscosity;
Cerrai and Debussche \cite{CerDeb19b} instead considered $2$D Navier--Stokes equation with Gaussian random forcing, establishing LDPs in the regime where the strength of the noise and its correlation are correctly tuned and both vanishing; observe the analogy with our setting, where the Fourier cutoff in the definition of $Q^n$ may be thought of as measuring the correlation scale of the noise. Further interesting applications of the weak convergence approach to SPDEs (outside of fluid dynamics) include \cite{BrGoJe17, CerDeb19a,DoWuZhZh}.
On other remarkable works on (possibly singular) SPDEs in the style of \cite{CerDeb19b}, where one cannot naively apply the weak convergence approach, but which require a fine tuning of the correlations of the noise (as measured e.g. by a mollification parameter $\delta$) and its intensity (and possibly keep track of the renormalization constants), let us mention Hairer and Weber \cite{HaiWeb} for Allen-Cahn equations, Mariani \cite{Mariani} for stochastic conservation laws and Gess and Fehrman \cite{FehGes1} for stochastic porous-media equations.

\textit{Central Limit Theorems for SPDEs.}
A classical reference proving CLTs for various interacting particle systems is the work of Holly and Stroock \cite{HolStr}, who also introduced the theory of generalized Ornstein-Uhlenbeck processes.
Several CLTs for fluid dynamics models with multiplicative noise can be found in the literature, including: $2$D Navier--Stokes \cite{WangZhaiZhang15}, $2$D primitive equations \cite{ZhangZhouGuo19} and stochastic reaction-diffusion equations \cite{XiongZhang21}. The aforementioned works also employ the weak convergence approach to LDP to establish moderate deviation principles (MDP); we expect an MDP to hold also in our setting, up to readapting the proofs.
Another class of CLTs for parabolic SPDEs with multiplicative noise can be found in \cite{CKNP, HuLiWa, HNVZ}.
All of the above works however do not include the case of multiplicative \textit{transport} noise; in this sense, a much closer paper to our setting is the one by Slav\'ik \cite{Slavik21}, proving an LDP, MDP and CLT for stochastic $3$D viscous primitive equations.
At the same time the equation considered in \cite{Slavik21} is truly parabolic, while our setting is of hyperbolic nature with conservative noise, which makes it closer to the results obtained in \cite{DiFeGe} (resp. \cite{GeGvKo}), treating conservative local (resp. nonlocal) SPDEs. Like here, \cite{Slavik21} and \cite{DiFeGe} establish strong convergence for the fluctuations, while \cite{GeGvKo} even provides a rate of convergence which is shown to be optimal.
An alternative pathwise approach to CLTs for SPDEs is being developed in \cite{Gussetti}, in the context of Landau-Lifschitz equation.
Let us conclude by mentioning the nice analogy of our setting to that of McKean-Vlasov particle systems, like those treated classically in \cite{Tanaka}: therein the empirical measure $\mu^N=N^{-1} \sum_{i=1}^N \delta_{X^i}$ can be shown to solve an SPDE, which is of the form \eqref{eq:SEE-intro} up to replacing our noise $\sqrt{\eps_n}\circ \d W^{n,\alpha} \cdot \nabla \xi^n = \sqrt{\eps_n}\, \nabla\cdot(\xi^n \circ \d W^{n,\alpha})$ with the \textit{random current}
\[
\nabla \cdot\bigg(\frac{1}{N}\sum_{i=1}^N \delta_{X^i} \circ \d B^i\bigg) =: \frac{1}{\sqrt{N}} \nabla \cdot J^N.
\]
Similarly to our hyperbolic setting, the empirical measures $\mu^N$ do not improve their regularity over time, yet they converge as $N\to\infty$ to the solution $\bar{\mu}$ of a deterministic parabolic Fokker--Plack PDE.
In this setting, it can be shown that the currents $J^N$ converge in law to a non-trivial limit as $N\to\infty$ and that the fluctuations $\sqrt{N}(\mu^N-\bar\mu)$ converge in law to an Ornstein-Uhlenbeck process; see the work of Fernandez and M\'el\'eard \cite{FerMel} for regular interaction kernels and more recently Wang, Zhao and Zhu \cite{WaZhZh} for singular kernels.

\subsection{Notations and conventions}\label{subs-notations}
Throughout the paper, we fix a $T>0$ and assume the functional spaces have time variable $t\in [0,T]$.
The notation $x\lesssim y$ stands for $x\leq C y$ for some constant $C>0$; we write $x \lesssim_\lambda y$ to stress the dependence $C=C(\lambda)$ and $x\sim y$ if $x\lesssim y$ and $y\lesssim x$.
We shall write $x\cdot y$ to denote inner product on $\R^d$, $|x|^2= x\cdot x$; $\<\cdot, \cdot\>$ stands for the inner product in $L^2$ or the duality between functions and distributions.

For $m\geq 1$, $p,q\in [1,\infty]$, we write $L^q_t L^p_x$ for the space $L^q(0,T; L^p(\T^d, \R^m))$, with norm
\[
\| f\|_{L^q L^p}
:= \bigg( \int_0^T \| f_s\|_{L^p}^q \,\d s\bigg)^{1/q}
= \bigg( \int_0^T \Big( \int_{\T^d} |f_s(x)|^p\, \d x\Big)^{q/p} \d s\bigg)^{1/q}.
\]
If $p=q$ then we will simply write $L^p_{t,x}$ with norm $\| \cdot\|_{L^p_{t,x}}$; instead, $L^p$ and $L^p_x$ will be used as a shortcut for $L^p(\T^d, \R^m)$ with norm $\| \cdot\|_{L^p}$.
As already mentioned, $H^s=H^s(\T^d, \R^m)\, (s\in \R)$ will denote standard fractional Sobolev spaces, while $C^\beta=C^\beta(\T^d, \R^m)\, (\beta\geq 0)$ will denote H\"older spaces,  with respective norms $\| \cdot\|_{H^s}$ and $\| \cdot \|_{C^\beta}$.
Given a Banach space $E$, $\beta\in (0,1)$, we set $C^\beta_t E := C^\beta([0,T];E)$ endowed with
\[
\| f\|_{C^\beta E} = \sup_{t\in [0,T]} \| f_t\|_E + \llbracket f \rrbracket_{C^\beta E}, \quad \llbracket f \rrbracket_{C^\beta E} =\sup_{s\neq t} \frac{\|f_t-f_s \|_E}{|t-s|^\beta} .
\]
Similarly, we consider $C^0 _tE$ endowed with the supremum norm and $L^q_t E$ with
\[
\| f\|_{L^q_t E} = \bigg( \int_0^T \| f_s\|_E^q\, \d s \bigg)^{1/q}.
\]
Such definitions apply for the above choices of $E$, allowing to define $L^2_t \cH^\alpha$, $C^\beta_t C^1$, $C^0_t H^s$, etc.
Finally, we write $P_t$ for the heat semigroup $e^{t\Delta},\, t\geq 0$.

\medskip

\noindent \textbf{Structure of the paper.}
Section \ref{sect-LDPs} is devoted to the proofs of the large deviation principles.
More specifically, we first recall briefly the weak convergence approach for LDP in Section \ref{subs-LDP-abstract}, then we prove LDPs for solutions to stochastic linear transport equations (Section \ref{subsec-LDP-SLTE}) and to stochastic 2D Euler equations (Section \ref{subs-LDP-SEE});
finally, we discuss further generalizations in Section \ref{sec:conclusive-LDP}.
Section 3 instead revolves around the study of the Gaussian fluctuations; strong convergence results with explicit rates, for the same equations, are proved respectively in Sections \ref{subs-CLT-SLTE} and \ref{subs-CLT-SEE}. Some concluding remarks are given in Section \ref{sec:conclusive-CLT}.
In Appendix \ref{appendix} we collect some technical results used in the paper.

\medskip

\noindent \textbf{Acknowledgements.} LG is funded by the DFG under Germany's Excellence Strategy - GZ 2047/1,
project-id 390685813. DL would like to thank the financial supports of the National Key R\&D Program of China (No. 2020YFA0712700), the National Natural Science Foundation of China (Nos. 11931004, 12090014), and the Youth Innovation Promotion Association, CAS (Y2021002). Both authors thank Massimo Sorella for pointing out the reference \cite{BCC}.

\section{Large Deviation Principles} \label{sect-LDPs}

This section consists of four parts.
In Section \ref{subs-LDP-abstract}, we recall the abstract framework of weak convergence approach for LDPs.
We then establish in Section \ref{subsec-LDP-SLTE} a LDP for solutions to stochastic transport equations in three steps: i) check well-posedness of skeleton equation; ii) prove a LDP for $L^\infty$-initial data by weak convergence method; iii) transfer it to $L^2$-initial data by $\Gamma$-convergence arguments.
Similar techniques are applied in Section \ref{subs-LDP-SEE} to prove a LDP for solutions to \eqref{eq:SEE-intro} with $L^\infty$-initial vorticity.
Finally, Section \ref{sec:conclusive-LDP} contains some further remarks and possible future problems.

\subsection{Weak convergence approach for LDPs}\label{subs-LDP-abstract}

We recall here several results from \cite[Section 2]{BDM08}, which gives an abstract framework of LDP; see also \cite[Section 3]{SS06} for a concise introduction.

Let $\big(\Omega, \F, \{\F_t\}_{t\in [0,T]}, \P\big)$ be a stochastic basis satisfying the usual conditions. Let $U$ be a Hilbert space and $Q$ a trace class operator on $U$; the subspace $U_0 = Q^{1/2}(U)$, endowed with the inner product
\begin{equation*}
\<g,h\>_0 = \big\<Q^{-1/2}g, Q^{-1/2}h \big\>_U, \quad g,h\in U_0,
\end{equation*}
is also a Hilbert space; we write the norm in $U_0$ as $|\cdot |_0$. Let $\{W(t) \}_{t\in [0,T]}$ be a $Q$-Wiener process on $U$; then it is a cylindrical Wiener process on $U_0$. Define
\begin{equation}\label{eq:defn-S^M}
S^M=S^M(U_0):= \bigg\{ v\in L^2_t U_0: \int_0^T |v(s)|_0^2 \,\d s \leq M \bigg\},
\end{equation}
which is a Polish space when endowed with the weak topology; for any $v\in S^M$, we shall write $\text{Int}(v)(\cdot)= \int_0^\cdot v(s)\,\d s$. We denote by $\mathcal{P}_2(U_0)$  the class of $U_0$-valued $\mathcal{F}_t$-predictable processes $u$ such that $\int_0^T |u(s)|^2_0\, \d s<\infty$ $\P$-a.s. and set
\begin{equation}\label{eq:defn-P^M_2}
\mathcal{P}^M_2=\mathcal{P}^M_2(U_0):= \big\{ u\in \mathcal{P}_2(U_0): u(\cdot,\omega)\in S^M(U_0) \mbox{ for } \P\text{-a.s. } \omega \big\}.
\end{equation}
In the following, $\mathcal{E}$ and $\mathcal{E}_0$ denote Polish spaces.

\begin{definition}
A function $I: \mathcal{E}\to [0,\infty]$ is called a rate function if for any $M<\infty$, the level set $\{f\in \mathcal{E}: I(f)\leq M\}$ is a compact subset of $\mathcal{E}$.
A family of rate functions $I_x$ on $\mathcal{E}$, parametrized by $x\in\mathcal{E}_0$,
is said to have compact level sets on compacts if for all compact subsets $\mathcal K$ of $\mathcal{E}_0$ and each $M<\infty$, $\cup_{x\in \mathcal K} \{f\in \mathcal{E}: I_x(f)\leq M\}$ is a compact subset of $\mathcal{E}$.
\end{definition}

Let $I$ be a rate function on $\mathcal{E}$; for any Borel set $A\in \mathcal{B}(\mathcal{E})$, set $I(A):=\inf_{f\in A} I(f)$. Recall that a family $\{X^\eps\}$ of $\mathcal{E}$-valued random variables is said to safisfy a large deviation principle (LDP) with speed $\eps$ and rate function $I$ if
\begin{equation*}
-I(A^\circ) \leq \liminf_{\eps\to 0} \eps \log \P(X^\eps\in A) \leq \limsup_{\eps\to 0} \eps \log \P(X^\eps\in A) \leq -I(\bar{A})\quad \forall\, A\in \mathcal{B}(\mathcal{E}),
\end{equation*}
where $A^\circ$ and $\bar{A}$ denote respectively the interior and closure of $A$.
A classical result in large deviations theory is that on Polish space the LDP is equivalent to the Laplace principle, see \cite[Theorem 1]{BDM08}.

In practical cases, the random variables $X^\eps$ are often strong solutions to some S(P)DEs driven by $\sqrt\eps W$ with initial data $x\in\mathcal{E}_0$, so that they can be represented as $X^{\eps,x}= \mathcal G^\eps(x, \sqrt\eps W)$ for some family of measurable maps $\mathcal G^\eps: \mathcal{E}_0\times C^0_t U \to \mathcal{E}$.
In these cases, rather than establishing directly an LDP for $\{X^{\eps,x}\}_\eps$ for fixed $x\in \mathcal{E}_0$, it is natural to investigate the following stronger version of Laplace principle (see \cite[Definition 5]{BDM08}).

\begin{definition}
Let $I_x$ be a family of rate functions on $\mathcal{E}$ parameterized by $x\in\mathcal{E}_0$ and assume that this family has compact level sets on compacts.
The family $\{X^{\eps,x}\}$ is said to satisfy the Laplace principle on $\mathcal{E}$ with rate function $I_x$, uniformly on compacts, if for all compact subsets $\mathcal K$ of $\mathcal{E}_0$ and all bounded continuous functions $G$ mapping $\mathcal{E}$ into $\R$, one has
\begin{equation}\label{eq:uniform-laplace}
\lim_{\eps\to 0}\, \sup_{x\in \mathcal K} \bigg| \eps \log \E_x \bigg[ \exp\Big(-\frac{1}{\eps} G(X^{\eps,x})\Big)\bigg] + \inf_{f\in \mathcal{E}} \{G(f)+I_x(f)\}\bigg| = 0.
\end{equation}
\end{definition}

In \cite{BDM08}, the authors provide practical assumptions in order to verify the validity of \eqref{eq:uniform-laplace}.

\begin{hypothesis}\label{hyp:LDP}
There exists a measurable $\mathcal G^0: \mathcal{E}_0\times C^0_t U \to \mathcal E$ such that:
\begin{itemize}
\item[1.] For any $M<\infty$ and compact set $\mathcal K\subset \mathcal{E}_0$, $\Gamma_{\mathcal K, M}:= \big\{\mathcal{G}^0 (x, \text{Int}(v) ) : v\in S^M,\, x\in \mathcal K \big\}$ is a compact subset of $\mathcal{E}$.
\item[2.] Consider $M<\infty$ and families $\{x^\eps \}\subset \mathcal{E}_0,\, \{u^\eps\}\subset \mathcal{P}^M_2$ such that, as $\eps \to 0$, $x^\eps \to x$ and $u^\eps$ converge in law to $u$ as $S^M$-valued random elements.
Then $\mathcal G^\eps \big(x^\eps, \sqrt{\eps}\, W + \text{Int}(v^\eps) \big)$ converges in law to $\mathcal{G}^0 (x, \text{Int}(v) )$ in the topology of $\mathcal{E}$.
\end{itemize}
\end{hypothesis}
\begin{theorem}[\cite{BDM08}, Theorem 5]\label{thm-LDP-abstract}
Let $X^{\eps,x} = \mathcal{G}^\eps(x,
\sqrt\eps W)$ and suppose that Hypothesis \ref{hyp:LDP} holds.
For $x\in \mathcal{E}_0$ and $f\in \mathcal{E}$, define
\begin{equation}\label{eq:rate-funct}
I_x(f):= \inf_{\{ v\in L^2_t U_0:\, f= \mathcal G^0 (x,{\rm Int}(v)) \}} \bigg\{\frac12 \int_0^T |v(s)|_0^2 \,\d s \bigg\},
\end{equation}
with the convention that $\inf \emptyset = \infty$.
Assume that for all $f\in \mathcal{E}$, $x\mapsto I_x(f)$ is a lower semicontinuous map from $\mathcal{E}_0$ to $[0,\infty]$.
Then for all $x\in \mathcal{E}_0$, $f\mapsto I_x(f)$ is a rate function on $\mathcal{E}$ and the family $\{I_x\}_{x\in \mathcal{E}_0}$ of rate functions has compact level sets on compacts.
Furthermore, the family $\{ X^{\eps,x} \}$ satisfies the Laplace principle on $\mathcal{E}$, with the rate functions $\{I_x\}$, uniformly on compact subsets of $\mathcal{E}_0$.
\end{theorem}

\subsection{Stochastic transport equation}\label{subsec-LDP-SLTE}

The purpose of this section is to prove a LDP for the solutions to stochastic transport equations:
\begin{equation}\label{SLTE}
\d f^n + b\cdot\nabla f^n\, \d t + \circ \d \Pi_n \big(\sqrt{\eps_n}\, W^\alpha \big) \cdot \nabla f^n =0, \quad f^n\vert_{t=0}=f_0,
\end{equation}
where we recall that the noise $W^{n,\alpha}$ is defined as in Section \ref{sec:intro-noise} and $\circ\d$ denotes Stratonovich differential.
We will enforce the following:

\begin{assumption}\label{ass:drift-transport}
There exists parameters $(q,p)\in [2,\infty]$ such that
\begin{equation*}
b\in L^q_t L^p_x,\quad \frac{2}{q}+\frac{d}{p}<1;
\end{equation*}
moreover there exists $r\in \big(\frac{2d}{d+2},\infty \big]$ such that\,\footnote{The condition $\nabla \cdot b\in L^2_t L^r_x$ is non-standard and we believe it might be removed, at the price of slightly more technical proofs. However, the main applications we have in mind for our models come from fluid dynamics, where often $\nabla\cdot b\equiv 0$, so that this additional constraint is automatically satisfied.}
\begin{equation*}
\nabla\cdot b\in L^1_t L^\infty_x\cap L^2_t L^r_x  .
\end{equation*}
\end{assumption}

Under Assumption \ref{ass:drift-transport}, for any given $f_0\in L^2$ and $n\in \N$, there exists a unique, analytically weak, probabilistically strong solution to \eqref{SLTE}, which is of the form $f^n_t(x)=f_0\big((X^n_t)^{-1}(x) \big)$, where $\{X^n_t \}_{t\in [0,T]}$ denotes the stochastic flow associated to the underlying Stratonovich SDE:\footnote{We have not found a direct reference for the above result, which appears to be quite sparse in the literature. The existence of a flow $\{X^n_t\}$ associated to SDE follows from \cite[Theorem 1.1]{zhang2011}, while the proof that $f^n_t$ as defined above is a weak solution to the transport equation comes from \cite[Proposition 2.3]{zhang2010}.
Equation (2.2) from the same paper allows to derive the pathwise bound \eqref{eq:a-priori-transport}, although technically the author established its validity under stronger requirements on $b$; this issue can be solved by an approximation argument.
Uniqueness for sufficiently regular $f_0$ has been shown in \cite{FedFla}, but for general integrable initial data one would need to readapt the commutators approach from \cite{FGP}.
An alternative route based on a flow transformation, which fits the assumptions on our coefficients, is sketched in \cite[Section 1.9]{BFGM}.}
$$\d X^n_t= \sqrt{\eps_n} \sum_{|k|\leq n,\, i} |k|^{-\alpha} \sigma_{k,i}(X^n_t) \circ\d B^{k,i}_t + b(t, X^n_t)\,\d t, \quad X^n_0 =x. $$
Noting that $\sigma_{k,i} \cdot \nabla \sigma_{k,i} \equiv 0$ due to our choice of vector fields $\{\sigma_{k,i} \}$, the SDE has the same It\^o form whose noise has constant covariance matrix.
Moreover since $\sigma_{k,i}$ are divergence free, for any $p\in [1,\infty]$, $f^n$ satisfies the pathwise bound
\begin{equation}\label{eq:a-priori-transport}
\sup_{t\in [0,T]} \| f^n_t \|_{L^p} \leq \exp(\| \nabla\cdot b\|_{L^1 L^\infty}) \| f_0\|_{L^p}\quad \P\text{-a.s.}
\end{equation}
As mentioned in the introduction, equation \eqref{SLTE} has equivalent It\^o form
\begin{equation}\label{SLTE-Ito}
\d f^n + b\cdot \nabla f^n\, \d t + \d \Pi_n \big(\sqrt{\eps_n}\, W^\alpha \big) \cdot \nabla f^n = \Delta f^n \,\d t.
\end{equation}
The solutions $\{f^n\}_n$ converge weakly to the unique solution $f$ of the advection-diffusion equation:
\[\partial_t f + b\cdot\nabla f= \Delta f, \quad f|_{t=0} =f_0.\]

Our aim is to establish a LDP for the laws of $\{f^n \}_n$ using the weak convergence method recalled in the last section.
The rest of this section consists of two parts: in Subsection \ref{subsec-skeleton-eq}, we study the well posedness of the skeleton equation and introduce the rate functions; we then provide in Subsection \ref{subs-LDP-SLTE-bounded} the proof of the LDP, first establishing it for $L^\infty$-initial data by the weak convergence method, and then transferring it to the $L^2$-setting via a $\Gamma$-convergence argument.

\subsubsection{Study of the skeleton equation and rate function} \label{subsec-skeleton-eq}

In order to carry out the weak convergence approach to LDP, a key role is played by the well-posedness and stability properties of the skeleton equation
\begin{equation}\label{eq:skeleton}
\partial_t f+ b\cdot\nabla f + g\cdot \nabla f= \Delta f
\end{equation}
ranging over $g\in L^2_t \cH^\alpha$, while keeping $b$ fixed, with suitable initial data $f_0$.

It actually turns out to be useful to take a slightly more general perspective and study the family of equations
\begin{equation}\label{eq:skeleton-modified}
\partial_t f - \Delta f = h\cdot\nabla f + c f
\end{equation}
for $h:[0,T]\times \T^d\to\R^d$, $c:[0,T]\times \T^d\to\R$ satisfying the following:
\begin{assumption}\label{ass:coefficients-skeleton}
There exists a parameter $r> \frac{2d}{d+2}$ such that
\begin{equation*}
h\in L^2_t L^2_x,\quad \nabla\cdot h,\, c\in L^2_t L^r_x\cap L^1_t L^\infty_x.
\end{equation*}
\end{assumption}

\begin{remark}\label{rem:generalized-skeleton}
Allowing for the presence of $c$, we can also treat Fokker-Planck equations by choosing $c=\nabla\cdot h$. Observe that equation \eqref{eq:skeleton} is of the form \eqref{eq:skeleton-modified} for $c\equiv 0$, $h=-b-g$;
in particular, since $b$ satisfies Assumption \ref{ass:drift-transport} and $g\in L^2_t \mathcal{H}^\alpha$, Assumption \ref{ass:coefficients-skeleton} holds as well.
\end{remark}

\begin{proposition}\label{prop:stability-strong-skeleton}
For any $f_0\in L^\infty$, $(h,c)$ satisfying Assumption \ref{ass:coefficients-skeleton}, there exists a unique weak solution $f\in L^\infty_{t,x}$ to \eqref{eq:skeleton-modified}, belonging to $C^0_t L^2_x \cap L^2_t H^1_x$; moreover, given two solutions $f^i$ associated to data $f_0^i,\, h^i,\, c^i$, there exists a constant $C=C(\| \nabla \cdot h^i\|_{L^1 L^\infty\cap L^2 L^r}, \| c^i\|_{L^1 L^\infty\cap L^2 L^r})$ such that
\begin{equation}\label{eq:stability-strong-skeleton}\begin{split}
\| f^1-f^2& \|_{C^0 L^2} + \|f^1-f^2\|_{L^2 H^1}\\
\leq &\, C \big[ \| f_0^1-f_0^2\|_{L^2} + \big(\|h^1-h^2\|_{L^2_{t,x}}+ \| c^1-c^2-\nabla\cdot (h^1-h^2)\|_{L^2 L^r} \big) \| f^2_0\|_{L^\infty} \big].
\end{split}\end{equation}
Finally, given sequences $f^n_0\stackrel{\ast}{\rightharpoonup} f_0$, $(h^n,c^n) \rightharpoonup (h,c)$, while satisfying uniform bounds in the norms appearing in Assumption \ref{ass:coefficients-skeleton},  the associated solutions $f^n$ converge weakly-$\ast$ in $L^{\infty}_{t,x}$ to $f$, as well as strongly in $L^2_{t,x}\cap C^0_t H^{-\delta}_x$ for any $\delta>0$.
\end{proposition}

\begin{proof}
Existence of solutions satisfying the $L^\infty$-bound follows from usual compactness arguments and the maximum principle (alternatively, Feynman-Kac formula), since for smooth coefficients we have the estimate
\begin{equation}\label{eq:skel-basic-bound}
\sup_{t\in [0,T]} \| f_t\|_{L^\infty} \leq \exp(\| c\|_{L^1 L^\infty}) \| f_0\|_{L^\infty}.
\end{equation}
Under our assumptions, $h\cdot\nabla f + c f = \nabla \cdot (h f) + (c-\nabla \cdot h) f$ is a well defined distribution and there is no problem in taking the weak-$\ast$ limit in the compactness argument.

Observe that, if $f\in L^\infty_{t,x}$ is a weak solution, then by Assumption \ref{ass:coefficients-skeleton} and Sobolev embeddings $\nabla \cdot (h f)\in L^2_t H^{-1}_x$, $(c-\nabla\cdot h) f\in L^2_t L^r_x \hookrightarrow L^2_t  H^{-1}_x$, so that $f=(\partial_t-\Delta)^{-1} g$ for some $g\in L^2_t H^{-1}$. Maximal regularity for parabolic equations implies that $f\in L^2_t H^1_x$, $\partial_t f\in L^2_t H^{-1}_x$.

Assume now we are given two solutions $f^i$ as above, then setting $\xi:=f^1-f^2$ it holds $\xi\in L^2_t H^1_x\cap L^\infty_{t,x}$, $\partial_t \xi\in L^2_t H^{-1}_x$ and
\begin{equation*}
\partial_t \xi -\Delta \xi = h^1\cdot\nabla \xi + c^1 \xi + (h^1-h^2)\cdot\nabla f^2 + (c^1-c^2) f^2.
\end{equation*}
Given the regularity of $\xi$, we are allowed to apply Lions-Magenes lemma (see \cite[Chap. 1, Theorem 3.1]{LM72} or \cite[Lemma 2.1.5]{KS12}); arguing by approximation, we can also extend the relation $2\langle \varphi_s\cdot \nabla \psi_s, \psi_s\rangle = -\langle \nabla \cdot \varphi_s, \psi_s^2\rangle$ (for a.e. $s\in [0,T]$) to any $\varphi\in L^2_t L^2_x$ such that $\nabla\cdot\varphi\in L^1_t L^\infty_x$ and any $\psi\in L^\infty_{t,x} \cap L^2_t H^1_x$. Thus by integration by parts we obtain
\begin{align*}
\frac{\d}{\d t} \frac{\| \xi_t\|_{L^2}^2}{2}
& = - \|\nabla \xi_t\|_{L^2}^2 + \Big\langle c^1_t-\frac12 \nabla\cdot h^1_t,\xi^2_t \Big\rangle -\langle (h^1_t-h^2_t)\cdot \nabla \xi_t, f^2_t \rangle \\
&\quad\, + \langle c^1_t-c^2_t -\nabla\cdot (h^1_t-h^2_t), f^2_t \xi_t \rangle\\
& \leq - \|\nabla \xi\|_{L^2}^2
+ \big( \|c^1_t\|_{L^\infty}+\| \nabla\cdot h^1_t\|_{L^\infty} \big) \|\xi_t\|_{L^2}^2
+ \| h^1_t-h^2_t\|_{L^2} \| f^2_t\|_{L^\infty} \| \nabla \xi_t\|_{L^2}\\
& \quad \, + \| c^1_t -c^2_t -\nabla\cdot (h^1_t-h^2_t)\|_{L^r} \| f^2_t\|_{L^\infty} \| \xi_t\|_{H^1}.
\end{align*}
Cauchy's inequality implies
\begin{align*}
\frac{\d}{\d t} \frac{\| \xi_t\|_{L^2}^2}{2}& \leq -\frac{1}{2} \| \nabla \xi_t \|_{L^2}^2  + \big(\|c^1_t\|_{L^\infty}+\| \nabla\cdot h^1_t \|_{L^\infty} +1 \big) \| \xi_t\|_{L^2}^2 \\
& \quad \, + \big(\| h^1_t-h^2_t\|_{L^2}^2 + \| c^1_t -c^2_t -\nabla\cdot (h^1_t-h^2_t)\|_{L^r}^2 \big) \| f^2\|_{L^\infty_{t,x}}^2 .
\end{align*}
An application of Gronwall's lemma, together with \eqref{eq:skel-basic-bound}, readily yields
\begin{align*}
\| \xi\|_{C^0 L^2}^2 + \int_0^T \|\nabla \xi_t\|_{L^2}^2 \d t
&\lesssim \| \xi_0\|_{L^2}^2  + \big(\| h^1-h^2\|_{L^2_{t,x}}^2 + \| c^1 -c^2 -\nabla\cdot (h^1-h^2)\|_{L^2 L^r}^2 \big) \| f^2\|_{L^\infty_{t,x}}^2\\
& \lesssim \| \xi_0\|_{L^2}^2  + \big( \| h^1-h^2\|_{L^2_{t,x}}^2 + \| c^1 -c^2 -\nabla\cdot (h^1-h^2)\|_{L^2 L^r}^2 \big) \| f^2_0\|_{L^\infty}^2
\end{align*}
which is exactly \eqref{eq:stability-strong-skeleton}; this also establishes uniqueness of solutions.

It remains to prove the claim concerning stability in the weak topology.
Having established uniqueness, we know that any solution satisfies the bound \eqref{eq:skel-basic-bound}.
Moreover, applying again maximal regularity theory, one has the uniform bounds
\begin{align*}
\| \partial_t f\|_{L^2 H^{-1}} + \| f\|_{L^2 H^1}
& \lesssim \| \nabla \cdot ( h f) + (c-\nabla \cdot h) f\|_{L^2 H^{-1}}\\
& \lesssim (\| h\|_{L^2_{t,x}} + \| c-\nabla \cdot h\|_{L^2 L^r}) \| f\|_{L^\infty_{t,x}}\\
& \lesssim (\| h\|_{L^2_{t,x}} + \| c-\nabla \cdot h\|_{L^2 L^r}) \exp(\| c\|_{L^1 L^\infty}) \|f_0\|_{L^\infty} .
\end{align*}
Given $f^n_0\ \stackrel{\ast}{\rightharpoonup} f_0$, $(h^n,c^n)\rightharpoonup (h,c)$, it follows that the associated solutions $f^n$ satisfy uniform bounds in $C^{1/2}_t H^{-1}_x \cap C^0_t L^2_x \cap L^2_t H^1_x$;
one can deduce by Aubin-Lions lemma that $\{f^n\}_n$ is sequentially compact in $C^0_t H^{-\delta}_x \cap L^2_{t,x}$ for any $\delta>0$.
From here, standard arguments allow to show that any limit point must be a weak  $L^\infty_{t,x}$ solution to \eqref{eq:skeleton-modified}; for instance, combining $f^n\to f$ in $L^2_{t,x}$, $\sup_n \| f^n\|_{L^\infty_{t,x}}<\infty$ and $c_n\rightharpoonup c$ in $L^2_t L^r_x$, it is easy to check that $c^n f^n \rightharpoonup c f$ in $L^1_{t,x}$.
But then necessarily $f^n$ weakly-$\ast$ converge in $L^\infty_{t,x}$ to the unique solution $f$ associated to $(h,c,f_0)$.
\end{proof}

We can build on the previous result to extend the well-posedness theory for \eqref{eq:skeleton-modified} to $f_0\in L^2$.

\begin{proposition}\label{prop:stability-skeleton-L2}
For any $f_0\in L^2$, $(h,c)$ satisfying Assumption \ref{ass:coefficients-skeleton}, there exists a unique weak solution $f\in L^\infty_t L^2_x\cap L^2_t H^1_x$ to \eqref{eq:skeleton-modified}; moreover, given two solutions associated to data $f_0^i,\, h^i\, c^i$, there exists a constant $C=C(\| \nabla \cdot h^i\|_{L^1 L^\infty}, \| c^i\|_{L^1 L^\infty}, \| c^i-\nabla h^i\|_{L^2 L^r})$ such that
\begin{equation}\label{eq:stability-skeleton-L^2}
\| f^1-f^2\|_{L^\infty L^1} \leq C\big[ \| f_0^1-f_0^2\|_{L^1} + \| h^1-h^2\|_{L^2 L^2} \| f^2_0\|_{L^2} + \| c^1-c^2\|_{L^2 L^r} \| f^2_0\|_{L^2} \big].
\end{equation}
Finally, given sequences $f^n_0\rightharpoonup f_0$, $(h^n,c^n) \rightharpoonup (h,c)$, while satisfying uniform bounds in the norms appearing in Assumption \eqref{ass:coefficients-skeleton}, the associated solutions $f^n$ converge weakly to $f$ in $L^2_t H^1_x$, as well as strongly in $L^2_{t,x}\cap C^0_t H^{-\delta}_x$ for any $\delta>0$.
\end{proposition}

\begin{proof} 
The proof is similar to that of Proposition \ref{prop:stability-strong-skeleton}, so we mostly sketch it. Formally testing the solution $f$ against equation \eqref{eq:skeleton-modified} and applying Gronwall's lemma, one obtains the a priori estimate
\begin{equation}\label{eq:skel-energy-estimate}
\sup_{t\in [0,T]} \|f_t\|_{L^2}^2 + \int_0^T \| \nabla f_t\|_{L^2}^2\, \d t \leq \exp\big(\| \nabla\cdot h\|_{L^1 L^\infty} + 2 \| c\|_{L^1 L^\infty} \big)\,\| f_0\|_{L^2}^2.
\end{equation}
Together with compactness arguments, this allows to construct weak solutions satisfying \eqref{eq:skel-energy-estimate} (the bound on $L^2_t H^1_x$-norm, together with Sobolev embedding, shows that $h\cdot\nabla f+ c f\in L^1_{t,x}$, so that the weak formulation is meaningful).
Uniqueness can be shown by linearity and estimates similar to those yielding \eqref{eq:stability-skeleton-L^2}, thus we will only focus on proving the latter starting from any two solutions $f^i$ satisfying \eqref{eq:skel-energy-estimate}.
Setting $\xi:=f^1-f^2$, in a weak sense it holds
\begin{equation*}
\partial_t \xi -\Delta \xi = h^1\cdot\nabla \xi + c^1 \xi + (h^1-h^2)\cdot\nabla f^2 + (c^1-c^2) f^2.
\end{equation*}

Now for any $\varphi\in L^\infty$ and $t\in (0,T]$, let $\{g_s \}_{s\in [0,t]}$ be the unique weak solution in $L^\infty$ to the backward equation
\begin{equation*}
\partial_s g + \Delta g = \nabla \cdot (h^1\, g) -  c^1\, g
= h^1\cdot\nabla g - (c^1-\nabla\cdot h^1) g, \quad g\vert_{s=t}
= \varphi,
\end{equation*}
which can be constructed in the same way as in Proposition \ref{prop:stability-strong-skeleton} (since $(h^1,c^1)$ satisfy Assumption \ref{ass:coefficients-skeleton}, so do $(h^1,c^1-\nabla\cdot h^1)$); $g$ satisfies $\| g_s\|_{L^\infty} \lesssim \| \varphi\|_{L^\infty}$ for all $s\in [0,t]$.

Testing $g$ against the equation for $\xi$ (to be rigorous, one should argue by approximation or by mollifying $g$ and proving that the commutators vanish in the limit; we omit the details here since the arguments are standard) and integrating over $[0,t]$, we obtain
\begin{align*}
|\langle \xi_t, \varphi\rangle|
& = \bigg |\langle \xi_0, g_0\rangle + \int_0^t \langle (h^1_s-h^2_s)\cdot\nabla f^2_s, g_s\rangle\, \d s + \int_0^t \langle (c^1_s-c^2_s) f^2_s, g_s \rangle\, \d s \bigg|\\
& \leq \| \xi_0\|_{L^1} \| g_0\|_{L^\infty} + \| h^1-h^2\|_{L^2_{s,x}} \| \nabla f^2\|_{L^2_{s,x}} \| g\|_{L^\infty_{s,x}}
+ \| c^1-c^2\|_{L^2 L^r} \| f^2\|_{L^2 L^{r'}} \| g\|_{L^\infty_{s,x}}\\
& \lesssim \big(\|\xi_0\|_{L^1} + \| h^1-h^2\|_{L^2_{s,x}} \| f^2_0\|_{L^2} + \| c^1-c^2\|_{L^2 L^r} \| f^2_0\|_{L^2} \big) \| \varphi\|_{L^\infty}
\end{align*}
where in the last passage we used the Sobolev embedding $L^2_t H^1_x \hookrightarrow L^2_t L^{r'}_x$ and estimate \eqref{eq:skel-energy-estimate} with $f= f^2$.
Taking the supremum over all $\varphi\in L^\infty$ with $\| \varphi\|_{L^\infty}=1$ in the above expression, by duality one readily obtains \eqref{eq:stability-skeleton-L^2}.

The proof of the statement involving weak convergence follows the same line of arguments as in Proposition \ref{prop:stability-strong-skeleton}; the only part which needs a bit more attention is the uniform bound in $C^\alpha_t H^{-s}_x$ for suitable $\alpha$ and $s$, which is required in order to apply Aubin-Lions.
Since by assumption $r>2d/(d+2)$, by Sobolev embeddings there exists $a\in (0,1)$ such that $H^a_x \hookrightarrow L^{r'}_x$ and by interpolation it holds $\| f_t\|_{L^{r'}} \lesssim \| f_t\|_{L^2}^{1-a} \| f_t\|_{H^1}^a$.
Therefore, for any fixed $t$, we have
\[ \|(c_t- \nabla\cdot h_t) f_t\|_{L^1}
\leq \|c_t- \nabla\cdot h_t \|_{L^r} \|f_t\|_{L^{r'}}
\lesssim \|f_t\|_{L^2}^{1-a} \|f_t\|_{H^1}^a \|c_t- \nabla\cdot h_t \|_{L^r_x}. \]
Taking $\delta= (1-a)/(1+a)>0$, so that $a(1+\delta)/(1-\delta)=1$, by H\"older's inequality it holds
\[ \|(c - \nabla\cdot h) f \|_{L^{1+\delta} L^1}
\lesssim \|f\|_{L^\infty L^2}^{1-a} \|f_t\|_{L^2 H^1}^a \|c_t- \nabla\cdot h_t \|_{L^2 L^r}. \]
This implies $(c - \nabla\cdot h) f \in L^{1+\delta}_t H^{-s}_x$ for $s>d/2$. Since $f\in L^\infty_t L^2_x$, one has $hf\in L^2_t L^1_x \hookrightarrow L^2_t H^{-s}_x$ and thus $\nabla\cdot (h f)\in L^2_t H^{-s-1}_x$.
Combining these facts with $\Delta f\in L^2_t H^{-1}_x$, we conclude that
\[ \partial_t f= \Delta f+ \nabla\cdot (h f) + (c - \nabla\cdot h) f \in L^{1+\delta}_t H^{-s-1}_x; \]
in particular, $f\in C^\alpha_t H^{-s-1}_x$ for $\alpha =\delta/(\delta +1)$. All the above estimates are uniform when dealing with a sequence $(b^n,c^n,f^n)$, so that we can apply Aubin-Lions lemma.

Furthermore, convergence of $f^n$ in $L^2_{t,x}$ combined with a uniform bound in $L^2_t H^1_x$ implies convergence in $L^2_t L^{r'}_x$, which together with $(b^n,c^n)\rightharpoonup (b,c)$ implies that $(c^n-\nabla\cdot b^n)f^n\rightharpoonup(c-\nabla\cdot b)f$.
Combining all these facts allow to show that the weak limit of $f^n$ is still a solution to \eqref{eq:skeleton-modified} satisfying \eqref{eq:skel-energy-estimate} and so to conclude.
\end{proof}

In light of Proposition \ref{prop:stability-skeleton-L2} and Remark \ref{rem:generalized-skeleton}, for fixed $b$ satisfying Assumption \ref{ass:drift-transport} we can define the solution map $\mathcal G^0(f_0, {\rm Int}(g))$, which to any $f_0\in L^2_x$ and $g\in L^2_t\mathcal{H}^\alpha$ associates the unique solution $f\in L^2_t H^1_x\cap L^\infty_t L^2_x$ to \eqref{eq:skeleton} (if $f_0\in L^\infty$, we further know that $f\in L^\infty_{t,x}\cap C^0_t L^2_x$).

We can define the corresponding candidate rate function
\begin{equation}\label{rate-funct-1}
I_{f_0}(f)= \inf_{\{ g\in L^2 \cH^\alpha:\, f= \mathcal G^0 (f_0, \textup{Int}(g)) \}} \bigg\{\frac12 \int_0^T \|g_s \|_{H^\alpha}^2 \,\d s \bigg\}
\end{equation}
with the convention $\inf\emptyset=+\infty$.
Clearly, the notation for $\mathcal{G}^0$ and $I_{f_0}$ comes from Section \ref{subs-LDP-abstract}.

We shall take
$$\mathcal E_0 := L^2_x \quad \mbox{endowed with the strong topology},$$
which is a Polish space.
Observe that if $f_0\in L^2$ and $I_{f_0}(f)<\infty$, then $f$ is a solution to \eqref{eq:skeleton}, implying that it satisfies the energy bound \eqref{eq:skel-energy-estimate} with $c\equiv 0$ and $h=-b-g$; in particular, we can restrict the definition of $I_{f_0}$ to the Polish space
$$\mathcal{E}=C^0_t H^{-\delta}\quad \mbox{for any } \delta>0.$$

For reasons that will become clear later, we need to study the dependence $f_0\mapsto I_{f_0}$, where $f_0\in L^2$ and $I_{f_0}$ is regarded as a map on $\mathcal{E}$; we will employ the notion of $\Gamma$-convergence, see \cite{Braides} for a complete account on the topic.

\begin{definition}
Let $E$ be a Polish space, $F$, $F_n:E\to (-\infty,+\infty]$ a family of functionals; we say that $F_n$ $\Gamma$-converge to $F$ if:
\begin{itemize}
\item[i)] lower bound: for any sequence $x_n\to x$ it holds $F(x)\leq \liminf F_n(x_n)$;
\item[ii)] upper bound: for any $x$ there exists a sequence $x_n\to x$ such that $F(x)\geq \limsup F_n(x_n)$.
\end{itemize}
\end{definition}
Let us recall the following facts:
i) if $F_n$ $\Gamma$-converge to $F$ and $G$ is a continuous bounded function, then $F_n+G$ also $\Gamma$-converge to $F+G$;
ii) if $F_n$ $\Gamma$-converge to $F$ and they are mildly equicoercive, i.e. there exists a compact set $\mathcal K\subset E$ such that $\inf_{x\in E} F_n(x)=\inf_{x\in \mathcal K} F_n(x)$ for all $n$, then convergence of the minima holds:
\[\inf_{y\in E} F(y) = \lim_{n\to\infty} \inf_{y\in E} F_n(y).\]

\begin{lemma}\label{lem:Gamma-converg}
Let $f^n_0\rightharpoonup f_0$ in $L^2$ and set $F_n = I_{f_0^n}$, $F=I_{f_0}$; then $F_n$ $\Gamma$-converges to $F$.
\end{lemma}

\begin{proof}
Let us first make a few observations on the properties of the functions $F_n$.

The last assertion of Proposition \ref{prop:stability-skeleton-L2} implies that $\mathcal{G}^0(f_0,{\rm Int} (g))$ is continuous in $f_0\in L^2_x$ and $g\in L^2_t \cH^\alpha$ in the weak topologies;
moreover, the set $E_M:=\big\{g\in L^2_t \cH^\alpha: \| g\|_{L^2_t \cH^\alpha}^2 \leq 2M \big\}$ is compact in the weak topology of $L^2_t \cH^\alpha$ for any $M>0$.
By the direct method in the calculus of variations (see \cite{Dac}), the infimum defining $F_n(f)= I_{f^n_0}(f)$ is always realized as a minimum.

Next, we claim that the family of maps $\{F_n\}$ is mildly equicoercive; the same proof, applied at fixed $n\in\N$, shows that each $F_n$ is lower semicontinuous. By the definition of $F_n$ and its realization as a minimum, for any $M>0$ it holds
\begin{equation*}\begin{split}
\Gamma_M
:=&\, \bigcup_n \big\{f\in C^0_t H^{-\delta}: F_n(f)\leq M\big\}\\
=&\, \bigcup_n \big\{\mathcal{G}^0(f^n_0, {\rm Int} (g)) : g\in L^2_t \cH^\alpha,\, \| g\|_{L^2_t \cH^\alpha}^2 \leq 2M\big\}
= \mathcal{G}^0(A\times E_M )
\end{split}\end{equation*}
for the choice $A=\{f^n_0:n\in\N\}\subset L^2$.
Since by assumption $f^n_0\rightharpoonup f_0$, $A$ is precompact in the weak topology and so is $A\times E_M$; overall we conclude that, for any $M>0$, $\Gamma_M$ is precompact in $C^0_t H^{-\delta}$, implying the claim.

We now pass to the proof of $\Gamma$-convergence, starting with the lower bound. Given $f^n\to f$ in $\mathcal{E}$, without loss of generality we may assume that $\liminf F_n(f^n)<\infty $.
As a consequence, there exists $M<\infty$ and a sequence $g^n\in L^2_t \mathcal{H}^\alpha$ such that $\| g^n\|_{L^2 H^\alpha}^2 \leq 2M$, $f^n=\mathcal{G}^0(f^n_0,{\rm Int} (g^n))$.
By weak compactness, up to extracting a subsequence, we may assume that $g^n\rightharpoonup g$ for some $g\in L^2_t\mathcal{H}^\alpha$; but then by Proposition \ref{prop:stability-skeleton-L2} we deduce that
$f^n=\mathcal{G}^0(f^n_0,{\rm Int} (g^n))\to \mathcal{G}^0(f_0,{\rm Int} (g))=f$ in $\mathcal{E}$.
By lower semicontinuity of strong norms in the weak convergence we get
\begin{equation*}
I_{f_0}(f) \leq \frac{1}{2} \| g\|_{L^2 H^\alpha}^2 \leq \liminf_{n\to \infty} \frac{1}{2} \| g^n\|_{L^2 H^\alpha}^2 = \liminf_{n\to \infty} I_{f_0^n}(f^n).
\end{equation*}

Next we give the proof of the upper bound.
We can assume $I_{f_0}(f)<\infty$, so that there exists $g\in L^2_t \cH^\alpha$ such that $f=\mathcal{G}^0(f_0, \textup{Int}(g))$ and $\| g\|_{L^2_t H^\alpha}^2 = 2 I_{f_0}(f)$.
Let $f_n:=\mathcal{G}^0(f^n_0, \textup{Int}(g))$, then by Proposition \ref{prop:stability-skeleton-L2}, $f_n$ converge to $f$ in $\mathcal E$ and it holds
\begin{equation*}
I_{f_0}(f) = \frac{1}{2}\| g\|_{L^2 H^\alpha}^2 \geq I_{f^n_0}(f_n)
\end{equation*}
for all $n\in\N$, so that the same holds taking the $\limsup$ on the left-hand side.
\end{proof}

\subsubsection{Proofs of Large Deviation Principles}\label{subs-LDP-SLTE-bounded}

As explained before, we will first employ the abstract weak convergence approach to establish in Proposition \ref{thm-bounded-data} an LDP for bounded initial data $f_0\in L^\infty_x$ and then use the structure of the SPDE to further extend the result to general $f_0\in L^2_x$ (see Theorem \ref{thm-L2-data}).
Recall that solutions $f^n$ to \eqref{SLTE} satisfy the pathwise bound \eqref{eq:a-priori-transport}, in particular $\sup_{t\in [0,T]} \| f^n_t\|_{L^\infty} \leq \| f_0\|_{L^\infty}$ $\P$-a.s.

Let us define the spaces
\begin{equation}\label{space-E-T-R}
\mathcal{E}^R_0=\{f_0\in L^\infty_x: \| f_0\|_{L^\infty}\leq R\},\quad
\mathcal E^{T,R}= \Big\{f\in C_w([0,T], L^\infty_x):\sup_{t\in [0,T]}\|f_t \|_{L^\infty} \leq R\Big\},
\end{equation}
where the subscript $w$ now means weak-$\ast$ continuity in time. We endow them respectively with the topologies of $H^{-\delta}$ and $C^0_t H^{-\delta}$ for some $\delta>0$, which makes them Polish spaces (the topology being independent of the choice of $\delta>0$); see Lemma \ref{lem-topology} for a short proof of the first claim, the second one being similar.

The spaces $\mathcal E^R_0$ and $\mathcal E^{T,R}$ will play respectively the roles of $\mathcal{E}_0$ and $\mathcal{E}$ in Hypothesis \ref{hyp:LDP}.
Finally, for $M>0$, recall the spaces $S^M(\mathcal{H}^\alpha)$ and $\mathcal{P}^M_2(\mathcal{H}^\alpha)$ as defined in \eqref{eq:defn-S^M}, \eqref{eq:defn-P^M_2}.

With a slight abuse of notation, we will still use the notation $(f_0,g)\mapsto \mathcal{G}^0(f_0,{\rm Int}(g))$, introduced in the previous section, to denote the restriction of the solution map associated to \eqref{eq:skeleton} for initial data $f_0\in L^\infty_x$; in this case, for any fixed $R>0$, we may consider $\mathcal{G}^0$ as a well defined map from $\mathcal{E}^R_0\times  L^2_t\mathcal{H}^\alpha$ to $\mathcal{E}^{T,R}$.

Thanks to Proposition \ref{prop:stability-strong-skeleton}, $\mathcal{G}^0$ is a continuous map in the weak topologies, which implies the validity of Condition 1 in Hypothesis \ref{hyp:LDP} (continuous image of a compact set is compact).

\begin{lemma}\label{prop-cond-1}
For any $M>0$ and compact subset $\mathcal K\subset \mathcal E^R_0$, the set
\[ \Gamma_{\mathcal K,M}:= \big\{ \mathcal G^0( f_0, \textup{Int}(g)): f_0\in \mathcal K,\, g\in S^M( \cH^\alpha) \big\} \]
is compact in the space $\mathcal E^{T,R}$.
\end{lemma}

In order to check Condition 2 in Hypothesis \ref{hyp:LDP}, we need a few preparations.
For any $f_0\in L^\infty$ and $g\in \mathcal P_2^M(\cH^\alpha)$, consider now the equation
\begin{equation}\label{eq:SLTE-perturbed}
\d f^n + b\cdot\nabla f^n\, \d t + \circ \d \Pi_n \big(\sqrt{\eps_n}\, W^\alpha \big) \cdot \nabla f^n + \Pi_n (g) \cdot \nabla f^n \,\d t =0, \quad f^n|_{t=0} = f_0.
\end{equation}
Arguing as for \eqref{SLTE}, under Assumption \ref{ass:drift-transport}, it admits a unique solution $f^n$ for any initial data $f_0\in L^\infty$, which moreover satisfies the pathwise bound \eqref{eq:a-priori-transport}; in particular, $\P$-a.s. $\sup_{t\in [0,T]} \|f^n_t \|_{L^\infty} \leq \|f_0 \|_{L^\infty}$.
Equation \eqref{eq:SLTE-perturbed} has the corresponding It\^o form
\begin{equation}\label{eq:SLTE-perturbed-Ito}
\d f^n + b\cdot\nabla f^n \,\d t+\d \Pi_n \big(\sqrt{\eps_n}\, W^\alpha \big) \cdot \nabla f^n + \Pi_n (g) \cdot \nabla f^n \,\d t = \Delta f^n\,\d t.
\end{equation}
We shall denote the unique solution by $f^n = \mathcal G^n\big(f_0, \sqrt{\eps_n}\, W^\alpha + \textup{Int}(g)\big)$; our next aim is to prove the following result.

\begin{lemma}\label{prop-cond-2}
Let $\{g^n\}_n \subset \mathcal P_2^M(\cH^\alpha)$ and $\{f^n_0 \}_n \subset \mathcal E^R_0$.
Assume that $f^n_0 \stackrel{\ast}{\rightharpoonup} f_0$ in $L^\infty$ and $g^n$ converge in law, as $S^M(\cH^\alpha)$-valued random elements, to $g$ as $n\to \infty$.
Then  the convergence
\begin{equation*}
\mathcal G^n\big(f^n_0, \sqrt{\eps_n}\, W^\alpha + \textup{Int}(g^n)\big) \to \mathcal G^0(f_0, \textup{Int}(g))
\end{equation*}
holds in law, in the topology of $\mathcal E^{T,R}$.
\end{lemma}

Before presenting the proof, we need to establish some uniform control on the martingale part of the equation \eqref{eq:SLTE-perturbed-Ito}, given by
\begin{equation*}
M^n_t
= \sqrt{\eps_n} \int_0^t \d (\Pi_n W^{\alpha}_s) \cdot \nabla f^n_s
= \sqrt{\eps_n} \sum_{|k|\leq n,\, i} |k|^{-\alpha} \int_0^t \sigma_{k,i}\cdot \nabla f^n_s \,\d B^{k,i}_s.
\end{equation*}

\begin{lemma}\label{lem-martingale}
For any $p\in[1,\infty), \beta>1+d/2$ and $\gamma<1/2$, there exists a constant $C>0$, independent of $n\geq 1$, such that
\begin{equation}\label{eq:martingale-estim}
\E\big[ \| M^n\|_{C^\gamma H^{-\beta}}^{2p} \big] \leq C \eps_n^p R^{2p}.
\end{equation}
\end{lemma}

\begin{proof}
The proof is similar to those given in \cite[Proposition 3.6]{FGL21b} and \cite[Section 2.2]{FGL21c}, so we will mostly sketch it. Denote by $[ M^n]_\beta$ the quadratic variation associated to $\| \cdot\|_{H^{-\beta}}$, i.e. the unique increasing process such that $\| M^n (t)\|_{H^{-\beta}}^2 -[M^n]_\beta(t)$ is a martingale; due to our choice of $B^{k,i}$, it holds
\begin{align*}
\frac{\d }{\d t} [M^n]_\beta = \eps_n \sum_{|k|\leq n,i} |k|^{-2\alpha} \| \sigma_{k,i}\cdot\nabla f^n_t \|_{H^{-\beta}}^2 \leq \eps_n \sum_{k ,i} \| \sigma_{k,i}\cdot\nabla f^n_t \|_{H^{-\beta}}^2.
\end{align*}
As $\sigma_{k,i}$ are divergence free, $\sigma_{k,i}\cdot\nabla f^n_t= \nabla \cdot(\sigma_{k,i}\, f^n_t)$; thus, for any $i\in \{1,\ldots, d-1\}$,
$$\| \sigma_{k,i}\cdot\nabla f^n_t \|_{H^{-\beta}}^2 \lesssim \| \sigma_{k,i} f^n_t \|_{H^{1-\beta}}^2 \lesssim \| e_k f^n_t\|_{H^{1-\beta}}^2 = \sum_{l} |l|^{-2(\beta-1)} |\langle e_k f^n_t, e_l \rangle|^2 . $$
As a result, $\P$-a.s. it holds
\begin{align*}
\frac{\d }{\d t} [M^n]_\beta
&\lesssim \eps_n \sum_{k,l} |l|^{-2(\beta -1)} |\langle f^n_t , e_{l-k} \rangle|^2
= \eps_n \sum_l |l|^{-2(\beta -1)} \|f^n_t \|_{L^2}^2
\lesssim \eps_n \| f^n_t\|_{L^2}^2 \lesssim \eps_n R^2,
\end{align*}
where the last step is due to $\| f^n_t\|_{L^2}\lesssim \| f^n_0 \|_{L^2} \leq R$. With the above estimate at hand, applying Burkholder-Davis-Gundy inequality in Hilbert spaces, for any $p\in [1,\infty)$ it holds
\begin{align*}
\E\big(\| M^n_t-M^n_s\|_{H^{-\beta}}^{2p} \big)
\lesssim \E\big([M^n_\cdot- M^n_s]_\beta^p(t) \big)
\lesssim |t-s|^p\eps_n^p R^{2p};
\end{align*}
estimate \eqref{eq:martingale-estim} then readily follows from Kolmogorov continuity theorem.
\end{proof}

Now we are ready to present the

\begin{proof}[Proof of Lemma \ref{prop-cond-2}]
The argument is as usual by compactness. To conclude, it suffices to check that any subsequence of $\{f^n\}_n$ converges to a weak $L^\infty_{t,x}$ solution $f$ of \eqref{eq:skeleton} associated to $(f_0, g)$; indeed by Proposition \ref{prop:stability-strong-skeleton}, uniqueness holds in this class, so that the limit is given by $\mathcal{G}^0(f_0,{\rm Int}(g))$.

By Assumption \ref{ass:drift-transport} and the pathwise estimate \eqref{eq:a-priori-transport}, up to choosing large $\beta>0$, it holds
\begin{equation*}
\bigg\| \int_0^\cdot \big[(b_s+g^n_s)\cdot\nabla f^n_s - \Delta f^n_s \big]\, \d s\bigg\|_{C^{1/2} H^{-\beta}} \lesssim \big(\| \nabla\cdot b\|_{L^2 L^r} + \| b\|_{L^2 L^2} + \| g^n\|_{L^2 L^2} \big) \| f^n_0\|_{L^\infty};
\end{equation*}
combining this with Lemma \ref{lem-martingale} and interpolating with the uniform $L^2$-bound, we can deduce from classical compact embedding results (cf. \cite[Corollary 9, p. 90]{Simon}) that the laws of $\{f^n\}_n$ are tight in $C^0_t H^{-\delta}$, for any $\delta>0$. Applying the Prohorov theorem and Skorohod's representation theorem, and extracting a subsequence if necessary, we can assume that $f^n$ converge $\P$-a.s., in the topology of $C^0_t H^{-\delta}$ (thus also in $\mathcal{E}^{T,R}$), to some limit process $f$, satisfying $\|f \|_{L^\infty_{t,x}} \leq \|f_0 \|_{L^\infty}$ $\P$-a.s.

On the other hand, we know by assumption that $g^n\to g$ in law in the topology of $S^M(\cH^\alpha)$, namely in the weak topology of $L^2_t H^\alpha_x$. To conclude that the limit point $f$ is a weak solution to \eqref{eq:skeleton}, by \eqref{eq:SLTE-perturbed-Ito} with $g=g^n$, it suffices to check: (i) the martingale part $\d M^n:= \d \Pi_n \big(\sqrt{\eps_n}\, W^\alpha \big) \cdot \nabla f^n$ vanishes as $n\to \infty$; (ii) $\Pi_n(g^n)\cdot\nabla f^n$ converge to $g \cdot\nabla f$ in the sense of distributions.

Claim (i) follows immediately from estimates \eqref{eq:martingale-estim}, since $\eps_n\to 0$. Concerning (ii), first observe that $\Pi_n(g^n)\cdot\nabla f^n= \nabla\cdot ( \Pi_n(g^n) f^n)$, so we only need to show that $\Pi_n(g^n) f^n\rightharpoonup g f$; since $g^n\rightharpoonup g$, by standard properties of Fourier projections $\Pi_n(g^n)\rightharpoonup g$ as well, thus $\Pi_n(g^n) f \rightharpoonup g f$.

To conclude, it remains to show that $\Pi_n(g^n) (f^n-f)\rightharpoonup 0$; this is the main point where the condition $\alpha>0$ enters crucially and the reason why we restrict ourselves to the noise $W^\alpha$ as given in Section \ref{sec:intro-noise}.
For any $\varphi \in C^\infty$, it holds $\Pi_n(g^n) \varphi\in L^2_t H^{\alpha}_x$ and $\|\Pi_n(g^n) \varphi \|_{L^2_t H^{\alpha}_x} \lesssim \|\Pi_n(g^n) \|_{L^2_t H^{\alpha}_x} \|\varphi \|_{C^{\alpha+\eps}_x} \leq M \|\varphi \|_{C^{\alpha+\eps}_x}$ as $g^n \in S^M(\cH^\alpha)$; therefore,
\begin{align*}
\bigg| \int_0^t \<\Pi_n(g^n_s)(f^n_s-f_s), \varphi\>\,\d s\bigg|
& \leq  \int_0^t |\<f^n_s-f_s, \Pi_n(g^n_s)\, \varphi\>|\,\d s \lesssim M\| \varphi\|_{C^{\alpha+\eps}_x} \| f^n-f\|_{C^0 H^{-\alpha}}
\end{align*}
and the last term converges to $0$, thus completing the proof of claim ({\rm ii}).

From these arguments we conclude that $f$ is a weak $L^\infty_{t,x}$ solution to $\partial_t f + (b+g)\cdot \nabla f = \Delta f$ with initial condition $f_0$; by Proposition \ref{prop:stability-strong-skeleton}, $f= \mathcal G^0( f_0, \textup{Int}(g))$, completing the proof.
\end{proof}

Now we can state and prove a preliminary LDP result.

\begin{proposition}\label{thm-bounded-data}
The solutions $\{f^n= \mathcal G^n(f_0, \sqrt{\eps_n} W^\alpha) \}_{n\geq 1, f_0\in \mathcal E^R_0}$ of \eqref{SLTE} satisfy the uniform Laplace principle with the rate function
\begin{equation}\label{thm-bounded-data.1}
I_{f_0}(f)= \inf_{\{ g\in L^2_t \cH^\alpha :\, f= \mathcal G^0 (f_0, \textup{Int}(g)) \}} \bigg\{\frac12 \int_0^T \|g_s \|_{H^\alpha}^2 \,\d s \bigg\},
\end{equation}
where $f_0\in \mathcal E^R_0$ and $f\in \mathcal E^{T,R}$.
\end{proposition}

\begin{proof}
By Lemma \ref{prop-cond-1}, Lemma \ref{prop-cond-2} and Theorem \ref{thm-LDP-abstract}, it suffices to show that for any fixed $f\in \mathcal{E}^{T,R}$, the map $f_0 \mapsto I_{f_0}(f)$ from $\mathcal{E}^R_0$ to $[0,\infty]$ is lower semicontinuous; but this has been shown in the proof of Lemma \ref{lem:Gamma-converg}.
\end{proof}

Now we will transfer the uniform Laplace principle from Proposition \ref{thm-bounded-data} to the case of $L^2$-initial data; the main tools allowing to do so are the linear structure of the equation and the $\Gamma$-convergence result from Lemma \ref{lem:Gamma-converg}.
Recall that here we take $\mathcal{E}_0= L^2$ and $\mathcal{E}=C^0_t H^{-s}$, both endowed with strong topologies. The exact value of $s>0$ is not important.

\begin{theorem}\label{thm-L2-data}
The solutions $\{f^n \}_n$ of \eqref{SLTE} with initial data $f_0\in L^2$ satisfy the Laplace principle on $C^0_t H^{-s}$, with rate function $I_{f_0}$, uniformly  on compact subsets of $L^2$.
In particular, they satisfy the LDP with the same rate function.
\end{theorem}

\begin{proof}
It is enough to verify \eqref{eq:uniform-laplace} for any uniformly continuous, bounded $G:\mathcal{E}\to \R$, so that it admits a modulus of continuity $\omega_G$; once this is done, the case of general $G$ follows by standard approximation arguments (alternatively, arguing as in the proof of Lemma \ref{prop-cond-2}, one can show that $f^n$ belong to $C^\gamma_t H^{-s+\eps}_x$ for some $\gamma,\eps>0$ and thus to a compact subset $\mathcal K_T$ of $C^0_t H^{-s}_x$, so that $G$ restricted to $\mathcal K_T$ is automatically uniformly continuous).

Let us make some preliminary observations. Firstly, for any fixed $n$, given any two solutions $f^{n},\, \tilde{f}^n$ associated to initial data $f_0,\,\tilde{f}_0$, by linearity of the transport equation and the pathwise estimate \eqref{eq:a-priori-transport}, we have the $\P$-a.s. bound $\| \tilde{f}^{n} - f^n\|_{L^\infty L^2} \leq C_b \|\tilde{f}_0 - f_0\|_{L^2}$ for $C_b=\| \nabla\cdot b\|_{L^1 L^\infty}$.

Secondly, for any (uniformly) continuous $G:\mathcal{E}\to\R$ and $f_0\in L^2_x$, let us set
\begin{equation*}
I_{f_0}(G):=\inf_{f\in \mathcal E} \{G(f)+I_{f_0}(f)\};
\end{equation*}
Lemma \ref{lem:Gamma-converg} and the two facts mentioned above imply that $I_{f^n_0}(G)\to I_{f_0}(G)$ as $n\to \infty$ whenever $f^n_0\to f_0$ in $L^2$.

Now let $\mathcal K$ be a compact subset of $L^2$ and fix $\delta>0$; let $\rho_\delta$ be a standard mollifier and set $f^\delta_0= \rho_\delta\ast f_0$.
From now on, we will use the notation $f^{n,\delta}$ to denote the solution to \eqref{SLTE} with initial data $f^\delta_0$, i.e. $f^{n,\delta}= \mathcal G^n(f^\delta_0, \sqrt{\eps_n}\, W^\alpha)$.
For fixed $\delta>0$, we set $\mathcal K^\delta:=\{ f^\delta_0: f_0\in \mathcal K\}$ which is a compact set in $L^\infty$.

By the $\P$-a.s. estimate
$|G(f^n)-G(f^{n,\delta})|
\leq \omega_G(\| f^n-f^{n,\delta}\|_{L^\infty L^2})
\leq \omega_G(C_b \| f_0-f^\delta_0\|_{L^2})$
and simple algebraic computations, it holds
\begin{align*}
\bigg| \eps_n \log \E\bigg[\exp \Big(-\frac{1}{\eps_n} G(f^{n})\Big)\bigg] - \eps_n \log \E\bigg[\exp \Big( -\frac{1}{\eps_n} G(f^{n,\delta}) \Big)\bigg]\bigg|
\leq \omega_G(C_b \| f_0-f_0^\delta\|_{L^2});
\end{align*}
as a consequence, we obtain
\begin{align*}
\sup_{f_0\in \mathcal K} \bigg| \eps_n \log \E\bigg[\exp \Big(-\frac{1}{\eps_n} G(f^{n})\Big)\bigg] + I_{f_0}(G)\bigg|
& \leq \sup_{f_0\in \mathcal K} \big\{ \omega_G(C_b \| f_0-f_0^\delta\|_{L^2})
+ \big|I_{f_0^\delta}(G)-I_{f_0}(G) \big| \big\} \\
& + \sup_{f_0\in \mathcal K} \bigg| \eps_n \log \E\bigg[\exp \Big(-\frac{1}{\eps_n} G(f^{n,\delta})\Big)\bigg] + I_{f_0^\delta}(G)\bigg|.
\end{align*}
For fixed $\delta>0$, since $\mathcal K^\delta$ is compact in $L^\infty$, we can apply the uniform Laplace principle coming from Proposition \ref{thm-bounded-data} to the family $\{f^{n,\delta}:f_0^\delta\in \mathcal K^\delta\}$; therefore
\begin{align*}
\limsup_{n\to\infty} &\sup_{f_0\in \mathcal K} \bigg| \eps_n \log \E\bigg[\exp \Big(-\frac{1}{\eps_n} G(f^{n})\Big)\bigg] + I_{f_0}(G)\bigg| \\
\leq & \sup_{f_0\in \mathcal K} \big\{ \omega_G(C_b \| f_0-f_0^\delta\|_{L^2}) + \big|I_{f_0^\delta}(G)-I_{f_0}(G) \big|\big\}.
\end{align*}
In order to conclude, it then suffices to show that
\begin{equation}\label{eq:goal-proof-uniform-L^2}
\lim_{\delta\to 0} \sup_{f_0\in \mathcal K} \big\{ \omega_G( C_b \| f_0-f_0^\delta\|_{L^2})
+ \big|I_{f_0^\delta}(G)-I_{f_0}(G) \big| \big\} =0.
\end{equation}
Let us show that $\sup_{f_0\in \mathcal K} \| f_0-f_0^\delta\|_{L^2}\to 0$ as $\delta\to 0$, which implies convergence of the first quantity in \eqref{eq:goal-proof-uniform-L^2}.
Suppose by contradiction this is not true, then there exist sequences $f_0^n\subset \mathcal K$, $\delta_n\to 0$ such that $\big\| f_0^{n,\delta_n}-f^n_0 \big\|_{L^2}\geq c>0$. Since $\mathcal K$ is compact, we can extract a subsequence (not relabelled for simplicity) such that $\|f_0^n- f_0\|_{L^2}\to 0$; by properties of mollifiers, $ \big\| f^{n,\delta_n}_0 - f^{\delta_n}_0 \big\|_{L^2} \leq \| f^n_0-f_0\|_{L^2}$ and $\| f^{\delta_n}_0-f_0\|_{L^2}\to 0$. But then it holds
\[
c \leq \big\| f^{n,\delta_n}_0 - f^{\delta_n}_0 \big\|_{L^2} + \big\| f^{\delta_n}_0 - f_0 \big\|_{L^2} + \| f_0 - f^n_0 \|_{L^2} \leq 2 \| f^n_0-f_0\|_{L^2} + \big\| f^{\delta_n}_0 - f_0 \big\|_{L^2} \to 0
\]
as $n\to\infty$, yielding a contradiction.
The verification that $\sup_{f_0\in \mathcal K} \big|I_{f_0^\delta} (G) - I_{f_0}(G) \big| \to 0$ as $\delta\to 0$ comes from a similar argument, based on the fact that $I_{f^n_0}(G)\to I_{f_0}(G)$ whenever $f^n_0\to f_0$ in $L^2$.
\end{proof}

\subsection{Stochastic 2D Euler equations}\label{subs-LDP-SEE}

In this section we establish a LDP for solutions to the vorticity form of stochastic 2D Euler equations with transport noise. Let us shortly recall the setting: we consider
\begin{equation}\label{eq:stoch-Euler}
\d \xi^n + (K\ast\xi^n)\cdot\nabla \xi^n\, \d t + \sqrt{\eps_n} \circ\d \Pi_n (W^{ \alpha})\cdot \nabla \xi^n=0, \quad \xi^n_0 = \xi_0,
\end{equation}
with equivalent It\^o form
\begin{equation}\label{eq:stoch-Euler-Ito}
\d \xi^n + (K\ast\xi^n)\cdot\nabla \xi^n\, \d t + \sqrt{\eps_n}\, \d \Pi_n (W^{ \alpha})\cdot \nabla \xi^n= \Delta \xi^n\,\d t.
\end{equation}
In the current section, we will restrict to the case $\xi_0 \in L^\infty$, in which equation \eqref{eq:stoch-Euler} is well posed with probabilistically strong solutions; thus we can define a solution map $\xi^n= \mathcal G^n(\xi_0, \sqrt{\eps_n}\, W^\alpha),\, n\geq 1$.
As in Section \ref{subsec-LDP-SLTE}, in order to establish a LDP for $\{\xi^n \}_n$, we need to check Conditions 1 and 2 in Hypothesis \ref{hyp:LDP};
we consider the same spaces $\mathcal E^R_0$ and $\mathcal E^{T,R}$ as defined in \eqref{space-E-T-R}.

We start by showing the well posedness of the skeleton equation
\begin{equation}\label{eq:skeleton-NSE}
\partial_t \xi + g\cdot \nabla \xi + (K\ast\xi)\cdot\nabla \xi = \Delta\xi.
\end{equation}
The following lemma is the analogue of Proposition \ref{prop:stability-strong-skeleton}.

\begin{lemma}\label{lem:uniq-skeleton-navier-stokes}
For any $\xi_0\in L^\infty$ and any $g\in L^2_t \cH$, there exists a unique weak solution $\xi\in L^\infty_{t,x}$ to \eqref{eq:skeleton-NSE}.
Moreover, given two solutions $\xi^i$ associated to data $\xi^i_0$ and $g^i\in L^2_t \cH$, it holds
\begin{equation}\label{eq:stability-skeleton-NS}
\|\xi^1- \xi^2\|_{C^0 L^2} + \|\xi^1- \xi^2\|_{L^2 H^1} \lesssim \|\xi^1_0- \xi^2_0 \|_{L^2} + \|g^1- g^2 \|_{L^2 L^2} \|\xi^2_0 \|_{L^\infty}.
\end{equation}
Finally, given sequences $\xi^n_0 \stackrel{\ast}{\rightharpoonup} \xi_0$ and $g^n \rightharpoonup g$, the corresponding solutions $\xi^n$ converge weakly-$\ast$ in $L^{\infty}_{t,x}$ to $\xi$, as well as strongly in $L^2_{t,x}\cap C^0_t H^{-\delta}_x$ for any $\delta>0$.
\end{lemma}

\begin{proof}
The existence of solutions with uniform $L^\infty$-bound follows again from the classical maximum principle and approximation arguments.
Arguing as in the proof of Proposition \ref{prop:stability-strong-skeleton}, thanks to the $L^\infty$-bound and the properties of the Biot-Savart kernel, it is easy to check that $(g+ K\ast\xi )\cdot \nabla \xi\in L^2_t H^{-1}_x$ and so that by maximal regularity $\xi\in L^2_t H^1_x$.

Now assume we are given two $L^\infty_{t,x}$-solutions to \eqref{eq:skeleton-NSE} and set $v=\xi^1-\xi^2$; we can apply the Lions-Magenes lemma and the divergence free property of $g$ to deduce that
\begin{align*}
 \frac{1}{2}\frac{\d}{\d t} \| v\|_{L^2}^2 + \| \nabla v\|_{L^2}^2
& = - \langle (K\ast v)\cdot \nabla \xi^2, v\rangle = \<\xi^2, (K\ast v)\cdot \nabla v \>  .
\end{align*}
Note that
\[|\<\xi^2, (K\ast v)\cdot \nabla v \>| \leq \|\xi^2 \|_{L^\infty} \|K\ast v \|_{L^2} \|\nabla v\|_{L^2}
\leq C \|\xi^2 \|_{L^\infty} \|v \|_{L^2} \|\nabla v\|_{L^2},\]
hence
\[ \frac{1}{2}\frac{\d}{\d t} \| v\|_{L^2}^2 + \| \nabla v\|_{L^2}^2
\leq \frac12 \|\nabla v \|_{L^2}^2 + C \| \xi^2 \|_{L^\infty}^2 \|v \|_{L^2}^2.\]
Together with the fact that $\xi^2\in L^\infty_{t,x}$ and Gronwall's lemma, this readily yields uniqueness; similar estimates, together with the bound $\| \xi^2\|_{L^\infty_{t,x}}\leq \| \xi^2_0\|_{L^\infty}$, then yields \eqref{eq:stability-skeleton-NS}.

The last assertion can be proved in the same way as that of  Proposition \ref{prop:stability-strong-skeleton}.
\end{proof}

Lemma \ref{lem:uniq-skeleton-navier-stokes} allows us to define the solution map $\xi:= \mathcal G^0(\xi_0, \textup{Int}(g))$ associated to \eqref{eq:skeleton-NSE} for $\xi_0\in L^\infty$ and $g\in L^2_t \cH^\alpha$.
With these preparations at hand, we see that Condition 1 in Hypothesis \ref{hyp:LDP} holds, similarly to Lemma \ref{prop-cond-1}.

\begin{lemma}\label{prop-S2EE-cond-1}
For any $M>0$ and compact subset $\mathcal K\subset \mathcal E^R_0$, the set
\begin{equation*}
\Gamma_{\mathcal K,M}:= \big\{ \mathcal G^0( \xi_0, \textup{Int}(g)): \xi_0\in \mathcal K,\, g\in S^M( \cH^\alpha) \big\}
\end{equation*}
is compact in the space $\mathcal E^{T,R}$.
\end{lemma}

Next, in order to verify Condition 2 in Hypothesis \ref{hyp:LDP}, we consider stochastic 2D Euler equations with some extra drift term $g\in L^2_t \cH^\alpha$:
\begin{equation}\label{eq:stoch-Euler-drift}
\d \xi^n + (K\ast\xi^n)\cdot\nabla \xi^n\, \d t + \circ\d \Pi_n \big( \sqrt{\eps_n}\, W^\alpha + \textup{Int}(g) \big)\cdot \nabla \xi^n =0, \quad \xi^n_0 = \xi_0\in L^\infty;
\end{equation}
the equivalent It\^o form is
\begin{equation*}
\d \xi^n + (K\ast\xi^n)\cdot\nabla \xi^n\, \d t + \d \Pi_n \big( \sqrt{\eps_n}\, W^\alpha + \textup{Int}(g) \big)\cdot \nabla \xi^n = \Delta \xi^n \,\d t, \quad \xi^n_0 = \xi_0\in L^\infty.
\end{equation*}
Due to the projection $\Pi_n$, the noise part has only finitely many Brownian motions, hence it clearly satisfies Condition 2.3 in \cite{BFM}; similarly, for fixed $n\geq 1$, $\d \Pi_n ( \textup{Int}(g) ) = \Pi_n g\,\d t$ is smooth in the space variable, thus it can be regarded as a smooth perturbation of the equation.
One can follow the approach in \cite{BFM}\,\footnote{This can be accomplished directly by carrying the same computations as in \cite{BFM}, or more elegantly by applying the flow transformation outlined in Section 7 therein to reduce the problem to the deterministic one; indeed for finite $n$, since $\Pi_n(g)$ and $\Pi_n(W^\alpha)$ are spatially smooth, the SDE $\d X_t = \Pi_n(g)(X_t) \d t + \circ \d \Pi_n(W^\alpha)(X_t)$ admits a smooth stochastic flow $\psi$.} to show that these equations enjoy pathwise unique, probabilistically strong solutions $\xi^n$ satisfying $\P$-a.s., $\|\xi^n_t \|_{L^\infty} \leq \|\xi_0 \|_{L^\infty}$ for all $t\in [0,T]$;
in particular, we can define the solution map
\[\xi^n= \mathcal G^n\big(\xi_0, \sqrt{\eps_n}\, W^\alpha+ \textup{Int}(g) \big).\]

Recall the families $S^M(\cH^\alpha)$ and $\mathcal P_2^M(\cH^\alpha)$ defined in \eqref{eq:defn-S^M}, \eqref{eq:defn-P^M_2}; the next result verifies Condition 2 in Hypothesis \ref{hyp:LDP}.

\begin{proposition}\label{prop-S2EE-cond-2}
Let $\{g^n \}_n \subset \mathcal P_2^M(\cH^\alpha)$ and $\{\xi^n_0 \}_n \subset \mathcal E^R_0$.
Assume that $\xi^n_0 \stackrel{\ast}{\rightharpoonup} \xi_0$ and $g^n$ converge in law, as $S^M(\cH^\alpha)$ valued random elements, to $g$ as $n\to \infty$. Then the convergence
\begin{equation*}
\mathcal G^n\big(\xi^n_0, \sqrt{\eps_n}\, W^\alpha+ \textup{Int}(g^n) \big) \to \mathcal G^0\big(\xi_0, \textup{Int}(g) \big)
\end{equation*}
holds in law, in the topology of $\mathcal E^{T,R}$.
\end{proposition}

\begin{proof}
For given sequences $\{g^n \}_n \subset \mathcal P_2^M(\cH^\alpha)$ and $\{\xi^n_0 \}_n \subset \mathcal E^R_0$, we denote the corresponding solutions by
\[\xi^n= \mathcal G^n\big(\xi^n_0, \sqrt{\eps_n}\, W^\alpha+ \textup{Int}(g^n) \big), \quad n\geq 1. \]
For any $\delta>0$, thanks to the uniform $L^\infty$-bound (so that we also have weak-$\ast$ convergence in $L^\infty_{t,x}$) and the estimate on the martingale part coming from Lemma \ref{lem-martingale}, one can check as usual the tightness of the sequence $\{\xi^n \}_n$ in $C^0_t H^{-\delta}_x$.

Using the weak convergence of $g^n$ to $g$, it is then not difficult to check that (up to extracting a non-relabelled subsequence) $\xi^n$ converge in law to a weak $L^\infty$-solution $\xi$ of the skeleton equation \eqref{eq:skeleton-NSE}.
More precisely, the convergence of $\Pi_n(g^n)\cdot \nabla \xi^n$ follows the same proof of Lemma \ref{prop-cond-2}, while the convergence of the nonlinear term $(K\ast \xi^n) \cdot \nabla \xi^n$ can be treated in the same way as the proof of \cite[Theorem 2.2]{FGL21a}; the martingale terms $M^n$ vanish as $n\to \infty$ due to estimate \eqref{eq:martingale-estim}.

Lemma \ref{lem:uniq-skeleton-navier-stokes} finally allows to deduce that $\xi= \mathcal G^0\big(\xi_0, \textup{Int}(g) \big)$.
\end{proof}

Now we can state the main result of this section.

\begin{theorem}\label{thm-LDP-S2EE}
The solutions $\{\xi^n =\mathcal G^n(\xi_0, \sqrt{\eps_n}\, W^\alpha) \}_{n\geq 1,\, \xi_0\in \mathcal E^R_0} $ of \eqref{eq:stoch-Euler} satisfy the uniform Laplace principle with the rate function
\begin{equation}\label{thm-LDP-S2EE.1}
I_{\xi_0}(\xi)= \inf_{\{ g\in L^2_t \cH^\alpha:\, \xi= \mathcal G^0 (\xi_0, \textup{Int}(g)) \}} \bigg\{\frac12 \int_0^T \|g_s \|_{H^\alpha}^2 \,\d s \bigg\},
\end{equation}
where $\xi_0\in \mathcal E^R_0$ and $\xi\in \mathcal E^{T,R}$.
\end{theorem}

\begin{proof}
By Theorem \ref{thm-LDP-abstract} and Lemmas \ref{prop-S2EE-cond-1} and \ref{prop-S2EE-cond-2}, it suffices to show the lower semicontinuity of $\xi_0\mapsto I_{\xi_0}(\xi)$, which follows the same line of arguments as those of Lemma \ref{lem:Gamma-converg}.
\end{proof}

\subsection{Further remarks and future directions}\label{sec:conclusive-LDP}

We present here some consequences of the results presented above, together with possible problems to explore in the future. We start by some simple lower bounds on the rate functions, which can be deduced by the stability estimates presented in Section \ref{subsec-skeleton-eq}.

\begin{lemma}\label{lem:lower-bound-rate-fct}
Let $f_0\in L^2\setminus\{0\}$, $I_{f_0}$ be defined as in \eqref{rate-funct-1} and $\bar{f}=\mathcal{G}^0(f_0,0)$, i.e. the solution to \eqref{eq:skeleton} with $g=0$; then for any $\delta> d/2$ it holds
\begin{equation}\label{eq:lower-bound-rate-fct.1}
I_{f_0}(f) \gtrsim_{\delta,d} \frac{\| f-\bar{f}\|^2_{C^0 H^{-\delta}}}{\| f_0\|_{L^2}^2} \quad \forall\, f\in L^\infty_t L^2_x.
\end{equation}
If additionally $f_0\in L^\infty$, then
\begin{equation}\label{eq:lower-bound-rate-fct.2}
I_{f_0}(f) \gtrsim_{d} \frac{\| f-\bar{f}\|^2_{C^0 L^2}}{\| f_0\|_{L^\infty}^2} \quad \forall\, f\in L^\infty_t L^2_x.
\end{equation}
\end{lemma}

\begin{proof}
Given $f_0$ as above, suppose $f=\mathcal{G}^0(f_0,\textup{Int}(g))$ for some $g\in L^2_t \mathcal{H}^\alpha$; then clearly $f\in C^0_t H^{-\delta}$.
Moreover by Proposition \ref{prop:stability-skeleton-L2} (for the choice $c^i=0$, $h^1=b+g$ and $h^2=b$) and Sobolev embeddings it holds
\begin{align*}
\| f-\bar{f}\|_{C^0 H^{-\delta}} \lesssim \| f-\bar{f}\|_{L^\infty L^1} \leq \| f_0\|_{L^2} \| g\|_{L^2_{t,x}} \leq \| f_0\|_{L^2} \| g\|_{L^2 H^\alpha};
\end{align*}
inverting the inequalities we obtain
\begin{equation*}
\frac{1}{2} \| g\|_{L^2 H^\alpha}^2 \gtrsim \frac{\| f-\bar{f}\|_{C^0 H^{-\delta}}^2}{\| f_0\|_{L^2}^2}
\end{equation*}
from which \eqref{eq:lower-bound-rate-fct.1} follows by taking the infimum over $g$.
The proof of \eqref{eq:lower-bound-rate-fct.2} is similar, only observe that in this case if $f=\mathcal{G}^0(f_0,\textup{Int}(g))$, then necessarily $f\in C^0_t L^2\cap L^\infty_{t,x}$ and estimate \eqref{eq:stability-strong-skeleton} is available.
\end{proof}

\begin{remark}
For simplicity we only considered the cases $f_0\in L^2$ and $f_0\in L^\infty$, but several variants of the above estimates are available for $p\in (2,\infty)$.
Indeed, one can derive suitable stability estimates in the style of \eqref{eq:stability-strong-skeleton}, \eqref{eq:stability-skeleton-L^2}, using the information that the initial $L^p$ integrability is propagated over time;
in turn this provides estimates of the form \eqref{eq:lower-bound-rate-fct.1} where the parameter $\delta$ belongs to an interval determined by Sobolev embeddings.
\end{remark}

\begin{corollary}\label{cor:asymptotic-rate-1}
Let $f_0\in L^2$, $f^n= \mathcal{G}^n(f_0,\sqrt{\eps_n} W^\alpha)$ and $\bar{f}= \mathcal{G}^0(f_0,0)$ for the maps $\mathcal{G}^n$, $\mathcal{G}^0$ as defined in Section \ref{subs-LDP-SLTE-bounded}.  Then for any $\delta> d/2$ and any $R>0$ there exists a constant $C>0$ such that
\begin{equation}\label{eq:large-deviation-estim}
\P\big( \| f^n-\bar{f}\|_{C^0 H^{-\delta}} \geq R\big) \lesssim \exp\bigg(-C\, \frac{R^2}{\eps_n \| f_0\|_{L^2}^2}\bigg) .
\end{equation}
If additionally $f_0\in L^\infty$, then an estimate of the form \eqref{eq:large-deviation-estim} holds for any $\delta>0$, up to replacing $\| f_0\|_{L^2}$ by $\| f_0\|_{L^\infty}$.
\end{corollary}

\begin{proof}
By the definition of LDP and Lemma \ref{lem:lower-bound-rate-fct}, it holds
\begin{align*}
\limsup_{n\to\infty} \eps_n \log \P\big( \| f^n-\bar{f}\|_{C^0 H^{-\delta}} \geq R\big)
& \leq -\inf\big\{ I_{f_0}(f): f\in L^\infty_t L^2_x \text{ s.t. } \| f-\bar{f}\|_{C^0 H^{-\delta}}\geq R\big\}\\
& \leq - \frac{R^2}{\| f_0\|_{L^2}^2}
\end{align*}
which gives \eqref{eq:large-deviation-estim}; similarly in the case $f_0\in L^\infty$.
\end{proof}

The same technique can be applied in the case of $2$D Euler equations; to avoid unnecessary repetitions, we give the statement without proof.

\begin{corollary}\label{cor:asymptotic-rate-2}
Let $\xi_0\in L^\infty\setminus\{0\}$, $I_{\xi_0}$ be given by \eqref{thm-LDP-S2EE.1} and $\bar{\xi}=\mathcal{G}^0(\xi_0,0)$, where $\cG^0$ is defined as in Section \ref{subs-LDP-SEE}; then it holds
\begin{equation}\label{eq:lower-bound-rate-fct.3}
I_{\xi_0}(\xi) \gtrsim\frac{\| \xi-\bar{\xi}\|^2_{C^0 L^2}}{\| \xi_0\|_{L^\infty}^2} \quad \forall\, \xi\in L^\infty_t L^2_x.
\end{equation}
As a consequence, setting $\xi^n =\mathcal{G}^n(\xi_0,\sqrt{\eps_n} W^\alpha)$, for any $\delta> 0$ and any $R>0$ there exists a constant $C>0$ such that
\begin{equation}\label{eq:large-deviation-estim.2}
\P\big( \| \xi^n-\bar{\xi}\|_{C^0 H^{-\delta}} \geq R\big) \lesssim \exp\bigg(-C\, \frac{R^2}{\eps_n \| f_0\|_{L^\infty}^2}\bigg).
\end{equation}
\end{corollary}

The interest in results like Corollaries \ref{cor:asymptotic-rate-1} and \ref{cor:asymptotic-rate-2} comes from applications in the suppression of blow-up phenomena presented in \cite{FGL21b}.
To discuss this, let us turn to an abstract setting similar to that from the introduction; namely let us consider an SPDE of the form
\begin{equation}\label{eq:abstract-SPDE-LDP}
\d u + \sqrt{\nu\eps_n}\, \circ \d W^{n,\alpha}\cdot\nabla u = [\Delta u + F(u)] \,\d t.
\end{equation}
The equation is now truly parabolic and due to the presence of an additional parameter $\nu>0$, it has corresponding It\^o form
\begin{equation}\label{eq:abstract-SPDE-LDP.2}
\d u + \sqrt{\nu\eps_n}\, \d W^{n,\alpha}\cdot\nabla u = [(1+\nu)\Delta u + F(u)] \,\d t.
\end{equation}
For a large class of nonlinearities $F$ (e.g. those satisfying Hypothesis 1.2 from \cite{FGL21b}), equation \eqref{eq:abstract-SPDE-LDP} admits pathwise existence and strong uniqueness of maximal local solutions, which might however blow up in finite time, just like the deterministic counterpart of \eqref{eq:abstract-SPDE-LDP}.
However, it was shown in \cite{FGL21b} that if the parameters $\nu,\, n$ are chosen correctly, then the probability of blow-up for \eqref{eq:abstract-SPDE-LDP} occurring before a fixed time $T>0$ can be made arbitrarily small.

Establishing an LDP for the small noise limit of \eqref{eq:abstract-SPDE-LDP.2}, combined with estimates like those in  Corollaries \ref{cor:asymptotic-rate-1} and \ref{cor:asymptotic-rate-2}, would provide a more precise information on the asymptotics of this small blow-up probability, which is expected to decay exponentially as a function of $n$.
Indeed, the proof provided in \cite{FGL21b} roughly speaking follows three main steps:
\begin{itemize}
\item[i)] Choose $\nu>0$ sufficiently large so that the limit $\bar{u}$ of $u^n$, solving the equation with enhanced viscosity
\[
\partial_t \bar u = (1+\nu)\Delta \bar u + F(\bar u)
\]
does not blow up before time $T$ and it satisfies $\| \bar{u}\|_{H^{-\delta}} \leq R$ for some appropriately chosen parameters $\delta, R>0$.
\item[ii)] Instead of considering \eqref{eq:abstract-SPDE-LDP}, look at the equation with cutoff
\begin{equation}\label{eq:abstract-SPDE-LDP-cutoff}
\d \tilde u^n + \sqrt{\nu\eps_n}\, \circ \d W^{n,\alpha}\cdot\nabla \tilde u^n = [\Delta \tilde u^n + \rho (\tilde u^n) F(\tilde u^n)] \d t;.
\end{equation}
the cutoff function $\rho(\tilde{u}^n)=\rho (\| \tilde{u}^n\|_{H^{-\delta}})$ is devised in such a way that: a) global existence of solutions to \eqref{eq:abstract-SPDE-LDP-cutoff} always holds; b) as long as $\| \tilde{u}^n\|_{H^{-\delta}} \leq 2R$, $\rho(\tilde{u}^n)=1$ so that $\tilde{u}^n$ is also a solution to the equation \eqref{eq:abstract-SPDE-LDP} without cutoff.
\item[iii)] Show that $\tilde{u}^n\to \bar{u}$ as $n\to\infty$ in $C^0_t H^{-\delta}$; in particular, if $\| \tilde{u}^n-\bar{u}\|_{C^0 H^{-\delta}} \leq R$, then $\tilde{u}^n = u^n$, so that the probability of blow-up happening can be bounded by
\begin{equation}\label{eq:estim-blow-up}
\P\big( \|\tilde u^n_t \|_{H^{-\delta}} >2R \mbox{ for some } t\leq T \big) \leq \P\big(\| \tilde{u}^n-\bar{u}\|_{C^0 H^{-\delta}} >R \big).
\end{equation}
\end{itemize}
The results from \cite{FGL21b} only establish that the right hand side of \eqref{eq:estim-blow-up} converges to $0$ as $n\to\infty$; see \cite[Theorem 1.5]{FGL21c} for more quantitative, polynomial type of estimates in the specific case of the $2$D Keller-Segel system.

As the quantity on the right hand side of \eqref{eq:estim-blow-up} is exactly of the same form as those appearing in \eqref{eq:large-deviation-estim}, \eqref{eq:large-deviation-estim.2}, in order to establish its exponential decay it would be enough to show that $\tilde{u}^n$ satisfy an LDP with a rate function $I$ satisfying suitable lower bounds.

However, we warn the reader that this task might be far from obvious and different PDEs (i.e. different choices of the nonlinearity $F$) might require case-by-case treatment.
Recall that one of the key ingredients in establishing the LDP is the study of the skeleton equation
\begin{equation*}
\partial_t v + g\cdot\nabla v = \Delta v + \rho(v) F(v)
\end{equation*}
where the function $g$ is divergence free and only enjoys regularity $L^2_t \mathcal{H}^\alpha$, so that the Prodi-Serrin condition does not hold.
In particular, resorting to the mild formulation of the problem is not sufficient to solve it, while the presence of the nonlinearity $F$ might complicate the analysis presented in Section \ref{subsec-skeleton-eq}.
A crucial tool in establishing the well-posedness of the nonlinear skeleton equation \eqref{eq:skeleton-NSE} was its transport structure, allowing to derive a priori estimates in $L^\infty_t L^p_x$, which are then crucial in order for the proof of Lemma \ref{lem:uniq-skeleton-navier-stokes} to work.
Different nonlinearities $F$ might not allow to derive similar estimates and thus the techniques presented here might not necessarily work.

\section{Central Limit Theorems}

We provide here the proofs of the CLTs, by treating first the case of stochastic transport equations in Section \ref{subs-CLT-SLTE} and then stochastic $2$D Euler equations in Section \ref{subs-CLT-SEE}.
In both cases, we provide explicit rates of strong convergence, see Theorems \ref{thm:CLT-SLTE-general} and \ref{thm:CLT-SEE-1} respectively.
Finally we discuss in Section \ref{sec:conclusive-CLT} possible generalizations of these results.

\subsection{Stochastic transport equation}\label{subs-CLT-SLTE}

We recall the stochastic transport equation \eqref{SLTE-Ito} for the reader's convenience:
\begin{equation}\label{SLTE-Ito-sec3}
\d f^n + b\cdot \nabla f^n\, \d t + \d \Pi_n \big(\sqrt{\eps_n}\, W^\alpha \big) \cdot \nabla f^n = \Delta f^n \,\d t, \quad f^n_0= f_0\in L^2.
\end{equation}
Throughout this section, we will always impose the following assumption on the drift $b$.

\begin{assumption}\label{ass:drift-CLT}
The drift $b$ satisfies $\nabla\cdot b\in L^1_t L^\infty_x$ and there exist $(p,q)\in (2,\infty]$ such that
\begin{equation}
b,\, \nabla\cdot b \in L^q_t L^p_x, \quad \gamma:=\frac{2}{q}+\frac{d}{p}<1.
\end{equation}
\end{assumption}

In this case, we always have strong existence and uniqueness of solutions $f^n$ to \eqref{SLTE-Ito-sec3}. Let $f$ be the solution to
\begin{equation}\label{eq:limit-advect-diffus}
\partial_t f + b\cdot\nabla f = \Delta f, \quad f|_{t=0} =f_0.
\end{equation}
Then $X^n:= (f^n-f)/\sqrt{\eps_n}$ solves
\begin{equation}\label{eq:before-limit}
\d X^n + b\cdot\nabla X^n\, \d t + \d W^{n,\alpha} \cdot \nabla f^n = \Delta X^n\, \d t.
\end{equation}
It is natural to expect that the family $\{X^n \}_n$ converges as $n\to\infty$ to the solution $X$ of
\begin{equation}\label{eq:candidate-limit}
\d X + b\cdot\nabla X \, \d t + \d W^\alpha\cdot \nabla f = \Delta X\, \d t, \quad X\vert_{t=0}=0.
\end{equation}
At this stage, equation \eqref{eq:candidate-limit} is only heuristical, as we cannot expect to give it meaning in a classical sense; instead, we will adopt the following definition.

\begin{definition}\label{defn:weak-sol}
We say that a process $X$ with trajectories in $C^0_t H^{-\beta}$ for some $\beta>0$ is a solution to \eqref{eq:candidate-limit} if for any $\tau\in [0,T]$ and $\varphi\in C^\infty$ it holds
\begin{equation}\label{eq:defn-weak-sol}
\langle X_\tau, \varphi\rangle = \int_0^\tau \langle f_r \nabla g_r, \d W^\alpha_r\rangle\quad \P\text{-a.s.},
\end{equation}
where $g\in L^2_t C^1_x$ is any solution to the backward equation
\begin{equation}\label{eq:dual-equation}
\partial_t g + \nabla\cdot (b\, g) + \Delta g=0
\quad \text{for }t\in [0,\tau],
\quad g\vert_{t=\tau} = \varphi .
\end{equation}
\end{definition}

Observe that, since we are imposing $g\in L^2_t C^1_x$, the stochastic integral appearing in \eqref{eq:defn-weak-sol} is meaningful, since
\begin{align*}
\E\bigg[ \Big| \int_0^\tau \langle f_r \nabla g_r, \d W^\alpha_r\rangle \Big|^2\bigg]
= \int_0^\tau \| f_r \nabla g_r\|_{H^{-\alpha}}^2\, \d r
\lesssim \int_0^\tau \| f_r \nabla g_r\|_{L^2}^2\, \d r
\leq \| f\|_{L^\infty L^2}^2 \| g\|_{L^2 C^1}^2.
\end{align*}
Let us stress however that the existence of such a $g$ is part of the above definition. The well-posedness of \eqref{eq:dual-equation} follows from  Theorem \ref{thm:parabolic-pde} in Appendix \ref{appendix}, where more general equations are treated (up to time reversal); we summarize here some results which will be frequently used in this section.

\begin{lemma}\label{lem:dual-equation}
Under Assumption \ref{ass:drift-CLT}, for any $\varphi\in L^\infty$, equation \eqref{eq:dual-equation} has a unique solution of the form $g_t = P_{\tau-t} \varphi + v_t$, where $v\in C^0_t C^0_x$ satisfying $v_\tau=0$ and $\sup_{t\leq \tau} (\tau-t)^{\frac\gamma2} \|v_t \|_{C^1} \lesssim_b \|\varphi \|_{L^\infty}$.
Moreover, we can define a two-parameter solution operator $S_{\tau, t} \varphi :=g_t ,\, 0\leq t\leq \tau \leq T$, such that
\begin{itemize}
\item[\rm i)] $S_{\tau,t}= P_{\tau-t} + R_{\tau,t}$, where $R_{\tau,t}$ is a linear operator satisfying $$\|R_{\tau,t}\varphi \|_{C^\delta} \lesssim_\delta (\tau-t)^{(1-\gamma-\delta)/2} \|\varphi \|_{L^\infty} \quad \mbox{for any } \delta\in [0,2-\gamma) ;$$
\item[\rm ii)] for any $\eps>0$, $\|S_{\tau,\cdot}\varphi \|_{L^2_t C^1} \lesssim_\eps \|\varphi \|_{C^\eps}$;
\item[\rm iii)] $S$ enjoys the semigroup property: $S_{\tau,s}= S_{t,s}\circ S_{\tau,t}$ for any $0\leq s<t<\tau \leq T$.
\end{itemize}
\end{lemma}

Now we can establish the well-posedness of \eqref{eq:candidate-limit}.

\begin{proposition}\label{prop:wellposedness-limit-eq}
Let $b$ satisfy Assumption \ref{ass:drift-CLT}, then for any $f_0\in L^2$ there exists a unique solution $X$ to \eqref{eq:candidate-limit}, in the sense of Definition \ref{defn:weak-sol}.
Moreover $X$ is a Gaussian process; up to modification, for any $\delta\in (0,1-\gamma)$ and $\eps>0$, it has trajectories  in $C^{\delta/2-\eps}_t H^{-d/2-\delta-\eps}_x$.
\end{proposition}

\begin{proof}
We start with showing uniqueness. Assume there exist two processes $X^1$ and $ X^2$, defined on the same stochastic basis $(\Omega,\P,\mathcal{F}_t,W^\alpha)$.
By Lemma \ref{lem:dual-equation}, for any $\tau\in [0,T]$ and any $\varphi\in C^\infty$, there exists a unique solution $g$ to \eqref{eq:dual-equation} satisfying $g\in L^2_t C^1_x$.
Similarly, under Assumption \ref{ass:drift-CLT}, equation \eqref{eq:limit-advect-diffus} has a unique solution $f$ (see \cite[Lemma 1]{Galeati}), which is the limit of $\{f^n \}_n$.
Therefore it must hold $\langle X^1_\tau -X^2_\tau, \varphi \rangle = 0$ $\P$-a.s. for any fixed $\tau\in [0,T]$ and any $\varphi\in C^\infty$; standard arguments based on countable quantifiers imply that $X^1\equiv X^2$.

Now we turn to the existence part. Under Assumption \ref{ass:drift-CLT}, by Lemma \ref{lem:dual-equation}, uniqueness of regular solutions holds for \eqref{eq:dual-equation}. Specifically, if we fix $\eps>0$, then for any $\varphi\in C^\eps$ and any $t\in [0,T]$, the solution $g_\cdot = S_{t,\cdot}\varphi \in C^0_{s,x} \cap L^2_s C^1_x$, and the identity \eqref{eq:defn-weak-sol} can be rewritten as
\begin{equation}\label{eq:weak-sol-explicit}
\langle X_t, \varphi \rangle =\int_0^t \langle f_r\nabla S_{t,r}\varphi, \d W^\alpha_r\rangle.
\end{equation}
We can actually use expression \eqref{eq:weak-sol-explicit} to \textit{define} a Gaussian random field $X$, which will then automatically be a solution to \eqref{eq:candidate-limit}, in the sense of Definition \ref{defn:weak-sol}, adapted to the filtration generated by $W^\alpha$.

Below we discuss the space-time regularity of $X$. In order to show the claim, it suffices to check that for any such $\delta$ and any $\eps>0$ it holds
\begin{equation}\label{eq:space-time-regularity}
\E\big[ |\langle X_t -X_s, e_k\rangle|^2\big]
\lesssim |t-s|^{\delta} |k|^{2(\delta+\eps)}
\quad \forall\ 0\leq s<t\leq T,\, k\in \mathbb{Z}^d_0.
\end{equation}
Indeed, by hypercontractivity of Gaussian variables, it then holds
\begin{equation*}
\E\big[ |\langle X_t -X_s, e_k\rangle|^p\big]^{1/p} \lesssim_p |t-s|^{\delta/2} |k|^{\delta+\eps} \quad\forall\, p\in [1,\infty)
\end{equation*}
which implies the conclusion by standard arguments, based on Kolmogorov continuity type theorems (and possibly a relabelling of the parameter $\eps$).
By formula \eqref{eq:weak-sol-explicit}, it holds
\begin{align*}
\E\big[ |\langle X_t-X_s,e_k\rangle|^2 \big]
& = \E \bigg[\Big| \int_0^t \langle f_r (\nabla S_{t,r} e_k - \nabla S_{s,r} e_k \mathbbm{1}_{r<s}), \d W^\alpha_r \rangle \Big|^2 \bigg]\\
& = \int_0^t \big\| f_r (\nabla S_{t,r} e_k - \nabla S_{s,r} e_k \mathbbm{1}_{r<s}) \big\|_{H^{-\alpha}}^2\,\d r\\
& \lesssim \int_0^t \| f_r\|_{L^2}^2 \| \nabla S_{t,r} e_k - \nabla S_{s,r} e_k \mathbbm{1}_{r<s}\|_{L^\infty}^2\,\d r\\
& \lesssim \| f_0\|_{L^2}^2 \bigg[ \int_s^t \| S_{t,r} e_k\|_{C^1}^2 \,\d r + \int_0^s \| (S_{t,r}-S_{s,r}) e_k\|_{C^1}^2\, \d r\bigg].
\end{align*}

Let us denote the two integrals appearing in the last line by $J^1,\, J^2$. By Lemma \ref{lem:dual-equation},
\begin{align*}
J^1
& \lesssim \int_s^t \| P_{t-r} e_k \|^2_{C^1}\, \d r +  \int_s^t \| R_{t,r} e_k \|^2_{C^1}\, \d r\\
& \lesssim \int_s^t e^{-8\pi^2|k|^2 (t-r)} |k|^2\, \d r + \int_s^t |t-r|^{-\gamma}\, \d r\\
& \sim 1-e^{-8\pi^2|k|^2 (t-s)} + |t-s|^{1-\gamma}
\lesssim |t-s|^\delta |k|^{2\delta},
\end{align*}
where in the last step we have used $1-e^{-a} \leq a^\delta$ for any $a\geq 0$  and $\delta<1-\gamma<1$.
On the other hand, for any $\eps>0$ such that $\delta+\eps<1-\gamma$, it holds
\begin{align*}
J^2
& = \int_0^s \| S_{s,r} (S_{t,s}-I)e_k\|_{C^1}^2 \, \d r
\lesssim_\eps \| (S_{t,s}-I)e_k\|_{C^\eps}^2\\
& \lesssim \| (P_{t-s}-I) e_k\|_{C^\eps}^2 + \| R_{t,s} e_k\|_{C^\eps}^2
\lesssim |k|^{2\eps}\, \big|1-e^{-4\pi^2|k|^2 (t-s)} \big|^2 + |t-s|^{1-\gamma-\eps}\\
& \lesssim  |k|^{2 (\delta+\eps)}\, |t-s|^\delta + |t-s|^{\delta}
\lesssim |t-s|^\delta |k|^{2(\delta+\eps)}.
\end{align*}
Combining the above bounds yields the desired claim \eqref{eq:space-time-regularity} and thus the conclusion.
\end{proof}

\begin{remark}
Expression \eqref{eq:weak-sol-explicit} also allows to derive a somewhat explicit formula for the covariance of $X$: for any $s,t\in [0,T]$ and $\varphi,\psi\in C^\infty$ it holds
\begin{align*}
\E[ \langle X_t, \varphi \rangle \langle X_s, \psi \rangle]
= \int_0^{t\wedge s} \big\langle Q^{\alpha}(f_r \nabla S_{t,r} \varphi ), f_r \nabla S_{s,r}\psi \big\rangle\, \d r,
\end{align*}
where $Q^\alpha= \Pi (-\Delta)^{-\alpha} \Pi$ denotes the covariance operator associated to $W^\alpha$.
\end{remark}

Before proving the main result of this subsection, we need an explicit estimate on the rate of convergence of $f^n$ to $f$. A partially similar result can be found in \cite[Theorem 1.7]{FGL21c} for the special case $b\equiv 0$.

\begin{proposition}\label{prop:convergence-rate}
Let $b$ satisfy Assumption \ref{ass:drift-CLT}; for any $\delta\in (0,d/2]$ and any $\eps>0$, it holds
\begin{equation}
\sup_{t\in [0,T]} \E\big[ \|f^n_t- f_t\|_{H^{-\delta}}^2 \big]
\lesssim n^{-2\delta \big(1-\frac{2\alpha}{d}\big) +\eps}\, \|f_0 \|_{L^2}^2.
\end{equation}
\end{proposition}

\begin{proof}
The idea of proof is similar to that of Proposition \ref{prop:wellposedness-limit-eq}.
It suffices to show that, for any $\eps>0$, it holds
\[
\E\big[ \| f^n_t - f_t\|_{H^{-d/2-\eps}}^2 \big] \lesssim n^{2\alpha-d} \| f_0\|_{L^2}^2;
\]
indeed, the general case follows, up to possibly redefining $\eps$, from the interpolation inequality $\| \cdot\|_{H^{-\theta s}} \lesssim \| \cdot \|_{H^{-s}}^{\theta} \,\| \cdot\|_{L^2}^{1-\theta}$ and the $\P$-a.s. bound $\| f^n_t-f_t\|_{L^2}\lesssim \| f_0\|_{L^2}$.

Since $f^n$ is a weak solution to \eqref{SLTE-Ito-sec3}, for any $t\in [0,T]$ and any  smooth $\psi:[0,t]\times \T^d\to \R$, it holds
\[
\langle f^n_t, \psi_t\rangle - \langle f_0, \psi_0\rangle
= \int_0^t \big\langle f^n_r, \partial_r \psi_r + \nabla\cdot(b_r \psi_r) + \Delta \psi_r \big\rangle\, \d r + \sqrt{\eps_n} \int_0^t \langle f^n_r \nabla \psi_r, \d W^{n,\alpha}_r \rangle .
\]
Now given $\varphi\in C^\infty(\T^d)$, consider the solution $\{S_{t,r} \varphi\}_{r\in [0,t]}$  to the backward equation \eqref{eq:dual-equation}; arguing by smooth approximations, it is not hard to check that $\P$-a.s. it holds
\begin{align*}
\<f^n_t, \varphi\>- \<f_0, S_{t,0}\varphi\> = \sqrt{\eps_n} \int_0^t \<f^n_r \nabla S_{t,r}\varphi , \d W^{n,\alpha}_r \>.
\end{align*}
Similarly, the limit $f$ satisfies $\<f_t, \varphi\> = \<f_0, S_{t,0}\varphi\>$; therefore, taking $\varphi=e_k$, we get
\begin{align*}
\E\big[ |\<f^n_t-f_t, e_k\>|^2 \big]
= \eps_n\, \E\bigg[ \Big| \int_0^t \<f^n_r \nabla S_{t,r} e_k , \d W^{n,\alpha}_r \> \Big|^2 \bigg]
= \eps_n\,\E \int_0^t \big\| (Q_n^\alpha)^{1/2} (f^n_r \nabla S_{t,r} e_k) \big\|_{L^2}^2\,\d r,
\end{align*}
where $Q_n^\alpha$ is the covariance operator of $W^{n,\alpha}$. Therefore,
\begin{align*}
\E\big[ |\<f^n_t -f_t, e_k\>|^2 \big]
& \lesssim \eps_n\,\E \int_0^t \big\| f^n_r \nabla S_{t,r} e_k \big\|_{L^2}^2\,\d r \leq \eps_n\, \E \big[ \|f^n \|_{L^\infty L^2}^2\big] \int_0^t \|S_{t,r} e_k \|_{C^1}^2\,\d r .
\end{align*}
By Lemma \ref{lem:dual-equation},
\begin{align*}
\int_0^t \|S_{t,r} e_k \|_{C^1}^2\,\d r & \lesssim  \int_0^t \big( \| P_{t-r} e_k\|_{C^1}^2 + \|R_{t,r} e_k\|_{C^1}^2 \big)\, \d r \\
& \lesssim \int_0^t \big(e^{-8\pi^2|k|^2 (t-r)} |k|^2 + |t-r|^{-\gamma} \big)\, \d r
\lesssim 1,
\end{align*}
implying $\E\big[ |\<f^n_t -f_t, e_k\>|^2 \big] \lesssim \eps_n\, \| f_0\|_{L^2}^2$. Since $\eps_n\sim n^{2\alpha-d}$, for any $\eps>0$ it holds
\begin{align*}
\E\big[ \| f^n_t -f_t\|_{H^{-d/2-\eps}}^2 \big]
= \sum_k |k|^{-d-2\eps}\, \E\big[ |\<f^n_t -f_t, e_k\>|^2 \big]
\lesssim n^{2\alpha-d}\, \| f_0\|_{L^2}^2 \sum_k |k|^{-d-2\eps}
\end{align*}
which yields the conclusion.
\end{proof}

\begin{theorem}\label{thm:CLT-SLTE-general}
Let $b$ satisfy Assumption \ref{ass:drift-CLT}, $X^n$ and $X$ solutions respectively to \eqref{eq:before-limit}, \eqref{eq:candidate-limit} for some $f_0\in L^2$.
Then for any $0<\delta<\alpha \wedge(1-\gamma)$ and any $\eps>0$ it holds
\begin{equation}\label{eq:CLT-SLTE-general}
\sup_{t\in [0,T]} \E\big[ \| X^n_t - X_t\|_{H^{-d/2-\delta-\eps}}^2 \big]\lesssim n^{-2\delta \big( 1- \frac{2\alpha}{d}\big) + \eps}\, \|f_0 \|_{L^2}^2.
\end{equation}
\end{theorem}

\begin{proof}
Similarly to the case of $X$, we may test $X^n$ against the solutions of \eqref{eq:dual-equation}: for any $\varphi\in C^\infty$ and $t\in [0,T]$, one has
  $$\langle X^n_t, \varphi \rangle = \int_0^t \langle f^n_r\nabla S_{t,r} \varphi, \d W^{n,\alpha}_r \rangle. $$
Therefore, for any $l\in \Z^d_0$ it holds
\begin{align*}
\langle X_t-X^n_t, e_l \rangle
& = \int_0^t \langle f_r\nabla S_{t,r} e_l, \d W^\alpha_r \rangle - \int_0^t \langle f^n_r\nabla S_{t,r} e_l, \d W^{n,\alpha}_r \rangle\\
& = \sum_{|k|>n,\, i} |k|^{-\alpha} \int_0^t \langle f_r\nabla S_{t,r} e_l, \sigma_{k,i} \rangle\, \d B^{k,i}_r
+ \int_0^t \langle (f_r - f^n_r)\nabla S_{t,r} e_l, \d W^{n,\alpha}_r \rangle.
\end{align*}
Let us denote the two terms above by $I^{l,1}_t$ and $I^{l,2}_t$.
By the trivial estimate $|k|^{-\alpha}\leq n^{-\alpha}$ and the orthogonality of $\{\sigma_{k,i} \}$, it holds
\begin{align*}
\E\Big[ \big|I^{l,1}_t \big|^2\Big]
&\leq n^{-2\alpha} \sum_{|k|>n,\, i} \int_0^t |\langle f_r\nabla S_{t,r} e_l, \sigma_{k,i} \rangle|^2\, \d r
\leq n^{-2\alpha} \int_0^t \| f_r\nabla S_{t,r} e_l \|_{L^2}^2\, \d r \\
&\leq n^{-2\alpha}\, \| f\|_{L^\infty L^2}^2 \| S_{t,\cdot} e_l \|_{L^2 C^1}^2
\lesssim_\eps n^{-2\alpha}\, \| f_0\|_{L^2}^2 ,
\end{align*}
where $\| S_{t,\cdot} e_l \|_{L^2 C^1}$ can be estimated as at the end of the proof of Proposition \ref{prop:convergence-rate}.

Concerning the second term, fix $\delta$ as in the statement and choose $\eps>0$ such that $\delta+\eps< \alpha\wedge (1-\gamma)$.
Let us recall the following fact: for any $\psi\in C^{\delta +\eps} $ and $ \chi\in H^{-\delta}$, it holds $\psi\chi\in H^{-\delta}$ with
$\| \psi\chi \|_{H^{-\delta}} \lesssim \| \psi\|_{C^{\delta+\eps}} \| \chi\|_{H^{-\delta}}$.
Combining it with $\| \psi\chi\|_{H^{-\alpha}} \leq \| \psi\chi\|_{H^{-\delta}}$, we obtain
\begin{align*}
\E\Big[ \big|I^{l,2}_t \big|^2\Big]
& \lesssim \E\bigg[ \int_0^t \| (f_r-f^n_r) \nabla S_{t,r} e_l\|_{H^{-\alpha}}^2\, \d r \bigg]\\
& \lesssim \sup_{r\in [0,T]} \E\big[ \| f_r-f^n_r\|_{H^{-\delta}}^2\big] \int_0^t \| \nabla S_{t,r} e_l \|_{C^{\delta+\eps}}^2\, \d r.
\end{align*}
Now we can apply again Lemma \ref{lem:dual-equation}, as well as Proposition \ref{prop:convergence-rate}, to get
\begin{align*}
\E\Big[ \big|I^{l,2}_t \big|^2\Big]
\lesssim n^{-2\delta\big(1-\frac{2\alpha}{d}\big) +\eps} \| f_0\|_{L^2}^2 \bigg[ 1+ \int_0^t \| P_{t-r} e_l \|_{C^{1+\delta+\eps}}^2\, \d r \bigg]
\lesssim n^{-2\delta\big(1-\frac{2\alpha}{d}\big) +\eps} \| f_0\|_{L^2}^2\, |l|^{2(\delta+\eps)},
\end{align*}
where in the first step we have used item i) in Lemma \ref{lem:dual-equation}:
$$\int_0^t \| R_{t,r} e_l\|_{C^{1+\delta+\eps}}^2\, \d r \lesssim \int_0^t (t-r)^{1-\gamma-(1+\delta+\eps)} \, \d r= \int_0^t (t-r)^{-\gamma-\delta-\eps} \, \d r \lesssim 1 $$
since $\gamma+ \delta+\eps<1$.
Observe that, since we are imposing $\delta<\alpha$, the estimate for $I^{l,2}_t$ is worse than the one for $I^{l,1}_t$; combining all the above estimate we arrive at
\begin{equation*}
\sup_{t\in [0,T]} \E\big[|\< X_t-X^n_t, e_l\>|^2 \big]
\lesssim n^{-2\delta\big(1-\frac{2\alpha}{d}\big) +\eps}\| f_0\|_{L^2}^2\, |l|^{2(\delta+\eps)}.
\end{equation*}
From here, standard computations like those in Proposition \ref{prop:convergence-rate} allow to reach \eqref{eq:CLT-SLTE-general} (up to possibly relabelling the parameter $\eps$).
\end{proof}

\subsection{Stochastic 2D Euler equations}\label{subs-CLT-SEE}

Recall the setting: using the notation $W^{n,\alpha}$, the It\^o form of equation \eqref{eq:SEE-intro} is given by
\begin{equation}\label{CLT-SEE-0}
\d \xi^n + (K\ast\xi^n)\cdot\nabla \xi^n\, \d t + \sqrt{\eps_n}\, \d W^{n,\alpha}\cdot \nabla \xi^n= \Delta \xi^n\,\d t, \quad \xi^n_0= \xi_0 \in L^2;
\end{equation}
as $n\to\infty$, $\xi^n\to \xi$ solution to the deterministic 2D Navier-Stokes equation in vorticity form:
\begin{equation}\label{CLT-SEE-1}
\partial_t \xi + (K\ast\xi)\cdot\nabla \xi = \Delta \xi,\quad \xi|_{t=0} =\xi_0.
\end{equation}
Let us set $\Xi^n= (\xi^n-\xi)/\sqrt{\eps_n}$, which satisfy in a weak sense the equation
\begin{equation}\label{CLT-SEE-approx-eq}
\d \Xi^n + \big[ (K\ast\Xi^n)\cdot\nabla \xi^n + (K\ast\xi)\cdot\nabla \Xi^n\big]\, \d t + \d W^{n,\alpha}\cdot \nabla \xi^n= \Delta \Xi^n\,\d t.
\end{equation}
Letting $n\to \infty$, any limit point of $\Xi^n$ is expected to solve
\begin{equation}\label{CLT-SEE-limit-eq}
\d \Xi + \big[ (K\ast\Xi)\cdot\nabla \xi + (K\ast\xi)\cdot\nabla \Xi\big]\, \d t + \d W^{\alpha}\cdot \nabla \xi= \Delta \Xi\,\d t,
\end{equation}
where $\xi$ is the unique solution to \eqref{CLT-SEE-1}.

We will first show the well posedness of \eqref{CLT-SEE-limit-eq} and then prove the convergence of $\Xi^n$ to $\Xi$.
We will follow an alternative approach different from the one in Section \ref{subs-CLT-SLTE}; indeed, instead of adopting an analogue of Definition \ref{defn:weak-sol}, we interpret equation \eqref{CLT-SEE-limit-eq} in its mild form:
\begin{equation}\label{eq:CLT-limit-eq-mild}
\Xi_t = - \int_0^t P_{t-s} \big[ (K\ast\Xi_s)\cdot\nabla \xi_s + (K\ast\xi_s)\cdot\nabla \Xi_s \big]\,\d s - \int_0^t P_{t-s} (\d W^{\alpha}_s\cdot \nabla \xi_s).
\end{equation}
To solve this equation, we need some information on the regularity of the stochastic convolution
\[ Z_t:= \int_0^t P_{t-s} (\d W^{\alpha}_s\cdot \nabla \xi_s) = \sum_{k\in \Z^2_0} |k|^{-\alpha} \int_0^t P_{t-s} (\sigma_k \cdot\nabla \xi_s)\,\d B^k_s.\]

\begin{lemma}\label{lem:stoch-convol}
For any $\beta\in (0,\alpha)$ and $p\in [1,\infty)$, it holds
$ \E \big[ \|Z_t \|_{C^0 H^{\beta-1}}^p \big] < +\infty$.
\end{lemma}
The proof follows closely that of \cite[Lemma 2.5]{FGL21c} and is thus postponed to the Appendix.

In view of Lemma \ref{lem:stoch-convol}, we can turn to the study of the \textit{analytic} equation
\begin{equation}\label{eq:CLT-limit-analytic}
y_t = - \int_0^t P_{t-s} \big[ (K\ast y_s)\cdot\nabla \xi_s + (K\ast\xi_s)\cdot\nabla y_s \big]\,\d s +z_t
\end{equation}
where $\xi_0\in L^2$ is given, $\xi$ is the associated unique solution to \eqref{CLT-SEE-1}, $z$ is a given element of $C^0_t H^{\beta-1}$ for some $\beta\in (0,\alpha)$ and $y$ is the unknown.

\begin{proposition}\label{prop:CLT-limit-analytic}
Let $\xi_0\in L^2$ be fixed, $\xi$ be the unique associated solution to \eqref{CLT-SEE-1}, $\beta>0$.
Then for any $z\in C^0_t H^{\beta-1}$, there exists a unique solution $y\in C^0_t H^{\beta-1}$ to \eqref{eq:CLT-limit-analytic}.
Moreover, the solution map $z\mapsto y=:Sz$ is a linear bounded operator and there exists $C>0$ such that
\begin{equation}\label{eq:CLT-limit-estim}
\| S z\|_{C^0 H^{\beta-1}} \lesssim \exp\big(C(1+T)\| \xi_0\|_{L^2}\big)\, \| z \|_{C^0 H^{\beta-1}}.
\end{equation}
\end{proposition}

\begin{proof}
Let us define $a(t):=\int_0^t \big(\| \xi_s\|_{H^1}^2 + \| \xi_s\|_{L^2}^2 \big)\, \d s$; endow $C^0_t H^{\beta-1}_x$ with the equivalent norm
\begin{equation*}
\| y\|_\lambda := \sup_{t\in [0,T]} \big\{ e^{-\lambda a(t)} \| y_t\|_{H^{\beta-1}} \big\}
\end{equation*}
for a suitable $\lambda>0$ to be chosen later. Define a map $\Gamma$ on $C^0 H^{\beta-1}$ by
\begin{equation*}
(\Gamma y)_t := - \int_0^t P_{t-s} \big[ (K\ast y_s)\cdot\nabla \xi_s + (K\ast\xi_s)\cdot\nabla y_s \big]\,\d s +z_t.
\end{equation*}
We are going to show that $\Gamma$ is a contraction on $(C^0_t H^{\beta-1}, \| \cdot\|_\lambda)$ for some large $\lambda>0$, which immediately implies existence and uniqueness of solutions to \eqref{eq:CLT-limit-analytic}; since $\Gamma$ is an affine map, the same computation shows that indeed $\Gamma y\in C^0_t H^{\beta-1}$ whenever $y$ does so.

Given $y^1,\, y^2\in C^0_t H^{\beta-1}$, set $v=y^1-y^2$, then by Lemma \ref{lem:heat-kernel} it holds
\begin{align*}
\| (\Gamma y^1)_t-(\Gamma y^2)_t\|_{H^{\beta-1}}^2
& = \bigg\| \int_0^t P_{t-s} \big[(K\ast v_s)\cdot \nabla \xi_s+ (K\ast \xi_s)\cdot \nabla v_s \big] \,\d s\bigg\|_{H^{\beta-1}}^2\\
& \lesssim \int_0^t \| (K\ast v_s)\cdot \nabla \xi_s\|_{H^{\beta-2}}^2 \,\d s + \int_0^t \| (K\ast \xi_s)\cdot \nabla v_s\|_{H^{\beta-2}}^2 \,\d s =: I^1_t +I^2_t.
\end{align*}
Applying Lemma \ref{lem-estim-transp-term}-(ii) with $b=1-\beta>0$, we get
\begin{align*}
I^1_t
&\lesssim \int_0^t \| K\ast v_s \|_{H^\beta}^2 \| \xi_s \|_{L^2}^2 \, \d s \lesssim \int_0^t \| v_s \|_{H^{\beta-1}}^2 \| \xi_s \|_{L^2}^2 \, \d s \\
&\lesssim \| v\|_\lambda^2 \int_0^t e^{2\lambda a(s)} \| \xi_s \|_{L^2}^2 \, \d s \lesssim \frac{\| v\|_\lambda^2}{2\lambda}\,  e^{2\lambda a(t)};
\end{align*}
next, applying Lemma \ref{lem-estim-transp-term}-(i) with $a=1,\, b=1-\beta$ leads to
\begin{align*}
I^2_t
&\lesssim \int_0^t \| K\ast\xi_s\|_{H^2}^2 \| v_s \|_{H^{\beta-1}}^2 \, \d s
\lesssim \| v\|_\lambda^2 \int_0^t e^{2\lambda a(s)} \| \xi_s\|_{H^1}^2 \, \d s
\lesssim \frac{\| v\|_\lambda^2}{2\lambda}\,  e^{2\lambda a(t)}.
\end{align*}
Combining the above estimates, multiplying both sides by $e^{-2\lambda a(t)}$ and taking the supremum over $t\in [0,T]$, we obtain the existence of a constant $C>0$ independent of $\lambda$ such that
\begin{equation*}
\|\Gamma y^1- \Gamma y^2 \|_\lambda^2
\leq \frac{C}{2\lambda} \| v\|_\lambda^2
= \frac{C}{2\lambda}\, \| y^1- y^2 \|_\lambda^2;
\end{equation*}
contractivity follows by choosing $\lambda$ large enough.

Linearity of $z\mapsto Sz$ follows from the affine structure of \eqref{eq:CLT-limit-analytic}, so we are left with proving \eqref{eq:CLT-limit-estim};
the proof is very similar to the one above, based on several applications of Lemmas \ref{lem:heat-kernel} and \ref{lem-estim-transp-term}.
Indeed if $y$ solves \eqref{eq:CLT-limit-analytic}, then for any $t\in [0,T]$ it holds
\begin{align*}
\| y_t \|_{H^{\beta-1}}^2
& \lesssim \bigg\| \int_0^t P_{t-s} \big[(K\ast y_s)\cdot \nabla \xi_s+ (K\ast \xi_s)\cdot \nabla y_s \big] \,\d s \bigg\|_{H^{\beta-1}}^2 + \| z_t\|_{H^{\beta-1}}^2\\
& \lesssim \int_0^t \big[\| (K\ast y_s)\cdot \nabla \xi_s\|_{H^{\beta-2}}^2 + \| (K\ast \xi_s)\cdot \nabla y_s\|_{H^{\beta-2}}^2 \big]\, \d s + \| z\|_{C^0 H^{\beta-1}}^2\\
& \lesssim \int_0^t \big(\| \xi_s\|_{L^2}^2 + \| \xi_s\|_{H^1}^2 \big) \, \| y_s\|_{H^{\beta-1}}^2\, \d s + \| z\|_{C^0 H^{\beta-1}}^2.
\end{align*}
Estimate \eqref{eq:CLT-limit-estim} readily follows from an application of Gronwall's lemma and the well known bound $\| \xi_t\|_{L^2}^2 + \int_0^t \| \nabla \xi_s\|_{L^2}^2 \d s \leq \| \xi_0\|_{L^2}^2$.
\end{proof}

\begin{corollary}\label{cor:CLT-limit-eq}
Let $\xi_0\in L^2$ be fixed, $\xi$ be the unique associated solution to \eqref{CLT-SEE-1}.
Then there exists a unique solution $\Xi$ to \eqref{eq:CLT-limit-eq-mild}, which is given by $\Xi=S Z$. In particular $\Xi$ is a Gaussian field which satisfies
$\E\big[\| \Xi\|_{C^0 H^{\beta-1}}^p \big]<\infty$ for all $p\in [1,\infty)$ and $\beta\in (0,\alpha)$.
\end{corollary}

\begin{proof}
Strong existence and pathwise uniqueness immediately follow from Proposition \ref{prop:CLT-limit-analytic}, while the moment bounds follow from Lemma \ref{lem:stoch-convol} and estimate \eqref{eq:CLT-limit-estim}. $\Xi$ being Gaussian follows from it being a linear image of the Gaussian field $Z$.
\end{proof}

Before showing the convergence of $\Xi^n$ to $\Xi$, we need to discuss a technical matter.
As already mentioned in the introduction, for $L^2$-initial data $\xi_0$, equation \eqref{CLT-SEE-0} is known to have only probabilistically weak solutions, which might not be unique nor strong.
This means that, for any $\xi_0\in L^2$ and $n\in\N$, we can construct a tuple $(\Omega, \F, (\F_t), \P; W^{\alpha}, \xi^n)$, where $\big(\Omega, \F, (\F_t), \P\big)$ is a filtered probability space, $W^{\alpha}$ is a $Q^\alpha$-Wiener process on $L^2$ with respect to the filtration $\mathcal{F}_t$ and $\xi^n$ is an $\mathcal{F}_t$-adapted process with trajectories in $L^\infty_t L^2_x$ such that \eqref{CLT-SEE-0} (for $W^{n,\alpha}=\Pi_n W^\alpha$) holds in the weak sense.
Since $\xi$ is deterministic, $\Xi^n= (\xi^n -\xi)/\sqrt{\eps_n}$ lives on the same space.

As the limit equation \eqref{CLT-SEE-limit-eq} admits a pathwise unique solution, $\Xi$ can be defined on the given stochastic basis $\big(\Omega, \F, (\F_t), \P; W^\alpha\big)$; in other words, we can deal with the difference $\Xi^n - \Xi$, at fixed $n$. However, we cannot consider the difference of $\Xi^n$ and $\Xi^m$ for $n\neq m$.

Before proceeding further, let us recall the following useful estimate, which is a particular subcase of \cite[Theorem 1.1]{FGL21c}.

\begin{proposition}
Let $\xi_0\in L^2$, $\xi$ be the unique solution to \eqref{CLT-SEE-1}; then for any $p\in [1\,\infty)$, any $\delta\in (0,1)$ and any $\eps\in (0,\delta/2)$, there exists a constant $C=C(\| \xi_0\|_{L^2}, T, \delta, \eps, p)$ such that, for any $n\in\N$ and any weak solution $\xi^n$ to \eqref{CLT-SEE-0}, it holds
\begin{equation}\label{eq:rate-LLN-euler}
\E\big[ \| \xi^n-\xi\|_{C^0 H^{-\delta}}^p\big]^{1/p} \leq C\, \eps_n^{\delta/2-\eps}.
\end{equation}
\end{proposition}

We can now give an explicit estimate on the speed of convergence of $\Xi^n$ to $\Xi$.

\begin{theorem}\label{thm:CLT-SEE-1}
Let $\xi_0\in L^2$ be given, $n\in\N$ and $\xi^n$ be a weak solution to \eqref{CLT-SEE-0}; define the correponding $\Xi^n,\,\Xi$ in the way described above.
Let $\beta\in (0,\alpha)$ and set $\gamma:=\beta\wedge(\alpha-\beta)$.
Then for any $\eps>0$ and any $T>0$ there exists a constant $C=C(\eps,T,\| \xi_0\|_{L^2},\alpha,\beta)$ such that
\begin{equation}\label{eq:rate-CLT-SEE}
\sup_{t\in [0,T]} \E\big[\|\Xi^n_t - \Xi_t \|_{H^{\beta-1}}^2 \big] \leq C n^{-2\gamma(1-\alpha) +\eps}.
\end{equation}
\end{theorem}

\begin{proof}
As before, writing the equations \eqref{CLT-SEE-approx-eq} and \eqref{CLT-SEE-limit-eq} in mild form yields
\begin{align*}
\Xi^n_t &= - \int_0^t P_{t-s} \big[ (K\ast\Xi^n_s)\cdot\nabla \xi^n_s + (K\ast\xi_s)\cdot\nabla \Xi^n_s\big]\, \d s- \int_0^t P_{t-s} (\d W^{n,\alpha}_s \cdot \nabla \xi^n_s), \\
 \Xi_t &= - \int_0^t P_{t-s} \big[ (K\ast\Xi_s)\cdot\nabla \xi_s + (K\ast\xi_s)\cdot\nabla \Xi_s\big]\, \d s- \int_0^t P_{t-s} (\d W^{\alpha}_s \cdot \nabla \xi_s).
\end{align*}
Therefore, for any $t\in [0,T]$,
\begin{align*}
\E\big[ \|\Xi^n_t - \Xi_t \|_{H^{\beta-1}}^2 \big] \lesssim I^n_1(t) + I^n_2(t) + I^n_3(t) + I^n_4(t),
\end{align*}
where
\begin{align*}
I^n_1(t) & = \E\bigg[ \Big\| \int_0^t P_{t-s} \big[ (K\ast(\Xi^n_s -\Xi_s))\cdot\nabla \xi^n_s \big]\, \d s \Big\|_{H^{\beta-1}}^2\bigg], \\
I^n_2(t) & = \E\bigg[ \Big\| \int_0^t P_{t-s} \big[ (K\ast \xi_s) \cdot\nabla (\Xi^n_s -\Xi_s) \big]\, \d s \Big\|_{H^{\beta-1}}^2 \bigg], \\
I^n_3(t) & = \E\bigg[ \Big\| \int_0^t P_{t-s} \big[ (K\ast \Xi_s) \cdot\nabla (\xi^n_s -\xi_s) \big]\, \d s \Big\|_{H^{\beta-1}}^2 \bigg], \\
I^n_4(t) & = \E\bigg[ \Big\| \int_0^t P_{t-s} (\d W^{n,\alpha}_s \cdot \nabla \xi^n_s- \d W^{\alpha}_s \cdot \nabla \xi_s) \Big\|_{H^{\beta-1}}^2 \bigg].
\end{align*}

We will estimate the terms one-by-one, starting with the first two. By Lemma \ref{lem:heat-kernel} and Lemma \ref{lem-estim-transp-term}-(ii), it holds
\begin{align*}
I^n_1(t)
&\lesssim \E\bigg[ \int_0^t \big\| (K\ast(\Xi^n_s -\Xi_s)) \cdot\nabla \xi^n_s \big\|_{H^{\beta-2}}^2\,\d s\bigg] \\
&\lesssim \int_0^t \E\big[ \|K\ast(\Xi^n_s -\Xi_s) \|_{H^{\beta}}^2 \|\xi^n_s \|_{L^2}^2 \big] \,\d s \\
&\lesssim \|\xi_0 \|_{L^2}^2 \int_0^t \E\big[ \| \Xi^n_s -\Xi_s \|_{H^{\beta-1}}^2\big] \,\d s .
\end{align*}
Next, we apply Lemma \ref{lem:heat-kernel} and Lemma \ref{lem-estim-transp-term}-(i) with $a=1$ and $b=1-\beta$, which gives
\begin{align*}
I^n_2(t) \lesssim \int_0^t \E\big[ \| (K\ast\xi_s) \cdot \nabla (\Xi^n_s-\Xi_s)\|_{H^{\beta-2}}^2 \big]\,\d s
\lesssim \int_0^t \| \xi_s\|_{H^1}^2 \E\big[ \|\Xi^n_s-\Xi_s\|_{H^{\beta-1}}^2\big] \, \d s.
\end{align*}
Combining the above estimates yields
\begin{align*}
\E \big[\| \Xi^n_t-\Xi_t\|_{H^{\beta-1}}^2 \big]
& \lesssim \int_0^t \big(\| \xi_0\|_{L^2}^2 + \|\xi_s\|_{H^1}^2 \big) \,\E\big[\| \Xi^n_s-\Xi_s\|_{H^{\beta-1}}^2 \big]\, \d s + I^n_3(t) + I^n_4(t),
\end{align*}
and so by Gronwall's lemma we find
\begin{equation}\label{eq:CLT-euler-proof.1}
\sup_{t\in [0,T]} \E \big[ \| \Xi^n_t-\Xi_t\|_{H^{\beta-1}}^2 \big]
\lesssim e^{C(1+T)\| \xi_0\|_{L^2}^2} \bigg( \sup_{t\in [0,T]} I^n_3(t) + \sup_{t\in [0,T]} I^n_4(t) \bigg).
\end{equation}

To estimate $I^n_3(t)$, let us fix $\delta_1$ such that $\delta_1<\beta\wedge (1-\beta)$ and $\eps$ small enough such that $\delta_1+\eps<1-\beta$; by Lemma \ref{lem:heat-kernel-estim}-(i) it holds
\begin{align*}
I^n_3(t)
& \lesssim \E\bigg[ \Big( \int_0^t \big\| P_{t-s} [(K\ast \Xi_s)\cdot\nabla (\xi^n_s-\xi_s)] \big\|_{H^{\beta-1}}\, \d s \Big)^2 \bigg]\\
& \lesssim \E\bigg[ \Big( \int_0^t (t-s)^{-\frac{1+\beta+\delta_1+\eps}{2}}\, \| (K\ast \Xi_s)\cdot\nabla (\xi^n_s-\xi_s) \|_{H^{-\delta_1-\eps-2}} \, \d s \Big)^2 \bigg].
\end{align*}
Next, by Lemma \ref{lem-estim-transp-term}-(iii) with $a=\beta,\, b=\delta_1$,
\begin{align*}
I^n_3(t)
& \lesssim  \E\bigg[ \Big( \int_0^t (t-s)^{-\frac{1+\beta+\delta_1+\eps}{2}}\, \| K\ast \Xi_s\|_{H^\beta}\, \|\xi^n_s-\xi_s\|_{H^{-\delta_1}}\, \d s \Big)^2 \bigg]\\
& \lesssim \E\bigg[ \| \Xi\|_{C^0 H^{\beta-1}}^2 \| \xi^n -\xi\|_{C^0 H^{-\delta_1}}^2 \Big( \int_0^t (t-s)^{-\frac{1+\beta+\delta_1+\eps}{2}} \d s \Big)^2 \bigg]\\
& \lesssim_T \E\big[ \| \Xi \|_{C^0 H^{\beta-1}}^4 \big]^{1/2}\, \E\big[ \| \xi^n -\xi\|_{C^0 H^{-\delta_1}}^4 \big]^{1/2}
\end{align*}
where in the last passage we used the Cauchy inequality and the fact that $\beta+\delta_1+\eps<1$.
Applying Corollary \ref{cor:CLT-limit-eq} and estimate \eqref{eq:rate-LLN-euler}, overall we obtain
\begin{equation}\label{eq:CLT-euler-proof.2}
\sup_{t\in [0,T]} I^n_3(t) \lesssim \eps_n^{\delta_1-\eps}.
\end{equation}

It remains to estimate $I^n_4(t)$, which is the term requiring the most work.
We split it in two parts: $I^n_4(t)\lesssim J^n_1(t) + J^n_2(t)$ where
\begin{align*}
& J^n_1(t):= \E\bigg[\Big\| \int_0^t P_{t-r} (\d (W^{n,\alpha}_r -W^{\alpha}_r)\cdot \nabla \xi_r)\Big\|_{H^{\beta-1}}^2 \bigg],\\
& J^n_2(t):=\E\bigg[\Big\| \int_0^t P_{t-r} (\d W^{n,\alpha}_r \cdot \nabla (\xi^n_r-\xi_r))\Big\|_{H^{\beta-1}}^2 \bigg].
\end{align*}
First, recalling the expression of $W^{n,\alpha}_r$, the properties of $\{B^k\}_k$ and the fact that $\xi$ is deterministic, by It\^o's isometry on Hilbert spaces we have
\begin{align*}
J^n_1(t)
&= \E \bigg[\Big\| \sum_{|k|>n} |k|^{-\alpha} \int_0^t P_{t-r}(\sigma_k\cdot \nabla \xi_r)\, \d B^k_r  \Big\|_{H^{\beta-1}}^2 \bigg] \\
&\sim \sum_{|k|>n} |k|^{-2\alpha} \int_0^t \| P_{t-r} (\sigma_k\cdot\nabla \xi_r)\|_{H^{\beta-1}}^2\, \d r \\
&\lesssim \sum_{|k|>n} |k|^{-2\alpha} \int_0^t (t-r)^{\eps-1} \| \sigma_k\cdot\nabla \xi_r \|_{H^{\beta+\eps-2}}^2\, \d r
\end{align*}
for small $\eps>0$. Noting that $\sigma_k = \frac{k^\perp}{|k|}  e_k$ is divergence free, for any $s\in \R$ it holds
\begin{equation}\label{eq:basic-estim-CLT-euler-proof}
\| \sigma_k\cdot\nabla \xi_r\|_{H^{s-1}}
= \| \nabla\cdot (\sigma_k\, \xi_r)\|_{H^{s-1}}
\lesssim \| \sigma_k\, \xi_r\|_{H^s} \lesssim
\| e_k\, \xi_r\|_{H^s}.
\end{equation}
Therefore, by Lemma \ref{lem:heat-kernel-estim}, it holds
\begin{align*}
J^n_1(t)
& \lesssim \sum_{|k|>n} |k|^{-2\alpha} \int_0^t (t-r)^{\eps-1} \| e_k\, \xi_r\|_{H^{\beta+\eps-1}}^2\, \d r \lesssim_\eps \sup_{r\in [0,T]} \sum_{|k|>n} |k|^{-2\alpha} \| e_k\, \xi_r\|_{H^{\beta+\eps-1}}^2
\end{align*}
with uniform estimate over $t\in [0,T]$.
Now take $\delta_2>0$ such that $\delta_2<\alpha-\beta$ and choose $\eps>0$ in the previous estimate small enough so that $\delta_2+\eps<\alpha-\beta$, and thus $\beta+\eps<\alpha <1$; then
\begin{align*}
\sum_{|k|>n} |k|^{-2\alpha} \| e_k\, \xi_r\|_{H^{\beta+\eps-1}}^2
& = \sum_{|k|>n} |k|^{-2\alpha} \sum_j |\langle \xi_r, e_{k-j}\rangle|^2 |j|^{-2(1-\beta-\eps)}\\
& \leq n^{-2\delta_2} \sum_{k,j} |k|^{-2(\alpha-\delta_2)} |\langle \xi_r, e_{k-j}\rangle|^2 |j|^{-2(1-\beta-\eps)}.
\end{align*}
Note that the last sum can be written as
$\<a,b\ast c\>_{\ell^2}$,
where
$a_k = |k|^{-2(\alpha-\delta_2)},\, b_k = |k|^{-2(1-\beta-\eps)}$
and
$c_k = |\<\xi_r, e_{k}\>|^2,\, k\in \Z^2_0$.
We have $c\in \ell^1(\Z^2_0)$ with $\|c \|_{\ell^1}= \|\xi_r \|_{L^2}^2$; moreover, we can find $p\in [1,\infty]$ such that $\alpha-\delta_2>1/p>\beta+\eps$, so that $a\in \ell^p(\Z^2_0),\, b\in \ell^{p'}(\Z^2_0)$, where $p'$ is the conjugate number of $p$.
By H\"older's inequality and Young's inequality for convolution, one has
\begin{align*}
\sum_{|k|>n} |k|^{-2\alpha} \| e_k\, \xi_r\|_{H^{\beta+\eps-1}}^2
\leq n^{-2\delta_2} \|a \|_{\ell^p} \|b\ast c\|_{\ell^{p'}}
\lesssim_{\alpha,\beta,\delta_2,\eps} n^{-2\delta_2} \| \xi_r\|_{L^2}^2
\leq n^{-2\delta_2} \| \xi_0\|_{L^2}^2
\end{align*}
and so overall we obtain
\begin{equation}\label{eq:CLT-euler-proof.3}
\sup_{t\in [0,T]} J^n_1(t) \lesssim n^{-2\delta_2}.
\end{equation}

We turn to estimating $J^n_2$; to this end, let us introduce the notation $\tilde{\xi}^n=\xi^n-\xi$.
Applying as before the It\^o isometry on Hilbert spaces, the structure of $W^{n,\alpha}$, as well as estimates of the form \eqref{eq:basic-estim-CLT-euler-proof}, for any $\eps>0$ it holds
\begin{align*}
J^n_2(t)
& = \E\bigg[  \Big\| \sum_{|k|\leq n} |k|^{-\alpha} \int_0^t P_{t-r} (\sigma_k\cdot\nabla \tilde{\xi}^n_r)\, \d B^k_r\Big\|_{H^{\beta-1}}^2 \bigg]\\
& \sim \E\bigg[ \sum_{|k|\leq n} |k|^{-2\alpha} \int_0^t \|P_{t-r} (\sigma_k\cdot\nabla \tilde{\xi}^n_r)\|_{H^{\beta-1}}^2\, \d r \bigg]\\
& \lesssim \E\bigg[ \sum_k |k|^{-2\alpha} \int_0^t (t-r)^{\eps-1} \| e_k \,\tilde{\xi}^n_r\|_{H^{\beta+\eps-1}}^2 \, \d r \bigg]\\
& \lesssim_{\eps,T} \sup_{r\in [0,T]} \E\bigg[ \sum_k |k|^{-2\alpha} \| e_k\, \tilde{\xi}^n_r\|_{H^{\beta+\eps-1}}^2 \bigg].
\end{align*}
Now take $\delta_3$ such that $\delta_3<\alpha-\beta$ and choose $\eps$ small accordingly so that $\delta_3+\eps<\alpha-\beta$; then
\begin{align*}
\sum_k |k|^{-2\alpha} \| e_k\, \tilde{\xi}^n_r\|_{H^{\beta+\eps-1}}^2
& = \sum_{k,j} |k|^{-2\alpha} |\langle \tilde{\xi}^n_r, e_{k-j}\rangle|^2 |j|^{-2(1-\beta-\eps)}\\
& = \sum_j |\langle \tilde{\xi}^n_r, e_j \rangle|^2 \bigg(\sum_k |k|^{-2\alpha} |j-k|^{-2(1-\beta-\eps)}\bigg)\\
& \lesssim \sum_j |j|^{-2\delta_3} |\langle \tilde{\xi}^n_r, e_j \rangle|^2 = \| \tilde{\xi}^n_r\|_{H^{-\delta_3}}^2
\end{align*}
where we applied in the third step Lemma \ref{lem-estimate} from the Appendix. By estimate \eqref{eq:rate-LLN-euler}, for any $\eps>0$ it holds
\begin{equation}\label{eq:CLT-euler-proof.4}
\sup_{t\in [0,T]} J^n_2(t)
\lesssim \E \big[ \| \tilde{\xi}^n\|_{C^0 H^{-\delta_3}}^2\big]
\lesssim \eps_n^{\delta_3-\eps}.
\end{equation}

Combining estimates \eqref{eq:CLT-euler-proof.1}, \eqref{eq:CLT-euler-proof.2}, \eqref{eq:CLT-euler-proof.3} and \eqref{eq:CLT-euler-proof.4}, we obtain
\begin{equation}\label{eq:CLT-euler-proof.5}
\sup_{t\in [0,T]} \E[\| \Xi^n_t-\Xi_t\|_{H^{\beta-1}}^2 \lesssim \eps_n^{\delta_1-\eps} + n^{-2\delta_2} + \eps_n^{\delta_3-\eps}.
\end{equation}
Using the facts that $\eps_n\sim n^{2\alpha-2} \gg n^{-2}$ and $\alpha-\beta<1-\beta$, by a few algebraic manipulations and the conditions on $\delta_i$, one can finally check that the best possible rate coming from \eqref{eq:CLT-euler-proof.5} is indeed the one given by \eqref{eq:rate-CLT-SEE}.
\end{proof}

\begin{corollary}\label{cor:CLT-euler-1}
Under the assumptions of Theorem \ref{thm:CLT-SEE-1}, for any $\eps>0$ there exists a constant $C=(\| \xi_0\|_{L^2}, T,\alpha,\eps)$ such that
\begin{equation}
\sup_{t\in [0,T]} \E\big[ \| \Xi^n_t - \Xi_t\|_{H^{-1}}^2\big] \leq C n^{-\alpha(1-\alpha)+\eps}
\end{equation}
\end{corollary}

\begin{proof}
The statement follows from Theorem \ref{thm:CLT-SEE-1}, the trivial estimate $\| \cdot\|_{H^{-1}}\leq \| \cdot \|_{H^{\beta-1}}$ and the choice $\beta=\alpha/2$, which yields $2\gamma=\alpha$.
\end{proof}

Next, we can formulate a result which does not involve the specific tuple $(\Omega,\mathcal{F}_t,\P; \xi^n,W^\alpha)$ in consideration, but only the law of $\Xi^n$ and $\Xi$, in terms of Wasserstein distances.
To this end, given two probability measures $\mu,\,\nu\in \mathcal{P}(L^2_t H^{-1})$ we denote by $d_2$ their $2$-Wasserstein distance as defined by
\begin{equation*}
d_2(\mu,\nu)=\inf_{\pi\in \mathcal C(\mu,\nu)} \bigg\{\int_{L^2 H^1\times L^2 H^1} \| x-y\|_{L^2 H^1}^2\, \d \pi(x,y)\bigg\}^{1/2}
\end{equation*}
where $\mathcal C(\mu,\nu)$ denotes the set of all possible couplings of $\mu,\nu$.

\begin{corollary}\label{cor:CLT-euler-2}
Let $\xi_0\in L^2$, $n\in\N$ and $\xi^n$ be a weak solution to \eqref{CLT-SEE-0}; set $\mu^n=Law(\Xi^n)$ for the correponding $\Xi^n=(\xi^n-\xi)/\sqrt{\eps_n}$ and $\mu=Law(\Xi)$. Then for any $\eps>0$ it holds
\begin{equation}\label{eq:rate-CLT-wasserstein}
d_2(\mu^n,\mu)\leq C^{1/2}\, n^{-\frac{\alpha(1-\alpha)+\eps}{2}}
\end{equation}
where $C$ is the same constant appearing in Corollary \ref{cor:CLT-euler-1}.
\end{corollary}

\begin{proof}
Since strong existence and uniqueness holds for \eqref{CLT-SEE-limit-eq}, we can construct $\Xi$ on the same stochastic basis $(\Omega,\mathcal{F}_t,\P;W^\alpha)$ where $\xi^n$ is defined; as a consequence, the joint law of the pair $(\Xi^n,\Xi)$ is a coupling of $(\mu^n,\mu)$.
Estimate \eqref{eq:rate-CLT-wasserstein} then immediately follows from the definition of $d_2$ and Corollary \ref{cor:CLT-euler-1}.
\end{proof}

\subsection{Concluding remarks}\label{sec:conclusive-CLT}

Let us recall here the abstract setting considered in the introduction, namely for
\begin{equation}\label{eq:abstract-SPDE}
\d u^n + \sqrt{\eps_n}\, \d W^{n,\alpha}\cdot \nabla u^n= [F(u^n) + \Delta u^n] \, \d t.
\end{equation}
We have shown in the previous sections the rigorous convergence of $(u^n-u)/\sqrt{\eps_n}$ to a Gaussian variable  $U$, satisfying in a suitable sense the equation
\[
\d U  + \d W^\alpha\cdot\nabla u = [DF_u (U) + \Delta U]\, \d t,
\]
in two special cases of interest, i.e. for $F(u)=-b\cdot\nabla u$ and $F(u)=-(K\ast u)\cdot\nabla u$.

The results have been accomplished by two alternative techniques:
\begin{itemize}
\item[i)] As done in Section \ref{subs-CLT-SLTE}, one can study the well-posedness and regularity theory for the backward  dual equation
\begin{equation}\label{eq:abstract-backward-dual-PDE}
\partial_s g +\Delta g= - (DF_{u})^\ast g, \quad s\leq t,\ g_{t}=\varphi
\end{equation}
where $(DF_{u})^\ast$ denotes the dual of $DF_{u}$ as a linear operator; denoting by $S_{t,r}$ the corresponding solution operator, it then holds (for $g$ regular enough)
\[
\langle U_t, \varphi\rangle = \int_0^t \langle u_r \nabla S_{t,r}\varphi, \d W^\alpha_r \rangle.
\]
\item[ii)] As done in Section \ref{subs-CLT-SEE}, one can write the equation in the corresponding mild form
\begin{equation}\label{eq:abstract-mild-form}
U_t = \int_0^t P_{t-s} DF_{u_s} (U_s)\, \d s - \int_0^t P_{t-s} (\d W^\alpha_s \cdot \nabla u_s)
\end{equation}
and try to solve it directly using heat kernel and functional estimates.
\end{itemize}

Both methods when applicable yield well-posedness for $U$, information on its space-time regularity and rates of convergence for the moments of $U^n-U$.
Remarkably, this also happens in situations where the well-posedness in law of the original nonlinear SPDE for $u$ is not known, like in the case of $2$D Euler equations with $L^2$ initial vorticity.

As the equation for $U$ is always of linear type, we expect the same techniques to be successful for a large variety of SPDEs; however since both equations \eqref{eq:abstract-backward-dual-PDE} and \eqref{eq:abstract-mild-form} depend crucially on the nonlinearity $F$, each different choice might require a separate study.
In particular, in order to obtain explicit rates of convergence, uniform estimates on $u^n$ (e.g. $L^\infty_t L^2_x$-type of bounds) might be needed, which depend on the structure of the nonlinearity.\medskip

Below we outline some possible generalizations of these CLT-type results; we keep the discussion at an heuristic level, without formulating exactly the necessary assumptions or giving rigorous statements.

So far we have considered for simplicity only the case $u^n_0=u_0$ for all $n\in \N$.
This can be generalized to $u^n_0=u_0+\sqrt{\eps_n}\, v_0$ for some $v_0\in L^2$; similarly, rather than studying $U^n=(u^n-u)/\sqrt{\eps_n}$, one might look at the convergence of $\tilde{U}^n=(u^n-\E[u^n])/\sqrt{\eps_n}$.
Assuming one has shown convergence of $\tilde{U}^n$ to $\tilde U$, since the function $u$ is deterministic it holds
\[
\frac{u^n-\E[u^n]}{\sqrt{\eps_n}}
=\frac{u^n-u}{\sqrt{\eps_n}} - \E\bigg[ \frac{u^n-u}{\sqrt{\eps_n}}\bigg]
\to \tilde U-\E[\tilde U];
\]
moreover $\tilde U$ must now solve in a suitable sense the equation
\begin{equation*}
\d \tilde U + \d W^\alpha\cdot\nabla u = [ DF_u (\tilde U) + \Delta \tilde U]\, \d t, \quad \tilde U\vert_{t=0} = v_0.
\end{equation*}
Since the stochastic integral is a martingale, by linearity of the equation, $\psi:=\E[\tilde U]$ solves $\partial_t \psi = DF_u(\psi) + \Delta \psi$ with initial condition $\psi\vert_{t=0} =v_0$. As a consequence, $U=\tilde{U}-\E[\tilde{U}]$ solves
\begin{equation*}
\d U + \d W^\alpha\cdot\nabla u = [ DF_u(U) + \Delta U ]\, \d t, \quad U\vert_{t=0} = 0.
\end{equation*}
Namely, $U$ is exactly the same solution that was studied in the previous sections, and the correct asymptotics in this case are given by $u^n \approx \E[u^n] + \sqrt{\eps_n}\,  U \approx u + \sqrt{\eps_n} (\psi + U)$.\medskip

Recall that, as explained in the introduction, the limit equation for $U^n := (u^n - u)/\sqrt{\eps_n}$ can be formally derived by imposing the expansion $u^n\approx u + \sqrt{\eps_n}\, U + o(\sqrt{\eps_n})$ inside \eqref{eq:abstract-SPDE}.
In the same spirit, one could try to derive higher order expansions by imposing
\[ u^n \approx \sum_{i\leq N} \eps_n^{i/2}\, U^{(i)} + o (\eps_n^{N/2}) \]
and assuming $F$ is regular enough to expand it as
\[ F(u^n)  \approx  \sum_j \frac{1}{j!} D^{(j)}_{u} F\big( (u^n-u)^{\otimes j} \big).\]
In this way, we recover $U=U^{(1)}$, while for instance the resulting equation for $U^{(2)}$ would be
\begin{equation}\label{eq:second-order-expansion}
\d U^{(2)} + \d W^\alpha \cdot \nabla U^{(1)} = \big[ D^{(2)}_u F(U^{(1)}, U^{(1)}) + D F_u (U^{(2)}) + \Delta U^{(2)} \big]\, \d t.
\end{equation}
Equation \eqref{eq:second-order-expansion} is still linear in  $U^{(2)}$, with the same linear part $\mathcal{L}_t=\Delta + DF_{u_t}(\cdot)$ as in the equation for $U^{(1)}$, but the term $D^{(2)}_u F(U^{(1)}, U^{(1)})\, \d t - \d W^\alpha \cdot \nabla U^{(1)}$ is nonlinear, thus the resulting solution will not be a Gaussian field.

Establishing well-posedness of \eqref{eq:second-order-expansion}, not to mention convergence of $(u^n-u-\sqrt{\eps_n} \, U^{(1)}) /\eps_n$ to $U^{(2)}$, might be much more challenging than for $U^{(1)}$, even in the case $F\equiv 0$.
Indeed, the regularity estimates available for $U^{(1)}$ are worse than those for $u$, making the stochastic integral term difficult to define.
Let us only point out that, for $F\equiv 0$, one would expect the iterated mild formulation
\begin{align*}
U^{(2)}_t
= - \int_0^t P_{t-s} \big(\d W^\alpha_s \cdot \nabla U^{(1)}_s \big)
= \int_0^t P_{t-s} \Big(\d W^\alpha_s \cdot \nabla \int_0^s P_{s-r} (\d W^\alpha_r\cdot \nabla u_r )\Big),
\end{align*}
which might still be meaningful in some weak sense.
We leave the rigorous investigation of the well-posedness of \eqref{eq:second-order-expansion} and the higher order expansion of $u^n$ for the future.

\appendix
\section{Some Technical Lemmas}\label{appendix}

\begin{lemma}\label{lem-topology}
Given $R>0$, let $\mathcal E_R=\{f\in L^\infty: \|f\|_{L^\infty} \leq R\}$. Then for any $\alpha>0$, $\mathcal E_R$ endowed with the distance $d(f,g)=\| f-g\|_{H^{-\alpha}}$ is a Polish space. Moreover, the weak-$\ast$ topology of $\mathcal E_R$ is equivalent to that induced from $H^{-\alpha}$.
\end{lemma}

\begin{proof}
We provide a short proof for the reader's convenience. Let $\{f_n\}_n \subset \mathcal E_R$ be a sequence such that $f_n\to f$ in $H^{-\alpha}$; since $\mathcal{E}_R$ is weakly-$\ast$ compact, we can always extract a (not relabelled) subsequence $f_n \stackrel{\ast}{\rightharpoonup} \hat{f}\in L^\infty$, $\hat{f}\in \mathcal{E}_R$.
As this also implies that $f_n\rightharpoonup \hat{f}$ in $H^{-\alpha}$, and the reasoning holds for any possible subsequence, we conclude that $\hat{f}=f$, $f_n\stackrel{\ast}{\rightharpoonup} f$ and that $\mathcal{E}_R$ is closed in $H^{-\alpha}$.
Since $\mathcal{E}_R$ is a closed subset of a Polish space, it is Polish as well.

The above argument shows that convergence in $H^{-\alpha}\cap \mathcal{E}_R$ implies weak-$\ast$ convergence in $L^\infty$; conversely, if a sequence $f_n \stackrel{\ast}{\rightharpoonup} f$, then by the compact embedding $L^\infty \hookrightarrow H^{-\alpha}$ it follows that $f_n\to f$ in $H^{-\alpha}$, showing the desired equivalence.
\end{proof}

The next two lemmas concern some well known heat kernel properties (see \cite[Lemma 2.3]{FGL21c} for a proof of the first one).

\begin{lemma}\label{lem:heat-kernel}
For any $s\in\mathbb{R}$ and any $f\in
L^2_t H^s$, it holds
\[
\bigg\|\int_{0}^{t} e^{(t-r) \Delta} f_r\, \,\d r
\bigg\|_{H^{s+1}}^{2} \lesssim\int_{0}^{t} \|
f_r\|_{H^s}^{2}\,\d r \quad\forall\, t\in[0,T].
\]
\end{lemma}

\begin{lemma}\label{lem:heat-kernel-estim}
Let $u\in H^a$, $a \in\mathbb{R}$. Then:
\begin{itemize}
\item[\rm{(i)}] for any $\rho\geq0$, it holds $\|e^{t\Delta} u \|_{H^{a+ \rho}} \leq C_{\rho}\, t^{-\rho/2} \|u\|_{H^a}$ for some constant increasing in $\rho$;
\item[\rm{(ii)}] for any $\rho\in[0,2]$, it holds $\| (I-e^{t\Delta}) u\|_{H^{a-\rho}}\lesssim t^{\rho/2} \|u\|_{H^a}$.
\end{itemize}
\end{lemma}

Next we recall some simple estimates on the transport term which are frequently used in the paper.

\begin{lemma}\label{lem-estim-transp-term}
Let $V\in L^2(\T^2,\R^2)$ be a divergence free vector field.
\begin{itemize}
\item[\rm(i)] Let $0<b< a\leq 1,\, V\in H^{1+ a}$ and $f\in H^{-b}$, then
\begin{equation*}
\|V\cdot\nabla f\|_{H^{-1-b}} \lesssim_{a, b} \|V \|_{H^{1+a}} \|f \|_{H^{-b}}.
\end{equation*}
\item[\rm(ii)] Let $b\in (0,1),\, V\in H^{1-b}$ and $f\in L^2$, then
\begin{equation*}
\|V\cdot\nabla f\|_{H^{-1-b}} \lesssim_{b} \|V \|_{H^{1-b}} \|f \|_{L^2}.
\end{equation*}
\item[\rm(iii)] Let $0<b\leq a< 1,\,\eps>0, V\in H^a,\,f\in H^{-b}$, then
\begin{equation*}
\|V\cdot\nabla f\|_{H^{-2-b -\eps}} \lesssim_{a, b,\eps} \|V \|_{H^a} \|f \|_{H^{-b}}.
\end{equation*}
\end{itemize}
\end{lemma}

\begin{proof}
Estimates \rm(i), \rm(ii) come from \cite[Lemma 2.1]{FGL21c}, so we only need to show \rm{(iii)}.
First observe that, since $V$ is divergence free, $\| V\cdot\nabla f\|_{H^{-2- b-\eps}}\leq \| V  f\|_{H^{-1-b -\eps}}$. Next, since $a\geq b$, for any $\varphi\in H^{1+b+\eps}\hookrightarrow C^{b+\eps}$, the product $\varphi\, V$ belongs to $H^b$ and we have the inequalities
\begin{align*}
|\langle V f,\varphi\rangle|
& = |\langle f, V\varphi\rangle|
\leq \| f\|_{H^{-b}} \| V \varphi\|_{H^b}\\
& \lesssim \| f\|_{H^{-b}} \| V\|_{H^a} \| \varphi\|_{C^{b+\eps}}
\lesssim \| f\|_{H^{-b}} \| V\|_{H^a} \| \varphi\|_{H^{1+ b+\eps}};
\end{align*}
by duality, this readily implies assertion \rm{(iii)}.
\end{proof}

We provide here the proof of Lemma \ref{lem:stoch-convol}; as already mentioned, it is very similar to those from \cite{FGL21c}, the main difference being that the exact expression $\theta_k=|k|^{-\alpha}$ allows to derive slightly better regularity estimates.

\begin{proof}[Proof of Lemma \ref{lem:stoch-convol}]

Fix $\eps\in (0, \alpha- \beta)$; by Burkholder-Davis-Gundy's inequality, we have
\begin{align*}
\E\big[\|Z_t \|_{H^{\beta-1}}^{2p} \big]
&= \E\bigg[ \Big\| \sum_{k\in \Z^2_0} |k|^{-\alpha} \int_0^t P_{t-s} (\sigma_k \cdot\nabla \xi_s)\,\d B^k_s \Big\|_{H^{\beta-1}}^{2p} \bigg] \\
&\lesssim \E\bigg[ \Big( \sum_{k} |k|^{-2\alpha} \int_0^t \| P_{t-s} (\sigma_k \cdot\nabla \xi_s) \|_{H^{\beta-1}}^2 \,\d s \Big)^{p} \bigg] \\
&\lesssim \E\bigg[ \Big( \sum_{k} |k|^{-2\alpha} \int_0^t (t-s)^{\eps-1} \| \sigma_k \cdot\nabla \xi_s \|_{H^{\beta-2+\eps}}^2 \,\d s \Big)^{p} \bigg],
\end{align*}
where in the last step we used Lemma \ref{lem:heat-kernel-estim}-(i).
Now we follow the treatment of $J^n_1$ in the proof of Theorem \ref{thm:CLT-SEE-1}. Similarly to \eqref{eq:basic-estim-CLT-euler-proof}, $\| \sigma_k \cdot\nabla \xi_s \|_{H^{\beta-2+\eps}} \lesssim \|e_k \, \xi_s\|_{H^{\beta-1+\eps}}$ and thus
\[ \sum_{k} |k|^{-2\alpha} \| \sigma_k \cdot\nabla \xi_s \|_{H^{\beta-2+\eps}}^2
\lesssim \sum_{k,l} |k|^{-2\alpha} |l|^{-2(1-\beta-\eps)} |\<\xi_s, e_{k-l}\>|^2.\]
Note that the right-hand side can be written as $\<a,b\ast c\>_{\ell^2}$,
where $a_k = |k|^{-2\alpha},\, b_k = |k|^{-2(1-\beta-\eps)}$
and $c_k = |\<\xi_s, e_{k}\>|^2,\, k\in \Z^2_0$;
$c\in \ell^1(\Z^2_0)$ with $\|c \|_{\ell^1}= \|\xi_s \|_{L^2}^2$ and for any $p\in (1/\alpha, 1/(\beta+\eps))$,
one has $a\in \ell^p(\Z^2_0),\, b\in \ell^{p'}(\Z^2_0)$, where $p'$ is the conjugate number of $p$.
By H\"older's inequality and Young's inequality for convolution,
\[ \sum_{k} |k|^{-2\alpha} \| \sigma_k \cdot\nabla \xi_s \|_{H^{\beta-2+\eps}}^2
\lesssim \|a \|_{\ell^p} \|b\ast c\|_{\ell^{p'}} \leq \|a \|_{\ell^p} \|b\|_{\ell^{p'}} \|c \|_{\ell^1}
\lesssim_{\alpha,\beta, p, \eps} \|\xi_s \|_{L^2}^2. \]
Substituting this estimate into the the above inequality yields
\[ \E\big[\|Z_t \|_{H^{\beta-1}}^{2p} \big]^{1/2p}
\lesssim \E\bigg[ \Big( \int_0^t (t-s)^{\eps-1} \|\xi_s \|_{L^2}^2 \,\d s \Big)^{p} \bigg]^{1/2p} \lesssim_{\eps} t^{\eps/2} \|\xi_0 \|_{L^2}. \]
A similar computation gives us
\[ \E\bigg[ \Big\| \int_s^t P_{t-r} (\d W^\alpha_r \cdot\nabla \xi_r) \Big\|_{H^{\beta-1}}^{2p} \bigg]^{1/2p}
\lesssim |t-s|^{\eps/2} \|\xi_0 \|_{L^2}. \]
Next, observing that by construction $Z$ satisfies
$Z_t = P_{t-s} Z_s + \int_s^t P_{t-r}(\d W^{\alpha}_r\cdot \nabla \xi_r)$,
by Lemma \ref{lem:heat-kernel-estim}-(ii) we obtain
\begin{align*}
\|Z_t-Z_s\|_{H^{\beta-1-\eps}}
& \leq \| (I- P_{t-s}) Z_s\|_{H^{\beta-1-\eps}} + \Big\| \int_s^t P_{t-r}(\d W^{\alpha}_r\cdot \nabla \xi_r) \Big\|_{H^{\beta-1-\eps}}\\
& \lesssim |t-s|^{\eps/2} \| Z_s\|_{H^{\beta-1}} + \Big\| \int_s^t P_{t-r}(\d W^{\alpha}_r\cdot \nabla \xi_r) \Big\|_{H^{\beta-1}}.
\end{align*}
Taking expectation and applying the previous estimates we arrive at
\begin{align*}
\E\big[ \| Z_t-Z_s\|_{H^{\beta-1-\eps}}^{2p}\big]^{1/2p}
& \lesssim_{p,T} |t-s|^{\eps/2}  \|\xi_0 \|_{L^2} .
\end{align*}
Choosing $p > 1/\eps$ and applying Kolmogorov's continuity criterion, we obtain the desired result, up to renaming $\beta -\eps$ as $\beta$.
\end{proof}

The following lemma was used to derive \eqref{eq:CLT-euler-proof.4}.

\begin{lemma}\label{lem-estimate}
Let $a,b>0$ such that $a+b >d$, $\delta\in (0, a\wedge b\wedge (a+b-d))$. Then
\begin{equation*}
\sum_{l\neq 0, j} |l|^{-a} |j-l|^{-b} \lesssim_{a,b,\delta} |j|^{-\delta} \quad \forall\, j\in \Z^d_0.
\end{equation*}
\end{lemma}

\begin{proof}
It is sufficient to deal with $j\in \Z^d_0$ with $|j|$ big, say $|j|\geq 10$. First observe that, for any positive pair $(\tilde{a}, \tilde{b})$ with $\tilde{a}+\tilde b> d$, by H\"older's inequality it holds
\begin{align*}
\sum_{l\neq 0,j} |l|^{-\tilde a} |j-l|^{-\tilde b} \lesssim \bigg( \sum_{l\neq 0} |l|^{-\tilde{a} p}\bigg)^{1/p} \bigg( \sum_{l\neq 0} |l|^{-\tilde{b} p'}\bigg)^{1/p'} \leq C
\end{align*}
uniformly over $j\in \Z^d_0$, as soon as we choose $p\in (1,\infty)$ such that
\[\frac{d}{\tilde a} <p< \frac{d}{d-\tilde b},\quad p'=\frac{p}{p-1}\]
which is always possible since $\tilde{a}+\tilde b> d$.
We have
\begin{align*}
\sum_{l\neq 0,j} |l|^{-a} |j-l|^{-b}
= \bigg(\sum_{1\leq |l|\leq |j|/2} + \sum_{|l|>|j|/2, l\neq j} \bigg) |l|^{-a} |j-l|^{-b}.
\end{align*}
We denote the two terms by $A_1, A_2$. For the first one, we have $|j-l| \geq |j|-|l|\geq |j|/2$, thus
\begin{align*}
A_1
& \lesssim |j|^{- \delta} \sum_{1\leq |l| \leq |j|/2} |l|^{-a} |j-l|^{-b+\delta}
\leq |j|^{-\delta} \sum_{l\neq 0,j} |l|^{-a} |j-l|^{-b+\delta} \lesssim |j|^{-\delta}
\end{align*}
where we used the observation at the beginning of the proof and the fact that $a+b -\delta>d$. The second term is similar:
\begin{align*}
A_2
& \lesssim |j|^{- \delta} \sum_{|l|> |j|/2, l\neq j} |l|^{-a+\delta} |j-l|^{-b}
\leq |j|^{-\delta} \sum_{l\neq 0,j} |l|^{-a+\delta} |j-l|^{-b} \lesssim |j|^{-\delta}.
\end{align*}
Summing up these estimates we obtain the desired result.
\end{proof}

Before presenting the next results, let us define the following seminorm: for $\gamma \in (0,1)$ and $v:[0,T]\times \T^d \to \R$, set
\begin{equation*}
\llbracket v\rrbracket_{\gamma; T}:=\sup_{t\in [0,T]} t^{\frac{\gamma}{2}} \| v_t\|_{C^1}.
\end{equation*}

\begin{theorem}\label{thm:parabolic-pde}
Consider the PDE
\begin{equation}\label{eq:parabolic-pde}
\partial_t h = \Delta h + b\cdot\nabla h + c\, h, \quad h\vert_{t=0}=h_0,
\end{equation}
where $b,\,c\in L^q_t L^p_x$ for parameters $(p,q)\in (2,\infty]$ satisfying
\begin{equation}\label{eq:KR-condition}
\gamma:= \frac{2}{q}+\frac{d}{p}<1.
\end{equation}
Then for any $h_0\in L^\infty$ there exists a weak solution to \eqref{eq:parabolic-pde}, which is of the form $h_t = P_t h_0 + v_t$ with $v\in C^0_t C^0_x$ satisfying $v_0=0$ and
\begin{equation}\label{eq:structure-solution-pde-1}
\llbracket v\rrbracket_{\gamma; T} \lesssim_{b,c} \| h_0\|_{L^\infty}.
\end{equation}
Moreover, $h$ is unique in the class of such solutions and for any $\delta\in [0,2-\gamma)$ it holds
\begin{equation}\label{eq:structure-solution-pde-2}
\sup_{t\in [0,T]} t^{\frac{\gamma+\delta-1}{2}} \| v_t\|_{C^\delta} \lesssim_{b,c} \| h_0\|_{L^\infty};
\end{equation}
in particular, $v\in L^\infty_t C^\delta_x$ for any $\delta\in [0,1-\gamma]$ and $v\in L^2_t C^\delta_x$ for any $\delta<2-\gamma$.
\end{theorem}

\begin{remark}
By time reversal, similar results hold if we considered the backward equation $(\partial_t + \Delta) h = b\cdot\nabla h + c\,h$ with given terminal condition $h_T\in L^\infty$. Choosing $c=\nabla\cdot b$, the statements apply to forward and backward Fokker-Planck equations $(\partial_t\pm \Delta h)=\nabla\cdot (bh)$.
\end{remark}

\begin{proof}
Regarding the first part of the statement, it suffices to prove property \eqref{eq:structure-solution-pde-1} in the case of smooth coefficients $b,\,c,\,h_0$; existence will then follow from the uniform estimates and standard compactness arguments, while uniqueness by the linearity of the equation.
In the smooth setting, the unique solution $h$ satisfies the mild formulation associated to \eqref{eq:parabolic-pde}, and so $h_t=P_t h_0 + v_t$ with $v$ satisfying
\begin{equation*}
v_t = \int_0^t P_{t-s} (b_s\cdot\nabla v_s + c_s v_s)\, \d s + \int_0^t P_{t-s}(b_s\cdot\nabla P_s h_0 + c_s \, P_s h_0)\, \d s.
\end{equation*}
Denote the two terms by $I^1_t$ and $I^2_t$; by standard heat kernel estimates we obtain
\begin{align*}
\|I^1_t\|_{C^1}
&\lesssim \int_0^t (t-s)^{-\frac{1}{2} -\frac{d}{2p}} (\|b_s\|_{L^p} + \| c_s\|_{L^p}) \|v_s\|_{C^1}\, \d s \\
&\leq \llbracket v\rrbracket_{\gamma; t} \int_0^t (\|b_s\|_{L^p} + \| c_s\|_{L^p}) (t-s)^{-\frac{1}{2}-\frac{d}{2p}} s^{-\frac\gamma2} \, \d s
\end{align*}
as well as
\begin{align*}
\|I^2_t\|_{C^1}
&\lesssim \int_0^t (t-s)^{-\frac{1}{2}-\frac{d}{2p}} (\|b_s\|_{L^p} + \| c_s\|_{L^p}) \|P_s h_0 \|_{C^1}\, \d s \\
&\leq \|h_0 \|_\infty \int_0^t (\|b_s\|_{L^p} + \| c_s\|_{L^p}) (t-s)^{-\frac{1}{2}-\frac{d}{2p}} s^{-\frac12} \, \d s.
\end{align*}
Noting that $\gamma <1$, combining these estimates and H\"older's inequality gives us
\begin{align*}
\| v_t\|_{C^1}
&\lesssim (\llbracket v\rrbracket_{\gamma; t}+ \|h_0 \|_\infty ) \int_0^t (\|b_s\|_{L^p} + \| c_s\|_{L^p}) (t-s)^{-\frac{1}{2}-\frac{d}{2p}} s^{-\frac12} \, \d s \\
&\leq (\llbracket v\rrbracket_{\gamma; t}+ \|h_0 \|_\infty ) \big(\| b\|_{L^q(0,t; L^p)}+\| c\|_{L^q(0,t; L^p)} \big) \bigg( \int_0^t (t-s)^{- \frac{q'}{2} \big(1+\frac{d}{p}\big)} s^{-\frac{q'}{2}}\, \d s \bigg)^{\frac{1}{q'}},
\end{align*}
where $q'$ is the conjugate number of $q$. The integral appearing in the last term is finite due to $\gamma<1$; we can compute it by the change of variable $\tilde s = s/t$:
\begin{align*}
\bigg( \int_0^t (t-s)^{- \frac{q'}{2} \big(1+\frac{d}{p}\big)} s^{-\frac{q'}{2}}\, \d s \bigg)^{\frac{1}{q'}} \sim_{p,q} t^{- \frac{1}{2} \big(1+\frac{d}{p}\big) -\frac{1}{2} +\frac1{q'}}
= t^{-\frac\gamma2}.
\end{align*}
Hence, multiplying both sides of the estimate above by $t^{\frac\gamma2}$ and taking supremum on $t\in [0,\tau]$, we arrive at
\begin{align*}
\llbracket v\rrbracket_{\gamma;\tau} \leq C (\llbracket v\rrbracket_{\gamma;\tau} + \| h_0\|_{L^\infty}) \big(\| b\|_{L^q(0,\tau; L^p)}+\| c\|_{L^q(0,\tau; L^p)} \big)
\end{align*}
for some constant $C$.
Choosing $\tau$ sufficiently small so that $C(\| b\|_{L^q(0,\tau; L^p)}+\| c\|_{L^q(0,\tau; L^p)})<1/2$, we can conclude that $\llbracket v\rrbracket_{\gamma,\tau} \leq \| h_0\|_{L^\infty}$;
the estimate can then be extended to the whole interval $[0,T]$ by a standard iteration procedure, which proves \eqref{eq:structure-solution-pde-1}.

The proof of \eqref{eq:structure-solution-pde-2} is similar, based again on applying heat kernel estimates on the terms $I^1,\, I^2$, this time knowing that \eqref{eq:structure-solution-pde-1} holds.
For notational simplicity, we show the estimate only for the term $I^2_t$ and in the case $c\equiv 0$. We have
\begin{align*}
\| I^2_t \|_{C^\delta}
& \lesssim \int_0^t (t-s)^{-\frac{d}{2p}-\frac{\delta}{2}} \| b_s\|_{L^p} \| P_s h_0\|_{C^1}\, \d s\\
& \lesssim \| h_0\|_{L^\infty} \| b\|_{L^q L^p} \bigg( \int_0^t (t-s)^{- \frac{q'}{2} \big(\delta+\frac{d}{p}\big)} s^{-\frac{q'}{2}}\, \d s \bigg)^{\frac{1}{q'}}\\
& \sim_{p,q} \| h_0\|_{L^\infty} \| b\|_{L^q L^p}\, t^{-\frac{1}{2} \big(\delta + \frac{d}{p} \big) -\frac{1}{2}+\frac{1}{q'}}
\sim_b \| h_0\|_{L^\infty} t^{\frac{1-\gamma-\delta}{2}}
\end{align*}
where the integral in the second line is finite for $q'(\delta+d/p)<2$, namely $\delta\in (0,2-\gamma)$.
\end{proof}

\end{document}